\documentclass[twoside,11pt,reqno]{amsart}
\usepackage{amsmath,amsthm,amssymb,amscd,mathrsfs,amscd}
\usepackage{latexsym,amsfonts,exscale,enumerate,comment}
\usepackage{ytableau}
\usepackage{bbm}
\usepackage{graphics,verbatim}
\usepackage{mathrsfs}
\usepackage{todonotes}
\usepackage{hyperref}
\usepackage[all]{xy}
\usepackage{tikz}
\usepackage{tikz-cd}
\usepackage{tikzit}
\usetikzlibrary{shapes.geometric}

\tikzstyle{Thick Blue Edge}=[-, thick, color=\clr]
\tikzstyle{Thick Blue Cup}=[-, thick, color=\clr, looseness=1.5]
\tikzstyle{Thick Blue Crossing}=[-, thick, color=\clr, looseness=1.25]

\tikzstyle{TB Edge}=[-, thick, color=\clr]
\tikzstyle{TB Cup}=[-, thick, color=\clr, looseness=1.5]
\tikzstyle{TB Crossing}=[-, thick, color=\clr, looseness=1.25]

\usepackage[left=1in, top=1in, right=1in, bottom=1in]{geometry}

\usetikzlibrary{decorations.markings}
\usetikzlibrary{decorations.pathreplacing}
\usetikzlibrary{arrows,shapes,positioning}
\tikzstyle directed=[postaction={decorate,decoration={markings, mark=at position #1 with {\arrow[scale=1]{>}}}}]
\tikzstyle rdirected=[postaction={decorate,decoration={markings, mark=at position #1 with {\arrow[scale=1]{<}}}}]
\usepackage[all]{xy}
\newcommand{\clr}{black}

\newcommand{\losemi}{{\otimes \kern -.78em \ltimes}}
\newcommand{\rosemi}{{\otimes \kern -.78em \rtimes}}

\newcommand{\Hom}{\ensuremath{\operatorname{Hom}}}
\newcommand{\End}{\ensuremath{\operatorname{End}}}

\newcommand{\Z}{\mathbb{Z}}
\newcommand{\C}{\mathbb{C}}

\newcommand{\gl}{\ensuremath{\mathfrak{gl}}}

\newcommand{\HH}{\operatorname{H}}

\newcommand{\WW}{\mathcal{W}}

\newcommand{\TT}{\mathtt{T}}

\renewcommand{\Gamma}{\varGamma}

\renewcommand{\k}{\mathbbm{k}}

\newcommand{\Std}{\operatorname{Std}}
\newcommand{\Tab}{\operatorname{Tab}}
\newcommand{\rgWeb}{2\text{ -}\operatorname{Web_{gr}}}
\newcommand{\gWeb}{2\text{ -}\operatorname{Web_{g}}}
\newcommand{\CKMWeb}{2\text{ -}\operatorname{Web_{CKM}}}
\newcommand{\Dst}{\mathcal{D}^{st}}
\newcommand{\nest}{\operatorname{nest}}

\definecolor{green}{RGB}{34,139,34}

\newtheorem{theorem}{Theorem}[subsection]
\makeatletter\let\c@fact\c@theorem\makeatother

\newtheorem*{theorem*}{Theorem}
\makeatletter\let\c@fact\c@theorem\makeatother

\makeatletter\let\c@note\c@theorem\makeatother

\newtheorem{lemma}{Lemma}[subsection]
\makeatletter\let\c@lemma\c@theorem\makeatother

\makeatletter\let\c@alg\c@theorem\makeatother

\newtheorem{remark}{Remark}[subsection]
\makeatletter\let\c@remark\c@theorem\makeatother

\newtheorem{prop}{Proposition}[subsection]
\makeatletter\let\c@prop\c@theorem\makeatother

\makeatletter\let\c@conj\c@theorem\makeatother

\makeatletter\let\c@cor\c@theorem\makeatother

\newtheorem{definition}{Definition}[subsection]
\makeatletter\let\c@defn\c@theorem\makeatother

\theoremstyle{definition}
\newtheorem{example}{Example}[subsection]
\makeatletter\let\c@example\c@theorem\makeatother
\numberwithin{equation}{subsection}

\usepackage[capitalise]{cleveref}
\crefname{theorem}{Theorem}{Theorems}
\crefname{fact}{Fact}{Facts}
\crefname{note}{Note}{Notes}
\crefname{lemma}{Lemma}{Lemmas}
\crefname{alg}{Algorithm}{Algorithms}
\crefname{remark}{Remark}{Remarks}
\crefname{example}{Example}{Examples}
\crefname{prop}{Proposition}{Propositions}
\crefname{conj}{Conjecture}{Conjectures}
\crefname{cor}{Corollary}{Corollaries}
\crefname{definition}{Definition}{Definitions}
\crefname{equation}{\!\!}{\!\!} 

\newcounter{listequation}

\makeatletter
\@namedef{subjclassname@2020}{\textup{2020} Mathematics Subject Classification}
\makeatother

\begin{document}
\title{Positivity and Web Bases for Specht Modules of Hecke Algebras}

\author{Samuel David Heard}
\address{Department of Mathematics \\
          University of Notre Dame \\
          Notre Dame, IN 46556}
\email{sheard3@nd.edu}
\author{Jonathan R. Kujawa}
\address{Department of Mathematics \\
          University of Oklahoma \\
          Norman, OK 73019}
\thanks{Research of the first author was partially supported by the McNair Scholars program and the Oklahoma Louis Stokes Alliances for Minority Participation (OK--LSAMP) program. Research of the second author was partially supported by a Simons Collaboration Grant.}\
\email{kujawa@ou.edu}
\date{\today}

\begin{abstract}  We show that the transition matrix from the standard basis to the web basis for a Specht module of the Hecke algebra is unitriangular and satisfies a strong positivity property whenever the Specht module is labeled by a partition with at most two parts.  This generalizes results of Russell--Tymoczko and Rhoades.
\end{abstract}

\maketitle

\section{Introduction}\label{S:Intro}

\subsection{Background}

The Specht module $S^{(n,n)}$ is a simple module for the symmetric group $S_{2n}$ defined over the rational numbers.  It admits (at least) two interesting combinatorial bases.  The first is the \emph{polytabloid basis} indexed by the standard tableaux of shape $(n,n)$.  Tableaux combinatorics play an important role not only in the representation theory of the symmetric group, but also symmetric functions, Schubert and Springer varieties, and more (e.g., see \cite{Fulton}).  The second is the \emph{web basis} given by non-crossing pairings of $\{1, \dotsc , 2n \}$.  Webs arise in the diagrammatic description of certain categories of representations of the Lie algebra $\gl_{n}(\C )$ and related algebras, are connected to various well-studied diagrammatic algebras (e.g., Temperley--Lieb and Brauer algebras), and appear in various guises in algebraic combinatorics, low-dimensional topology, cluster algebras, invariant theory, and categorical representation theory (e.g., see \cite{FP,FK,Kuperberg,PPR,RoseTubbenhauer} and related references).

In \cite{RT}, Russell and Tymoczko investigated the relationship between these two bases of $S^{(n,n)}$.  They showed that under a suitable ordering the change of basis matrix is unitriangular.  In this context the web basis is known to coincide with the Kazhdan--Lusztig basis and, as such, unitriangularity can be deduced from the work of Garsia--McLarnan \cite{GM} or Naruse \cite{Naruse}.   Russell--Tymoczko gave a more direct, combinatorial proof of this result.  They also sharpened it by giving additional conditions on when an entry can be nonzero.  Moreover, based on their results and the available data, Russell--Tymoczko conjectured that the entries of this matrix are always nonnegative.  Rhoades gave a proof of the Russell--Tymoczko positivity conjecture in \cite{Rhoades}.  This positivity property is known to be false for Specht modules in general and it remains an open question to determine when it holds.

\subsection{Main results} The purpose of the present paper is to generalize this work in two directions.  First, we replace the group algebra $\mathbb{Q} S_{d}$ with the Hecke algebra $\HH_{d}(q)$ defined over $\mathbb{Q}(q)$.  Second, we consider analogues of the polytabloid and web bases of the Specht module $S^{\lambda}$ for $\HH_{d}(q)$ where $\lambda = (a,b)$ is an arbitrary two-row partition of $d$.

In this setting the analogue of the polytabloid basis of $S^{\lambda}$ is naturally indexed by the set of standard tableaux of whose shape is the transpose of $\lambda$, $\Std (\lambda^{\TT})$.  The web basis is given by the set of non-crossing pairings of $\{1, \dotsc , a+b, 1', \dotsc , (a-b)' \}$, $\WW^{(a,b)}$.  We define a combinatorial bijection 
\[
\psi: \Std (\lambda) \to \WW^{(a,b)}
\] and partial order $\preceq$ on $\WW^{(a,b)}$ which generalize those given in \cite{RT}.  Our main results are \cref{T:uppertriangular,T:uppertriangular2,T:positivity}.  These unitriangularity and strong positivity results are summarized in the following theorem.

\begin{theorem*} Let $\lambda = (a,b)$ be a partition of $d$ with two parts.  For every $t \in \Std (\lambda^{\TT})$, 
\[
v_{t} =\psi (t^{\TT}) + \sum_{\substack{w \in \WW^{(a,b)}\\  w \prec  \psi (t^{\TT})}} c_{w,t} w,
\] where $c_{w,t} \in \Z_{\geq 0}[q]$. 
\end{theorem*}  We refer the reader to \cref{SS:ExampleII} and \cref{SS:ExampleIII} for examples when $\lambda = (3,3)$ and $\lambda = (4,2)$, respectively. 
Setting $q=1$ and $a=b$ recovers the results of \cite{Rhoades,RT}.

\subsection{Additional questions} Work of  Im--Zhu \cite{IZ} and Hwang--Jang--Oh \cite{HJO} gives combinatorial interpretations for the entries of this change of basis matrix when $q=1$ and $a=b$.  In particular, they determine precisely when an entry can be nonzero and thereby verify the conjecture of Russell--Tymoczko  that the conditions they gave for non-vanishing were both necessary and sufficient.  It would be interesting to refine and generalize these results in order to describe the matrix entries as polynomials in $q$ for all two-row partitions.  In particular, in the examples we have computed each entry of the matrix is palindromic in that it can be written in the form $q^{t}f$ where $f \in \Z_{\geq 0}[q,q^{-1}]$ and $f(q)=f(q^{-1})$. A combinatorial model for these polynomials could help answer if this behavior always occurs.  Related is the work of McDonough--Pallikaros which establishes unitriangularity between the standard basis and the Kazhdan--Lusztig basis for Specht modules of $\HH_{d}(q)$ and that the change of basis entries can be described using Kazhdan--Lusztig polynomials \cite{MP}.

In a different direction, there are versions of the polytabloid and web/skein bases when the partition has three equal rows,  $\lambda = (n,n,n)$, or is of a flag shape, $\lambda = (n,n, 1^{k})$.  See, for example, \cite{PPS,Rhoades2,RT2}.  The Tubbenhauer--Vaz--Wedrich diagrams we use here are defined for $U_{q}(\mathfrak{gl}_{n})$ for all $n \geq 1$ and, hence, can be used for partitions with any number of parts. Our methods could be used to investigate these and other shapes.

In a third direction, Khovanov--Mazorchuk--Stroppel gave a categorification of Specht modules via subcategories of graded parabolic category $\mathcal{O}$ for $\mathfrak{sl}_{n}(\C )$ where the action of the generators of $\HH_{d}(q)$ are given by translation functors \cite{KMS}.  In this categorification the isomorphism classes of the indecomposable projective-injective modules are in bijection with the Kazhdan--Lusztig basis of the Specht module.   We expect the results of this paper can be interpreted as graded multiplicities in this categorification.  

\subsection*{Acknowledgements}  The first author thanks Sophia Bolin-Dills, Rodney Bates, Regennia Johnson, and Bushra Asif for their support during the time of this project.  Both authors would like to thank Jieru Zhu for helpful conversations.  We also thank the referees for their helpful comments.

\section{Hecke Algebras and Specht Modules}\label{S:Prelim}

\subsection{Preliminaries}\label{SS:prelims}  Throughout the base field is $\k = \mathbb{Q}(q)$, where $q$ is an indeterminant.   Given a nonnegative integer $k$ and a $t \in \k$ which is not a root of unity, let 
\[
[k]_{t} = \frac{t^{k}-t^{-k}}{t-t^{-1}} = t^{k-1}+t^{k-3}+\dots + t^{-(k-3)}+t^{-(k-1)} \in \k.
\]
Set 
\[
[k]_{t}! = [k]_{t}[k-1]_{t}\dotsb [2]_{t}[1]_{t}.
\]

A partition $\lambda = (\lambda_{1}, \lambda_{2}, \dotsc )$ of $d$ is a weakly decreasing sequence of nonnegative integers which sum to $d$.  Write $\lambda \vdash d$ to indicate that $\lambda$ is a partition of $d$.  The nonzero entries of $\lambda$ are called the parts of $\lambda$.  let $\Lambda_{+}(2)$ be the set of all partitions with two or fewer parts. For $d \geq 0$, let $\Lambda_{+}(2,d)$ be the set of all partitions of $d$ with two or fewer parts.

Given a partition $\lambda = (\lambda_{1}, \dotsc , \lambda_{n}) \vdash d$, let $S_{\lambda} \cong S_{\lambda_{1}} \times \dotsb \times S_{\lambda_{n}}$ denote the corresponding Young subgroup of the symmetric group on $d$ letters, $S_{d}$.   Let $s_{1}, \dotsc , s_{d-1} \in S_{d}$ denote the simple transpositions and let $\ell: S_{d} \to \Z_{\geq 0}$ be the corresponding length function.

Given a partition, $\lambda$, there is the corresponding Young diagram consisting of $\lambda_{i}$ boxes in the $i$th row, left justified.  We use the so-called English notation when drawing Young diagrams.  A tableau of shape $\lambda$ is a filling of the boxes of the Young diagram for $\lambda$ with the numbers $1, \dotsc , d$.  A standard tableau of shape $\lambda$ is a tableau of shape $\lambda$ which satisfies the property that the entries strictly increase as one goes left-to-right along rows and top-to-bottom along columns.  Given a partition, $\lambda \vdash d$, let $t_{0}=t^{\lambda}$ be the distinguished standard tableau given by filling the boxes of $\lambda$ with $1, \dotsc , d$ in order by filling in each row from left to right, starting with the top row.  Let $\Tab (\lambda)$ be the set of all tableaux of shape $\lambda$ and let $\Std (\lambda)$ be the set of all standard tableaux of shape $\lambda$.

Given a partition $\lambda$ let $\lambda^{\TT}$ denote the transpose partion.  Likewise, given a tableau $t$ let $t^{\TT}$ denote the transpose tableau. There is a right action of $S_{d}$ on the set of tableaux of shape $\lambda$ by acting on the entries of a tableau.  We write $t.\sigma$ for this action.  The transpose map is $S_{d}$-equivariant.  For each partition $\lambda \vdash d$, fix $\sigma_{\lambda} \in S_{d}$ once and for all by the equality 
\begin{equation}\label{E:sigmadef}
(t^{\lambda}.\sigma_{\lambda})^{\TT } = t^{\lambda^{\TT}}.
\end{equation}

\subsection{The Hecke algebra of type A}\label{SS:HeckeAlgebra}

\begin{definition}\label{R:Heckeisom}  Given a nonnegative integer $d$, let $\HH_{d}(q)$ denote the associative, unital $\k$-algebra generated by $T_{1}, \dots , T_{d-1}$ subject to the relations (for all admissible $1 \leq i,j \leq d-1$):
\begin{itemize}
\item $T_{i}T_{j}=T_{j}T_{i}$ for $|i-j| > 1$;
\item $T_{i+1}T_{i}T_{i+1} = T_{i}T_{i+1}T_{i}$;
\item $ (T_{i}  - q)(T_{i}+q^{-1})=0 $. 
\end{itemize}
\end{definition}

For a reduced expression $w = s_{i_{1}}\dotsb s_{i_{r}}$, set $T_{w}= T_{i_{1}}\dotsb T_{i_{r}}$.   This is independent of the choice of reduced expression and the set $\left\{T_{w} \mid w \in S_{d} \right\}$ is a $\k$-basis for  $\HH_{d}(q)$.

\subsection{Specht modules}\label{SS:SpechtModules}  The representation theory of $\HH_{d}(q)$ is  well-known.  We give a brief summary in order to establish conventions.   We generally follow \cite{DipperJames,Mathas}. Given a partition $\lambda \vdash d$, define the following elements of $\HH_{d}(q)$:

\[
x_{\lambda} = \sum_{w \in S_{\lambda}} q^{\ell (w)}T_{w}, \hspace{.5in} y_{\lambda}  = \sum_{w \in S_{\lambda}} (-q)^{-\ell(w)}T_{w}, 
\] and 
\[
z_{\lambda}= x_{\lambda}T_{\sigma_{\lambda}}y_{\lambda^{\TT}},
\] where $\sigma_{\lambda}$ is as in \cref{E:sigmadef}.

The \emph{Specht module} associated to $\lambda \vdash d$ is the right $\HH_{d}(q)$-module generated by $z_{\lambda}$,
\[
S^{\lambda} = z_{\lambda}\HH_{d}(q).
\] For example, if $\lambda = (d)$, then $x_{\lambda}= \sum_{w \in S_{d}} q^{\ell (w)}T_{w}$, $y_{\lambda^{\TT}} = 1$, $\sigma_{\lambda}= e_{S_{d}}$, and $S^{\lambda}$ is the one-dimensional $\HH_{d}(q)$-module spanned by $z_{\lambda}=x_{\lambda}$  which satisfies $z_{\lambda}T_{i}=qz_{\lambda}$ for all $i=1, \dotsc , d-1$. If $\lambda = (1, 1, \dotsc , 1)$, then  $x_{\lambda}= 1$, $y_{\lambda^{\TT}} = \sum_{w \in S_{d}} (-q)^{-\ell(w)} T_{w}$, $\sigma_{\lambda}= e_{S_{d}}$, and $S^{\lambda}$ is the one-dimensional $\HH_{d}(q)$-module spanned by $z_{\lambda}=y_{\lambda^{\TT}}$  which satisfies $z_{\lambda}T_{i}=-q^{-1}z_{\lambda}$ for all $i=1, \dotsc , d-1$.

It is known that 
\[
\left\{ S^{\lambda}\mid \lambda \vdash d \right\}
\] is a complete, irredundant set of simple $\HH_{d}(q)$-modules.  The following lemma is well-known; e.g., see \cite[Corollary 4.2]{DipperJames} or \cite[Exercise 3.18(iv)]{Mathas}.

\begin{lemma}\label{L:KeyLemma}  Given a partition $\lambda \vdash d$, $z_{\lambda}y_{\lambda^{\TT}}$ is a nonzero scalar multiple of $z_{\lambda}$.  Moreover, if  $v \in S^{\lambda}$ has the property that $vy_{\lambda^{\TT}} = cv$ for a nonzero scalar $c \in \k$, then $v \in \k z_{\lambda}$.
\end{lemma}

\subsection{The weak Bruhat order}\label{SS:WeakBruhat}  Recall that the \emph{(right) weak Bruhat order}, $\leq$, on $S_{d}$ is defined by $u \leq v$ if and only if there is a reduced expression $v=s_{i_{1}}\dotsb s_{i_{t}}$ and a $0 \leq k \leq t$ such that $u=s_{i_{1}}\dotsb s_{i_{k}}$.  Write $u \leq_{j} v $ if there is a reduced expression $v=s_{i_{1}}\dotsb s_{i_{t-1}} s_{i_{t}}$, $u=s_{i_{1}}\dotsb s_{i_{t-1}}$,  and $s_{i_{t}}=s_{j}$.  Then $u \leq v$ if and only if $u=v$ or there are sequences $u_{1}, \dotsc , u_{r}$ and $1 \leq j_{1}, \dotsc , j_{r} \leq d-1$ such that 
\[
u \leq_{j_{1}} u_{1} \leq_{j_{2}} \dotsb \leq_{j_{r}} u_{r} = v.
\]

Given a partition $\lambda \vdash d$, define 
\[
\Dst_{\lambda} =\left\{\tau \in S_{d} \mid \tau \leq \sigma_{\lambda} \right\},
\] where $\sigma_{\lambda}$ was defined in \eqref{E:sigmadef}.  By \cite[Lemma 1.5]{DipperJames}, the function 
\begin{gather*}
\eta: \Dst_{\lambda}  \to \Std (\lambda)\\
\tau \mapsto t^{\lambda}.\tau
\end{gather*}
is a bijection.  Define a partial order on $\Std (\lambda)$ by declaring $\eta(\tau_{1}) \leq \eta(\tau_{2})$ if and only if $\tau_{1} \leq \tau_{2}$.  We refer to this as the weak Bruhat order on $\Std (\lambda)$. 

For example, with respect to the weak Bruhat order, the identity element $e_{S_{d}} \in S_{d}$ is the unique minimal element of $\Dst_{\lambda}$.  The corresponding minimal element of $\Std (\lambda)$ is  $t_{0}=t^{\lambda}$.  At the other extreme, $\sigma_{\lambda}$ is the unique maximal element of $\Dst_{\lambda}$ and $t^{\lambda}.\sigma_{\lambda} = (t^{\lambda^{\TT}})^{\TT} $ is the corresponding maximal element of $\Std (\lambda)$.  By definition, the Hasse diagram for this partial order on $\Std (\lambda)$ is the labeled directed graph with an edge labeled by $i$ from $t_{1}$ to $t_{2}$ whenever $t_{1} \leq_{i} t_{2}$.

\begin{remark}\label{R:weakBruhatOrder}  It is well-known that for $\tau_{1},\tau_{2} \in S_{d}$, $\tau_{1} \leq_{i} \tau_{2}$ if and only if $\tau_{2}=\tau_{1}s_{i}$ and $\tau_{2}$ has one more inversion than $\tau_{1}$ (e.g., see \cite[Section 3.1]{BB}). Via the bijection $\eta$ this translates into the statement that for $t_{1}, t_{2} \in \Std (\lambda)$, $t_{1} \leq_{i}  t_{2}$ if and only if $t_{2}=t_{1}.s_{i}$ and $i$ occurs before $i+1$ when reading the entries of $t_{1}$ left-to-right, top-to-bottom.
\end{remark}

We record some basic properties of the weak Bruhat order on $\Std ((n+r,n)^{\TT})$ which will be useful later.

\begin{lemma}\label{L:WeakBruhatProperties} Let $\lambda = (n+r,n)$.  Let $t, u, u_{1}, t_{p} \in \Std (\lambda^{\TT})$ for $p=1,2,3$.
\begin{enumerate}
\item If $t \leq_{i} t_{1}$ and $t \leq_{j} t_{2}$ with $i \neq j$, then $|i-j| > 1$.
\item If $t \leq_{i} t_{1}$ and $t \leq_{j} t_{2}$ with $i \neq j$, then there is a $t_{3}$ such that $t_{1} \leq_{j} t_{3}$ and $t_{1} \leq_{i} t_{3}$. 
\item If $t_{1} \leq_{i} t_{2} \leq_{j} t_{3}$, then $i \neq j$.
\item If $t \leq u$, $t \leq_{i} t_{1}$, and $u \leq_{i} u_{1}$, then $t_{1} \leq u_{1}$
\end{enumerate}

\end{lemma}

\begin{proof} For the first claim, the proof of the same statement for $\lambda=(n,n)$ given in \cite[Lemma 3.3]{RT} applies verbatim once one accounts for the transpose map.  The second claim is immediate from the first claim and the definition of the weak Bruhat order.  The third claim follows directly from the definition of the weak Bruhat order.  Given the previous claims the fourth claim is proven exactly as in \cite[Lemma 5.6]{RT}
\end{proof}

\subsection{The standard basis for Specht modules}\label{SS:polytabloidbasis}

By \cite[Theorem 5.6]{DipperJames} the set 
\[
\left\{z_{\lambda}T_{\tau}\mid \tau \in \Dst_{\lambda^{\TT}} \right\}
\] is a basis for $S^{\lambda}$.  Relabeling this basis using the bijection $\eta$ gives a basis for $S^{\lambda}$ indexed by the standard tableaux of shape $\lambda^{\TT}$:
\begin{equation}\label{E:standardbasis}
\left\{v_{t} := z_{\lambda}T_{\eta^{-1}(t)} \mid t\in \Std (\lambda^{\TT}) \right\}.
\end{equation}  This is the so-called \emph{standard basis} of $S^{\lambda}$.  It could also be called the \emph{polytabloid basis} for $S^{\lambda}$ since at $q=1$ this basis specializes to the polytabloid basis of the corresponding Specht module for the symmetric group $S_{d}$.

The weak Bruhat order can be transported to the standard basis by declaring $v_{t_{1}} \leq v_{t_{2}}$ if and only if $t_{1} \leq t_{2}$ in $\Std (\lambda^{\TT})$. For example,  
\begin{equation}\label{E:v0def}
v_{0}:=v_{t_{0}}=v_{t^{\lambda^{\TT}}}=z_{\lambda}
\end{equation}
is the unique minimal basis element of $S^{\lambda}$ and $v_{(t^{\lambda})^{\TT}}=v_{t^{\lambda^{\TT}}.\sigma_{\lambda^{\TT}}} = z_{\lambda}T_{\sigma_{\lambda^{\TT}}}$ is the unique maximal basis element of $S^{\lambda}$. Note, the appearance of the tranpose of $\lambda$ and the right weak Bruhat order are both due to the fact we choose to work with right $\HH_{d}(q)$-modules.

\begin{example}  Let $\lambda = (3,3)$.  Then $s_{\lambda^{\TT }}=s_{2}s_{4}s_{3}=s_{4}s_{2}s_{3}$. The Hasse diagrams for $\Dst_{\lambda^{\TT}}$, $\Std (\lambda^{\TT})$, and the standard basis for $S^{(3,3)}$, respectively, are:

\begin{equation*}
\begin{tikzpicture}[baseline=1ex,scale=.5, color=\clr, yscale=-1]
	\begin{pgfonlayer}{nodelayer}
		\node [style=none] (0) at (0, -4) {};
		\node [style=none] (1) at (2, 1) {};
		\node [style=none] (2) at (-2, 1) {};
		\node [style=none] (3) at (0, 2) {};
		\node [style=none] (4) at (0, -5) {};
		\node [style=none] (5) at (-2, -1) {};
		\node [style=none] (6) at (2, -1) {};
		\node [style=none] (7) at (0, -2) {};
		\node [style=none] (8) at (0, 3) {$e_{S_{d}}$};
		\node [style=none] (9) at (-2, 0) {$s_{4}$};
		\node [style=none] (10) at (2, 0) {$s_{2}$};
		\node [style=none] (11) at (0, -3) {$s_{2}s_{4}$};
		\node [style=none] (12) at (0, -6) {$s_{\lambda^{\TT}}$};
		\node [style=none] (14) at (-1.25, 1.75) {\tiny{$4$}};
		\node [style=none] (15) at (1.25, 1.75) {\tiny{$2$}};
		\node [style=none] (16) at (-1.25, -1.75) {\tiny{$2$}};
		\node [style=none] (17) at (1.25, -1.75) {\tiny{$4$}};
		\node [style=none] (18) at (0.5, -4.5) {\tiny{$3$}};
	\end{pgfonlayer}
	\begin{pgfonlayer}{edgelayer}
		\draw [thick] (4.center) to (0.center);
		\draw [thick] (2.center) to (3.center);
		\draw [thick] (1.center) to (3.center);
		\draw [thick] (5.center) to (7.center);
		\draw [thick] (7.center) to (6.center);
	\end{pgfonlayer}
\end{tikzpicture} \; ,
\hspace{.5in}
\begin{tikzpicture}[baseline=1ex,scale=.5, color=\clr, yscale=-1]
	\begin{pgfonlayer}{nodelayer}
		\node [style=none] (0) at (0, -4) {};
		\node [style=none] (1) at (2, 1) {};
		\node [style=none] (2) at (-2, 1) {};
		\node [style=none] (3) at (0, 2) {};
		\node [style=none] (4) at (0, -5) {};
		\node [style=none] (5) at (-2, -1) {};
		\node [style=none] (6) at (2, -1) {};
		\node [style=none] (7) at (0, -2) {};
		\node [style=none] (8) at (-0.25, 3) {\ytableausetup{centertableaux,smalltableaux}
\begin{ytableau}
1 & 2\\
3 & 4 \\
5& 6
\end{ytableau}};
		\node [style=none] (9) at (-2.25, 0) {\ytableausetup{centertableaux,smalltableaux}
\begin{ytableau}
1 & 2\\
3 & 5 \\
4& 6
\end{ytableau}};
		\node [style=none] (10) at (1.75, 0) {\ytableausetup{centertableaux,smalltableaux}
\begin{ytableau}
1 & 3\\
2 & 4 \\
5 & 6
\end{ytableau}};
		\node [style=none] (11) at (-0.2, -3) {\ytableausetup{centertableaux,smalltableaux}
\begin{ytableau}
1 & 3\\
2 & 5 \\
4 & 6
\end{ytableau}};
		\node [style=none] (12) at (-0.2, -6) {\ytableausetup{centertableaux,smalltableaux}
\begin{ytableau}
1 & 4\\
2 & 5 \\
3& 6
\end{ytableau}};
		\node [style=none] (14) at (-1.25, 1.75) {\tiny{$4$}};
		\node [style=none] (15) at (1.25, 1.75) {\tiny{$2$}};
		\node [style=none] (16) at (-1.25, -1.75) {\tiny{$2$}};
		\node [style=none] (17) at (1.25, -1.75) {\tiny{$4$}};
		\node [style=none] (18) at (0.5, -4.5) {\tiny{$3$}};
	\end{pgfonlayer}
	\begin{pgfonlayer}{edgelayer}
		\draw [thick] (4.center) to (0.center);
		\draw [thick] (2.center) to (3.center);
		\draw [thick] (1.center) to (3.center);
		\draw [thick] (5.center) to (7.center);
		\draw [thick] (7.center) to (6.center);
	\end{pgfonlayer}
\end{tikzpicture} \; ,
\hspace{0.5in}
\text{and}
\hspace{0.5in}
\begin{tikzpicture}[baseline=1ex,scale=.5, color=\clr, yscale=-1 ]
	\begin{pgfonlayer}{nodelayer}
		\node [style=none] (0) at (0, -4) {};
		\node [style=none] (1) at (2, 1) {};
		\node [style=none] (2) at (-2, 1) {};
		\node [style=none] (3) at (0, 2) {};
		\node [style=none] (4) at (0, -5) {};
		\node [style=none] (5) at (-2, -1) {};
		\node [style=none] (6) at (2, -1) {};
		\node [style=none] (7) at (0, -2) {};
		\node [style=none] (8) at (0, 3) {$z_{\lambda}$};
		\node [style=none] (9) at (-2, 0) {$z_{\lambda}T_{s_{4}}$};
		\node [style=none] (10) at (2, 0) {$z_{\lambda}T_{s_{2}}$};
		\node [style=none] (11) at (0, -3) {$z_{\lambda}T_{s_{2}s_{4}}$};
		\node [style=none] (12) at (0, -6) {$z_{\lambda}T_{s_{\lambda^{\TT }}}$};
		\node [style=none] (14) at (-1.25, 1.75) {\tiny{$4$}};
		\node [style=none] (15) at (1.25, 1.75) {\tiny{$2$}};
		\node [style=none] (16) at (-1.25, -1.75) {\tiny{$2$}};
		\node [style=none] (17) at (1.25, -1.75) {\tiny{$4$}};
		\node [style=none] (18) at (0.5, -4.5) {\tiny{$3$}};
	\end{pgfonlayer}
	\begin{pgfonlayer}{edgelayer}
		\draw [thick] (4.center) to (0.center);
		\draw [thick] (2.center) to (3.center);
		\draw [thick] (1.center) to (3.center);
		\draw [thick] (5.center) to (7.center);
		\draw [thick] (7.center) to (6.center);
	\end{pgfonlayer}
\end{tikzpicture} \; .
\end{equation*}

\end{example}

\section{Webs}\label{S:Webs}

\subsection{Webs}\label{SS:Webs} We next introduce a diagrammatic description of Specht modules labeled by two-row partitions using web diagrams. An $(n+r, n)$-web diagram is a certain non-crossing pairing of the elements of the set 
\[
\left\{1, 2, \dots , 2n+r, 1', 2', \dots , r' \right\}.
\]
This description is essentially known to experts (e.g., see \cite{FK,JohnsonStewart}).  However, we could not find one which exactly matches our preferred conventions and so opt to give a detailed treatment.

When discussing webs we will usually view them as graphs with vertices labeled by the set $\left\{1, 2, \dots , 2n+r, 1', 2', \dots , r' \right\}$ and where the pairing of the vertex labeled by $a$ and the vertex labeled by $b$ will be represented by a line connecting these vertices. We will call that line the \emph{edge} incident to vertex $a$ and vertex $b$.

\begin{definition}\label{D:Webs}    Let $n,r$ be nonnegative integers. An \emph{$(n+r,n)$-web} is a graph contained in the strip $\mathbb{R} \times [0,1]$ which satisfies: 
\begin{itemize}
\item vertices on the line $\mathbb{R} \times \{0 \}$ labeled by $1, \dots , 2n+r$ in that order when read left-to-right, vertices on the line $\mathbb{R} \times \{1 \}$ labeled by $1', \dots , r'$ in that order when read left-to-right;
\item each vertex is connected by an edge to one and only one other vertex;
\item vertices labeled by $1', \dots , r' $ are only connected to vertices labeled by $1, \dots , 2n+r $;
\item the graph is planar.
\end{itemize}

The non-crossing pairing associated to such a graph is the obvious one.  Graphs which are related by a planar isotopy give the same pairings and are considered equal as webs.
\end{definition}

For brevity's sake, given a web and a $k \in \left\{1, \dotsc , 2n+r, 1', \dotsc , r' \right\}$, we may say vertex $k$ to mean the vertex labeled by $k$ in the given web.  We will frequently leave the vertices labels implicit when it cannot cause confusion.  On the right-hand-side of \cref{E:etaexample} one can find an example of a $(10,6)$-web and in \cref{SS:ExampleI} are various examples of $(3,3)$-webs.  The following is an $(n+r,n)$-web which will play an important role in what follows:
\begin{equation}\label{E:w0definition}
w_{0} \; \; = \; \;
\begin{tikzpicture}[baseline=1ex,scale=.5, color=\clr ]
	\begin{pgfonlayer}{nodelayer}
		\node [style=none] (0) at (0, 0) {\tiny{$\cdots$}};
		\node [style=none] (1) at (1, 0) {};
		\node [style=none] (2) at (2.5, 0) {};
		\node [style=none] (3) at (3, 0) {};
		\node [style=none] (4) at (4.5, 0) {};
		\node [style=none] (5) at (5, 0) {};
		\node [style=none] (6) at (5.75, 0) {};
		\node [style=none] (7) at (6.5, 0) {\tiny{$\cdots$}};
		\node [style=none] (8) at (8, 0) {};
		\node [style=none] (9) at (7.25, 0) {};
		\node [style=none] (11) at (5, 3) {};
		\node [style=none] (12) at (5.75, 3) {};
		\node [style=none] (13) at (6.5, 3) {\tiny{$\cdots$}};
		\node [style=none] (14) at (8, 3) {};
		\node [style=none] (15) at (7.25, 3) {};
		\node [style=none] (16) at (-4.5, 0) {};
		\node [style=none] (17) at (-3, 0) {};
		\node [style=none] (18) at (-2.5, 0) {};
		\node [style=none] (19) at (-1, 0) {};
%
	\end{pgfonlayer}
	\begin{pgfonlayer}{edgelayer}
		\draw [thick, looseness=1.5] [out=90,in=90] (16.center) to (17.center);
		\draw [thick, looseness=1.5] [out=90,in=90] (18.center) to (19.center);
		\draw [thick, looseness=1.5] [out=90,in=90] (1.center) to (2.center);
		\draw [thick, looseness=1.5] [out=90,in=90] (3.center) to (4.center);
		\draw [thick, ] (5.center) to (11.center);
		\draw [thick, ] (6.center) to (12.center);
		\draw [thick, ] (9.center) to (15.center);
		\draw [thick, ] (8.center) to (14.center);
	\end{pgfonlayer}
\end{tikzpicture} \; \; . 
\end{equation}

If an edge in an $(n+r,n)$-web connects two vertices from among $\{1, \dots , 2n+r \}$ we call it an \emph{arc}.  If the edge connects a vertex from among $\{1, \dots , 2n+r \}$ to a vertex from among $\{1', \dots , r' \}$, then we call it a \emph{through-string}.
For an arc, the \emph{L(eft)-vertex} (resp.,  \emph{R(ight)-vertex}) of the arc is the leftmost (resp., rightmost) vertex adjacent to the arc. For a through-string, the \emph{L-vertex} is the vertex with label from the set $\{1, \dotsc , 2n+r \}$ and the \emph{R-vertex} is the one with label from the set $\{1', \dotsc , r' \}$.  Given a pair $a,b \in \left\{1, 2, \dots , 2n+r, 1', 2', \dots , r' \right\}$ which are connected by an edge in some web we write $\langle a,b \rangle$ for that edge and usually assume that $a$ is the L-vertex and $b$ is the R-vertex.

\subsection{Webs and standard tableaux}\label{SS:WebsAndTableaux}

Given $n,r \geq 0$, let $\WW^{(n+r,r)}$ be the set of $(n+r,n)$-webs. Let $\lambda = (n+r,n)$, viewed as a partition of $2n+r$. Define 
\[
\phi: \WW^{(n+r,n)} \to \Std(\lambda)
\] by setting $\phi(w)$ to be the standard tableau of shape $\lambda$ obtained by filling the boxes in the first row (resp., second row) of $\lambda$ with the labels of the L-vertices (resp., R-vertices of arcs) of $w$ read left-to-right.  By definition $\phi (w)$ is a row-standard tableau of shape $\lambda$.  That $\phi (w)$ is standard and that $\phi$ is a bijection will be verified in \cref{L:WebTableauxBijection}.

On the other hand, define
\[
\psi: \Std (\lambda) \to \WW^{(n+r,n)}
\] by the following algorithm.  By construction $\psi (t)$ will be a pairing of the elements of the set $\left\{1, \dotsc , 2n+r, 1', \dotsc , r' \right\}$. We will show in \cref{L:WebTableauxBijection} that $\psi (t)$ is an $(n+r,n)$-web and that $\psi$ is a bijection. 

Given a standard tableau of shape $\lambda$, $t$, let $\psi (t)$ be the graph constructed recursively as follows:
\begin{itemize}
\item [Step 0:] Mark vertices on the line $\mathbb{R} \times \{0 \}$ labeled by $1, \dots , 2n+r$ in that order when read left-to-right and vertices on the line $\mathbb{R} \times \{1 \}$ labeled by $1', \dots , r'$ in that order when read left-to-right;
\item[Step 1:]  Consider the entries in the second row of $t$ and consider those which label a vertex of the graph under construction which are not yet adjacent to an edge.  If there is no such entry, go to Step~3.  Otherwise, let $b$ be the entry of the leftmost such box of $t$.
\item [Step 2:] Consider the set of vertices in the graph under construction which are to the left of the vertex $b$ and which have labels which appear in the first row of $t$.  Connect vertex $b$ to the vertex from this collection which is closest to vertex $b$ in the graph and not yet adjacent to an edge. Return to Step~1. 
\item [Step 3:] At this step there are $r$ vertices labeled by $1, \dots , 2n+r $ which are not yet adjacent to an edge.  Connect these vertices to the vertices labeled by $1', \dots , r' $ using non-crossing edges.
\end{itemize}

\begin{example}  Say $n=2$, $r=1$ and $t$ is the following tableau of shape $(3,2)$:
\[
t= \ytableausetup{centertableaux}
\begin{ytableau}
1 & 2 & 3 \\
4 & 5  
\end{ytableau}
\]  Applying the algorithm yields the following $(3,2)$-web:
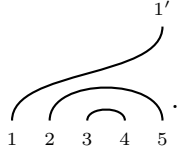
\[
\begin{tikzpicture}[baseline=1ex,scale=.5, color=\clr ]
	\begin{pgfonlayer}{nodelayer}
		\node [style=none] (0) at (-2, 0) {};
		\node [style=none] (1) at (-1, 0) {};
		\node [style=none] (2) at (0, 0) {};
		\node [style=none] (3) at (1, 0) {};
		\node [style=none] (4) at (2, 0) {};
		\node [style=none] (5) at (-2, -0.5) {\tiny{$1$}};
		\node [style=none] (6) at (-1, -0.5) {\tiny{$2$}};
		\node [style=none] (7) at (0, -0.5) {\tiny{$3$}};
		\node [style=none] (8) at (1, -0.5) {\tiny{$4$}};
		\node [style=none] (9) at (2, -0.5) {\tiny{$5$}};
		\node [style=none] (10) at (2, 2.5) {};
		\node [style=none] (11) at (2, 3) {\tiny{$1'$}};
	\end{pgfonlayer}
	\begin{pgfonlayer}{edgelayer}
		\draw [thick, looseness=1] [out=90,in=90] (1.center) to (4.center);
		\draw [thick, looseness=1] [out=90,in=90] (2.center) to (3.center);
		\draw [thick, looseness=0.75] [out=90,in=270] (0.center) to (10.center);
	\end{pgfonlayer}
\end{tikzpicture}.
\] Let us explain how the algorithm yields this web. In Step~0 the labeled vertices are placed as indicated. The first time Step~1 is applied, vertex $4$ is the leftmost entry of the second row of $\lambda$ which is not yet adjacent to an edge and so is selected.  In Step~2  we connect vertex $4$ to vertex $3$ since from among the vertices labeled by entries from the first row of $\lambda$ which are to the left of $4$ and not yet adjacent to an edge (namely, vertices $1$, $2$, and $3$), this is the nearest.  Having completed Step~2, we return to Step~1. In this iteration, in Step~1 we select vertex $5$ since it is the leftmost entry of the second row of $\lambda$ which is not yet adjacent to an edge. In Step~2  we connect vertex $5$ to vertex $2$ since from among the vertices labeled by entries from the first row of $\lambda$ which are to the left of $5$ are not yet adjacent to an edge (namely, vertices $1$ and $2$), this is the nearest.  Having completed Step~2, we return to Step~1. In this iteration, since all entries of the second row of $\lambda$ are now adjacent to an edge we go to Step~3.  There is one remaining open $L$-vertex, namely vertex $1$, and we connect it to vertex $1'$ by a through-string.  The algorithm stops and the result is the indicated web.
\end{example}

Two remarks are in order.  First, by construction the entries in the first (resp., second) row of $t$ will be the labels for the L-vertices (resp., R-vertices of arcs) of $\psi(t)$.   Second, say we are applying Step~2 of the algorithm for the $k$th time and $b_{k}$ is the selected entry of the second row of $t$.  There are $k$ vertex labels in the first row of $t$ which are weakly to the left of the box containing $b_{k}$.  Since $t$ is standard, all of these entries are smaller than $b_{k}$.  That is, in the graph under construction there are at least $k$ eventual L-vertices to the left of vertex $b_{k}$ in the graph under construction.  That is, when we apply Step~2 to the box containing $b_{k}$, we will be in the $k$th iteration of Step~2, have so far drawn $k-1$ arcs, and there remains at least one open L-vertex to the left of vertex $b_{k}$ which is not yet adjacent to an edge.  That is, for each box in the second row of $t$, the set of L-vertices considered when Step~2 is applied to the entry of that box will be nonempty and, hence, there will be a nearest one. In short, the algorithm never fails in Step~2 and, hence, never fails.
\begin{lemma}\label{L:WebTableauxBijection}  The images of the maps $\phi$ and $\psi$ lie in the given sets and the two maps are mutual inverses.  In particular, 
\[
|\WW^{(n+r+r)}| = |\Std (n+r,n)|.
\]
\end{lemma}

\begin{proof} We first explain why the tableau $\phi(w)$ is standard for any $(n+r,n)$-web $w$.  First, by definition the entries of $\phi(w)$ are increasing along rows.  Now consider the columns of $\phi(w)$. Say $a_{k}$ and $b_{k}$ are the entries in the $k$th column of $\phi(w)$ and are in the first and second rows, respectively.  From the definition of $\phi$ it follows that $b_{k}$ is the $k$th R-vertex of an arc as one reads left-to-right in $w$.  That is, there are at least $k$  L-vertices to the left of the vertex $b_{k}$. Thus, the first $k$ entries of the first row of $\phi(w)$ are filled with labels which are smaller than $b_{k}$. In particular, $a_{k}$ is smaller than $b_{k}$.  Since this is true for every column, $\phi(w)$ is standard.

We next explain why $\psi(t)$ is an $(n+r,n)$-web for any standard tableau $t$ of shape $\lambda = (n+r,n)$.  By construction, $\psi(t)$ provides a pairing of the vertices $\{1, \dots , 2n+r, 1', \dots , r' \}$ and the vertices $1', \dots , r'$ are only connected to vertices from among $\{1, \dots , 2n+r \}$.

It remains to confirm that the graph is planar.  By construction any two through-strings are non-crossing, so the only worry is the crossing of two arcs, or of an arc and a through-string.  First, consider two crossing arcs.  Let $\langle a_{k},b_{k} \rangle$ and $\langle a_{\ell}, b_{\ell} \rangle$ be two arcs which cross, where $a_{k}$ and $a_{\ell}$ are the L-vertices, $b_{k}$ and $b_{\ell}$ are the $R$-vertices, and assume without loss of generality that $b_{k} < b_{\ell}$.  Since the arcs cross it must be that $a_{k} < a_{\ell} <b_{k}$.  Since  $b_{k} < b_{\ell}$, we know the arc $\langle a_{k}, b_{k} \rangle$ was drawn before the arc $\langle a_{\ell}, b_{\ell}\rangle$.  At the time the arc  $\langle a_{k}, b_{k} \rangle$  was drawn, both $a_{k}$ and $a_{\ell}$ were eventual L-vertices to the left of $b_{k}$ which were not yet adjacent to an edge.  But if this were the case, then the algorithm would have chosen to connect $b_{k}$ to $a_{\ell}$ instead of $a_{k}$, since it was available and closer to $b_{k}$.  In short, the algorithm never creates a pair of crossing arcs.  Entirely similar reasoning explains why  the algorithm never creates the crossing of a through-string and an arc.  Therefore, $\psi (t)$ is always a $(n+r,n)$-web, as claimed.

We now claim that $\phi(\psi (t))=t$ for any standard Young tableau $t$.  The first row of $\phi(\psi (t))$ is filled with the labels of the L-vertices of $\psi (t)$.  However, by construction, the labels of the  L-vertices of $\psi (t)$ are precisely the entries of the first row of $t$.  Thus the entries of the first row of $\phi(\psi (t))$ and $t$ are the same. Likewise, the second rows of $\phi(\psi (t))$ and $t$ coincide.  But two standard tableaux of the same shape with the same entries in each row are necessarily equal.

Finally, we claim $\psi (\phi (w)) = w$ for any $(n+r,n)$-web $w$.  By construction the sets of L-vertices of $\psi (\phi (w))$ and $w$ coincide, as do the sets of R-vertices.  If $\psi (\phi (w))$ and $w$ were to differ in their arcs this would mean there is an R-vertex which is connected to different L-vertices in $\psi (\phi (w))$ and $w$.  Let $b$ the leftmost such R-vertex.  Say $a_{1}$ is the L-vertex connected to $b$ in $\psi (\phi (w))$ and say $a_{2}$ is the $L$-vertex connected to $b$ in $w$. We will assume $a_{1} < a_{2}$ and leave the other case to the reader.   If the edge adjacent to $a_{2}$ in  $\psi (\phi (w))$ is a through-string, then this through-string would cross the arc $\langle a_{1}, b \rangle$.  Therefore $a_{2}$ is the L-vertex of an arc in $\psi (\phi (w))$, say $\langle a_{2}, c \rangle$.  Since by assumption $\langle a_{1}, b \rangle$ is an arc of $\psi (\phi (w))$ we know $c \neq b$.   If $c > b$, then the arc $\langle a_{2}, c \rangle$ would cross the arc $\langle a_{1}, b \rangle$.  If $c < b$, then $c$ would be an $R$-vertex to the left of $b$ which is connected to two different $L$-vertices in $\psi (\phi (w))$ and $w$, contradicting our choice of $b$.   Thus, $\psi (\phi (w))$ and $w$ have the same set of arcs.  But if two webs have the same set of arcs, then it must be that their through-strings also agree and, hence,  $\psi (\phi (w))=w$.
\end{proof}

\subsection{Nesting number}\label{SS:nesting} The following generalizes the notion of nesting number defined in \cite{RT}.

Let $w$ be an $(n+r,n)$-web.  Given an arc $e=\langle a, b \rangle$ in $w$, set
\begin{equation*}
\nest (e) = \left(  \# \text{ arcs $\langle a', b' \rangle$ in $w$ with $a' < a < b < b'$} \right)  + \left(  \# \text{ through-strings  $\langle a', b' \rangle$ in $w$ with $a' < a$}\right).
\end{equation*}
  The \emph{nesting number} of the web $w$ is given by
\[
\nest (w) = \sum_{\text{$e$ is an arc of $w$}} \nest (e).
\]

If $e=\langle a, b \rangle$ is an arc in $w$, set
\begin{equation*}
\nest' (e) = \left(  \# \text{ arcs $\langle a', b' \rangle$ in $w$ with $a < a' < b' < b$} \right),
\end{equation*} and if $e =\langle a,b \rangle$ is a through-string in $w$, set
\begin{equation*}
\nest' (e) = \left(  \# \text{ arcs  $\langle a', b' \rangle$ in $w$ with $a < a' < b'$}\right).
\end{equation*} Then, obviously,
\[
\nest (w) = \sum_{\text{$e$ is an edge of $w$}} \nest' (e).
\]

For example, the $(9,3)$-web given on the right-hand side of \cref{E:etaexample} has nesting number $3$.

\subsection{The web module}\label{SS:Webmodule}  Given nonnegative integers $n$ and $r$, let $W^{(n+r,n)}$ be the $\k$-vector space with basis given by the set of $(n+r,n)$-webs, $\WW^{(n+r,n)}$. Define an action of $\HH_{d}(q)$ on $W^{(n+r,n)}$ as follows:
\begin{itemize}
\item If $i$ and $i+1$ are both through-strings, then $wT_{i}=qw$.
\item If  $i$ and $i+1$ are not both through-strings, then define the action of $T_{i}$ on $w$ by considering the following linear combination of graphs obtained via concatenation on the bottom of $w$: 

\begin{equation*}
wT_{i} \; \; =\; \; 
q\; \begin{tikzpicture}[baseline=0ex,scale=.65, color=\clr ]
	\begin{pgfonlayer}{nodelayer}
		\node [style=none] (0) at (-5, 1) {};
		\node [style=none] (1) at (0, 1) {};
		\node [style=none] (2) at (-3.5, -1) {};
		\node [style=none] (3) at (-5, 0) {};
		\node [style=none] (4) at (0, 1) {};
		\node [style=none] (5) at (-5, -1) {};
		\node [style=none] (6) at (-2.84, 0) {};
		\node [style=none] (7) at (-1, 1) {};
		\node [style=none] (8) at (0, 0) {};
		\node [style=none] (9) at (-2.5, 0.5) {$w$};
		\node [style=none] (10) at (0, -1.5) {\tiny{$2n+r$}};
		\node [style=none] (11) at (-0.75, -0.5) {\tiny{\dots}};
		\node [style=none] (12) at (-2.84, -1.475) {\tiny{$i$}};
		\node [style=none] (13) at (-2.16, 0) {};
		\node [style=none] (14) at (-5, -1.5) {\tiny{$1$}};
		\node [style=none] (15) at (-4.25, -0.5) {\tiny{\dots}};
		\node [style=none] (16) at (-2.16, -1.5) {\tiny{$i+1$}};
		\node [style=none] (17) at (-1.5, -1) {};
		\node [style=none] (18) at (0, -1)  {};
		\node [style=none] (19) at (-1.5, 0) {};
		\node [style=none] (20) at (-3.5, 0) {};
		\node [style=none] (21) at (-2.84, -1) {};
		\node [style=none] (22) at (-2.16, -1) {};
	\end{pgfonlayer}
	\begin{pgfonlayer}{edgelayer}
		\draw [style=TB Edge] (4.center) to (8.center);
		\draw [style=TB Edge] (8.center) to (3.center);
		\draw [style=TB Edge] (4.center) to (0.center);
		\draw [style=TB Edge] (0.center) to (3.center);
		\draw [style=TB Edge]  (6.center) to (21.center);
		\draw [style=TB Edge] (5.center) to (3.center);
		\draw [style=TB Edge]  (13.center) to (22.center);
		\draw [style=TB Edge] (19.center) to (17.center);
		\draw [style=TB Edge] (8.center) to (18.center);
		\draw [style=TB Edge] (20.center) to (2.center);
	\end{pgfonlayer}
\end{tikzpicture}
\; \; + \; \; \; \;
\begin{tikzpicture}[baseline=0ex,scale=.65, color=\clr ]
	\begin{pgfonlayer}{nodelayer}
		\node [style=none] (0) at (-5, 1) {};
		\node [style=none] (1) at (0, 1) {};
		\node [style=none] (2) at (-3.5, -1) {};
		\node [style=none] (3) at (-5, 0) {};
		\node [style=none] (4) at (0, 1) {};
		\node [style=none] (5) at (-5, -1) {};
		\node [style=none] (6) at (-2.84, 0) {};
		\node [style=none] (7) at (-1, 1) {};
		\node [style=none] (8) at (0, 0) {};
		\node [style=none] (9) at (-2.5, 0.5) {$w$};
		\node [style=none] (10) at (0, -1.5) {\tiny{$2n+r$}};
		\node [style=none] (11) at (-0.75, -0.5) {\tiny{\dots}};
		\node [style=none] (12) at (-2.84, -1.475) {\tiny{$i$}};
		\node [style=none] (13) at (-2.16, 0) {};
		\node [style=none] (14) at (-5, -1.5) {\tiny{$1$}};
		\node [style=none] (15) at (-4.25, -0.5) {\tiny{\dots}};
		\node [style=none] (16) at (-2.16, -1.5) {\tiny{$i+1$}};
		\node [style=none] (17) at (-1.5, -1) {};
		\node [style=none] (18) at (0, -1)  {};
		\node [style=none] (19) at (-1.5, 0) {};
		\node [style=none] (20) at (-3.5, 0) {};
		\node [style=none] (21) at (-2.84, -1) {};
		\node [style=none] (22) at (-2.16, -1) {};
	\end{pgfonlayer}
	\begin{pgfonlayer}{edgelayer}
		\draw [style=TB Edge] (4.center) to (8.center);
		\draw [style=TB Edge] (8.center) to (3.center);
		\draw [style=TB Edge] (4.center) to (0.center);
		\draw [style=TB Edge] (0.center) to (3.center);
		\draw [style=TB Cup] [out=270, in=270] (6.center) to (13.center);
		\draw [style=TB Edge] (5.center) to (3.center);
		\draw [style=TB Cup] [out=90, in=90] (21.center) to (22.center);
		\draw [style=TB Edge] (19.center) to (17.center);
		\draw [style=TB Edge] (8.center) to (18.center);
		\draw [style=TB Edge] (20.center) to (2.center);
	\end{pgfonlayer}
\end{tikzpicture}.
\end{equation*}
\end{itemize}
After concatenation there may be closed connected components which do not involve boundary vertices. These should be deleted and for each one which is deleted the remaining diagram should be scaled by $-[2]_{q}=-(q+q^{-1})$. After all such components are deleted the result will be a linear combination of $(n+r,n)$-webs.

Checking case-by-case verifies that the defining relations of $\HH_{d}(q)$ are satisfied and, hence, the following result holds.

\begin{lemma}\label{L:Webmodule}  The action given above makes $W^{(n+r,n)}$ into a well-defined right $\HH_{2n+r}(q)$-module.
\end{lemma}

\subsection{A partial order on webs}\label{SS:partialorderonwebs}  If $w$ is a $(n+r,n)$-web and $i =1, \dotsc , 2n+r-1$, then set  
\[
wE_{i}= wT_{i} - qw.
\]
 The following list summarizes the possible ways $E_{i}$ could act on a web $w$ and the resulting nesting number.  Each of these is an elementary diagrammatic calculation which we leave to the reader.

\begin{enumerate}
\item If $i$ and $i+1$ are both L-vertices of $w$ and are adjacent to through-strings, then $wE_{i} = 0$;
\item If $i$ and $i+1$ are both L-vertices of $w$ and exactly one is adjacent to a through-string, then it is necessarily $i$.  In this case, $wE_{i} = w'$ is an $(n+r,n)$-web with $\nest (w') < \nest(w)$;
\item If $i$ and $i+1$ are both L-vertices of $w$ and both are adjacent to arcs, then $wE_{i} = w'$  is an $(n+r,n)$-web with $\nest (w') < \nest(w)$;
\item If $i$ is an L-vertex and $i+1$ is an R-vertex, then necessarily an arc connects $i$ to $i+1$ and $wE_{i}=-[2]_{q}w$ with the nesting number unchanged;
\item If $i$ is an R-vertex and $i+1$ is a L-vertex, then  $wE_{i} = w'$  is an $(n+r,n)$-web with $\nest (w') = \nest(w)+1$;
\item If $i$ and $i+1$ are both R-vertices, then $wE_{i} = w'$  is an $(n+r,n)$-web with $\nest (w') < \nest(w)$;
\end{enumerate}

A contemplation of the above cases shows that there is only one case when the action of $E_{i}$ increases the nesting number of a web.  If $w, w' \in \WW^{(n+r,r)}$, write
\[
w \preceq_{i} w'
\] if $wE_{i}=w'$ and $\nest (w') > \nest (w)$.  Define a partial order $\preceq$ on $\WW^{(n+r,n)}$ by writing $w \preceq w'$ if $w=w'$ or if there are sequences $w_{1}, \dotsc , w_{t} \in \WW^{(n+r,n)}$ and $1 \leq i_{1}, \dotsc , i_{t} \leq  2n+r-1$ such that 
\[
w \preceq_{i_{1}} w_{1} \preceq_{i_{2}} w_{2} \preceq_{i_{3}}  \dotsb  \preceq_{i_{t}} w_{t} = w'.
\]  The Hasse diagram for this partial order is the set $\WW^{(n+r,n)}$ with a directed edge labeled by $i$ from $w$ to $w'$ whenever $w \preceq_{i} w'$.

Observe that the web $w_{0}$ given in \cref{E:w0definition} is the unique $(n+r,n)$-web with nesting number zero. Therefore it is the unique minimal element in $\WW^{(n+r,n)}$ with respect to the partial order $\preceq$. As another example, the following web is the unique maximal element and is the only one to have nesting number $\frac{n(n-1)}{2}+rn$:

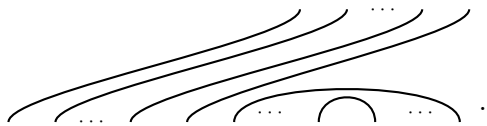
\begin{equation*}
\begin{tikzpicture}[baseline=1ex,scale=.5, color=\clr ]
	\begin{pgfonlayer}{nodelayer}
		\node [style=none] (0) at (2.25, 0.25) {\tiny{$\cdots$}};
		\node [style=none] (1) at (6.25, 0.25) {\tiny{$\cdots$}};
		\node [style=none] (5) at (-4.75, 0) {};
		\node [style=none] (6) at (-3.5, 0) {};
		\node [style=none] (7) at (-2.5, 0) {\tiny{$\cdots$}};
		\node [style=none] (8) at (0, 0) {};
		\node [style=none] (9) at (-1.5, 0) {};
		\node [style=none] (11) at (3, 3) {};
		\node [style=none] (12) at (4.25, 3) {};
		\node [style=none] (13) at (5.25, 3)  {\tiny{$\cdots$}};
		\node [style=none] (14) at (7.5, 3) {};
		\node [style=none] (15) at (6.25, 3) {};
		\node [style=none] (16) at (3.5, 0) {};
		\node [style=none] (17) at (5, 0) {};
		\node [style=none] (18) at (1.25, 0) {};
		\node [style=none] (19) at (7.25, 0) {};
	\end{pgfonlayer}
	\begin{pgfonlayer}{edgelayer}
		\draw [thick] [out=90,in=90, looseness=1.5] (16.center) to (17.center);
		\draw [thick, looseness=.5] [out=90,in=90] (18.center) to (19.center);
		\draw [thick, looseness=.25] [out=90,in=270] (5.center) to (11.center);
		\draw [thick, looseness=.25] [out=90,in=270] (6.center) to (12.center);
		\draw [thick, looseness=.25] [out=90,in=270] (9.center) to (15.center);
		\draw [thick, looseness=.25] [out=90,in=270] (8.center) to (14.center);
	\end{pgfonlayer}
\end{tikzpicture} \; .
\end{equation*}
 Note that whenever $w\preceq_{i} w'$, one has $\nest (w') = \nest (w)+1$.  Thus $\nest (w)$ equals the length of any path from $w_{0}$ to $w$ in the Hasse diagram for $\preceq$.

\begin{prop}\label{P:OrderingProposition} Let $\lambda = (n+r,n)$. For all $w, w' \in \WW^{\lambda}$ and all $1 \leq i < 2n+r$, the bijection 
\[
(-)^{\TT} \circ \phi : \WW^{\lambda} \to \Std (\lambda^{\TT}) 
\]
satisfies
\[
\phi (w)^{\TT} \leq_{i} \phi (w')^{\TT} \text{ if and only if }   w \preceq_{i}  w'.
\]   
\end{prop}

\begin{proof}  Say $w \preceq_{i}  w'$.  Since $wE_{i}= w'$, the sets of $L$-vertices and $R$-vertices of $w$ and $w'$ are identical with one exception: in $w$, $i$ is an R-vertex and $i+1$ is an L-vertex, while in $w'$ these are reversed. Consequently the entries of the first row of  $\phi (w)$ and $\phi (w')$ are identical except that the first contains $i+1$ while the second contains $i$.  The same statement with $i$ and $i+1$ reversed applies to their second rows, as well. As a consequence,   $\phi (w)^{\TT}.s_{i}=\phi (w')^{\TT}$.  Furthermore, since both $\phi (w)$ and $\phi (w')$ are standard, it must be that $i$ appears in a column of $\phi (w)$ which is to strictly to the left of the column containing $i+1$.  That is,  $i$ occurs earlier than $i+1$ when reading the entries of $\phi (w)^{\TT}$ from left-to-right, top-to-bottom.  As discussed in \cref{R:weakBruhatOrder}, this shows $\phi(w)^{\TT} \leq_{i} \phi (w')^{\TT}$.

Now say $\phi(w)^{\TT} \leq_{i} \phi (w')^{\TT}$.   Then $\phi (w)^{\TT}.s_{i}=\phi (w')^{\TT}$ and $i$ occurs earlier than $i+1$ when reading the entries of $\phi (w)^{\TT}$ from left-to-right, top-to-bottom.  Since both $\phi (w)^{\TT}$ and $\phi (w')^{\TT}$ are standard, it cannot be that $i$ and $i+1$ are in the same row or column of $\phi(w)^{\TT}$.  Furthermore, if $i$ were to be in the first column of $\phi (w)^{\TT}$ and $i+1$ in the second column of $\phi (w)^{\TT}$, then since $i$ occurs earlier than $i+1$ when reading the entries of $\phi (w)^{\TT}$, it would be that $i$ is strictly above $i+1$ in $\phi (w)^{\TT}$.  But in this case there would be an entry of $\phi (w)^{\TT}$ which is below $i$ and to the left of $i+1$ and this entry would be larger than $i$ and smaller than $i+1$. As this cannot happen, it must be that $i$ is in the second column of $\phi (w)^{\TT}$ and $i+1$ is in the first column of $\phi (w)^{\TT}$.  That is, $i$ is in the second row of $\phi (w)$ and $i+1$ is in the first row of $\phi (w)$.  That is, $i$ is an $R$-vertex of $w$ and $i+1$ is a $L$-vertex of $w$.  Hence, $\nest (w) < \nest (wE_{i})$ and $w \preceq_{i} wE_{i}$.

It remains to verify that $wE_{i}=w'$.  On the one hand, the sets of $L$-vertices and $R$-vertices of $w$ and $wE_{i}$ are identical with one exception: in $w$, $i$ is a R-vertex and $i+1$ is an L-vertex, while in $wE_{i}$ these are reversed.  That is, $\phi(wE_{i})=\phi (w).s_{i}$.  But since $\phi(w)^{\TT} \leq_{i} \phi (w')^{\TT}$, it is also true that $\phi(w')=\phi (w).s_{i}$. Since $\phi$ is a bijection, the claim follows.
\end{proof}

In particular, $(-)^{\TT} \circ \phi$ defines an isomorphism of directed, labeled graphs between the Hasse diagrams  $(\WW^{(n+r,n)}, \preceq)$ and  $(\Std (\lambda^{\TT}), \leq)$ and, hence, for $w \in \WW^{(n+r,n)}$, 
\[
\nest (w) = \ell (\eta^{-1}(\phi(w)^{\TT})).
\]

\begin{remark} Using the previous proposition, the statements about the weak Bruhat order on $\Std ((n+r,n)^{\TT})$ given in \cref{L:WeakBruhatProperties} directly translate to statements about the partial order on $\WW^{(n+r,n)}$.  In \cref{S:MainResults} we use the web versions of these results without further comment.
\end{remark}

\begin{lemma}\label{L:PartialOrder}  Let $w, w' \in \WW^{(n+r,n)}$ be one of the following pairs of non-crossing webs. In each of the pairs only a local region of the webs is drawn, it is assumed that $w$ and $w'$ are identical other than as indicated, and that every edge adjacent to a vertex in the region is represented in the diagram.  Finally, in the given diagrams, $d_{1}$, $d_{2}$, and $d_{3}$ are possibly empty, fixed, non-crossing webs consisting of only arcs.
\begin{enumerate}
\item 
\begin{equation*}
w \;\; =\;\;
\begin{tikzpicture}[baseline=0ex,scale=.4, color=\clr ]
	\begin{pgfonlayer}{nodelayer}
		\node [style=none] (0) at (-1, 0) {};
		\node [style=none] (1) at (1, 0) {};
		\node [style=none] (2) at (-5, 0) {};
		\node [style=none] (3) at (5, 0) {};
		\node [style=none] (4) at (-1, -0.5) {\tiny{$i$}};
		\node [style=none] (5) at (1, -0.5) {\tiny{$i+1$}};
		\node [style=none] (6) at (-5, -0.5) {\tiny{$p$}};
		\node [style=none] (7) at (5, -0.5) {\tiny{$q$}};
		\node [draw] (8) at (-3, 0.5) {$d_{1}$};
		\node [draw] (9) at (3, 0.5) {$d_{2}$};
	\end{pgfonlayer}
	\begin{pgfonlayer}{edgelayer}
		\draw [thick, looseness=1.5] [out=90,in=90] (0.center) to (2.center);
		\draw [thick, looseness=1.5] [out=90,in=90] (1.center) to (3.center);
	\end{pgfonlayer}
\end{tikzpicture} \;\;  \text{ and } \;\; 
w' \;\; =\;\;
\begin{tikzpicture}[baseline=0ex,scale=.4, color=\clr ]
	\begin{pgfonlayer}{nodelayer}
		\node [style=none] (0) at (-1, 0) {};
		\node [style=none] (1) at (1, 0) {};
		\node [style=none] (2) at (-5, 0) {};
		\node [style=none] (3) at (5, 0) {};
		\node [style=none] (5) at (1, -0.5) {\tiny{$i+1$}};
		\node [style=none] (6) at (-5, -0.5) {\tiny{$p$}};
		\node [style=none] (7) at (5, -0.5) {\tiny{$q$}};
		\node [draw] (8) at (-1.5, 0.5) {$d_{1}$};
		\node [draw] (9) at (3, 0.5) {$d_{2}$};
		\node [style=none] (10) at (-4, 0) {};
		\node [style=none] (11) at (-4, -0.5) {\tiny{$p+1$}};
	\end{pgfonlayer}
	\begin{pgfonlayer}{edgelayer}
		\draw [thick, looseness=1.5] [out=90,in=90] (10.center) to (1.center);
		\draw [thick, looseness=1.25] [out=90,in=90] (2.center) to (3.center);
	\end{pgfonlayer}
\end{tikzpicture} \; ;
\end{equation*}
\item 
\begin{equation*}
w \;\; =\;\;
\begin{tikzpicture}[baseline=0ex,scale=0.4, color=\clr ]
	\begin{pgfonlayer}{nodelayer}
		\node [style=none] (0) at (-1, 0) {};
		\node [style=none] (1) at (1, 0) {};
		\node [style=none] (2) at (-5, 0) {};
		\node [style=none] (3) at (5, 0) {};
		\node [style=none] (4) at (-1, -0.5) {\tiny{$i$}};
		\node [style=none] (5) at (1, -0.5) {\tiny{$i+1$}};
		\node [style=none] (6) at (-5, -0.5) {\tiny{$p$}};
		\node [style=none] (7) at (5, -0.5) {\tiny{$q$}};
		\node [draw] (8) at (-3, 0.5) {$d_{1}$};
		\node [draw] (9) at (3, 0.5) {$d_{2}$};
	\end{pgfonlayer}
	\begin{pgfonlayer}{edgelayer}
		\draw [thick, looseness=1.5] [out=90,in=90] (0.center) to (2.center);
		\draw [thick, looseness=1.5] [out=90,in=90] (1.center) to (3.center);
	\end{pgfonlayer}
\end{tikzpicture} \;\;  \text{ and } \;\; 
w' \;\; =\;\;
\begin{tikzpicture}[baseline=0ex,scale=0.4, color=\clr ]
	\begin{pgfonlayer}{nodelayer}
		\node [style=none] (0) at (-1, 0) {};
		\node [style=none] (1) at (1, 0) {};
		\node [style=none] (2) at (-5, 0) {};
		\node [style=none] (3) at (5, 0) {};
		\node [style=none] (4) at (-1, -0.5) {\tiny{$i$}};
		\node [style=none] (6) at (-5, -0.5) {\tiny{$p$}};
		\node [style=none] (7) at (5, -0.5) {\tiny{$q$}};
		\node [draw] (8) at (-3, 0.5) {$d_{1}$};
		\node [draw] (9) at (1.5, 0.5) {$d_{2}$};
		\node [style=none] (12) at (4, 0) {};
		\node [style=none] (13) at (4, -0.5) {\tiny{$q-1$}};
	\end{pgfonlayer}
	\begin{pgfonlayer}{edgelayer}
		\draw [thick, looseness=1.5] [out=90,in=90] (12.center) to (0.center);
		\draw [thick, looseness=1.25] [out=90,in=90] (2.center) to (3.center);
	\end{pgfonlayer}
\end{tikzpicture} \; ;
\end{equation*}
\item 
\begin{equation*}
w \;\; =\;\;
\begin{tikzpicture}[baseline=0ex,scale=0.4, color=\clr ]
	\begin{pgfonlayer}{nodelayer}
		\node [style=none] (0) at (-1, 0) {};
		\node [style=none] (1) at (2, 0) {};
		\node [style=none] (2) at (-5, 0) {};
		\node [style=none] (3) at (6, 0) {};
		\node [style=none] (4) at (-1, -0.5) {\tiny{$i$}};
		\node [style=none] (5) at (2, -0.5) {\tiny{$j$}};
		\node [style=none] (6) at (-5, -0.5) {\tiny{$p$}};
		\node [style=none] (7) at (6, -0.5) {\tiny{$q$}};
		\node [draw] (8) at (-3, 0.5) {$d_{1}$};
		\node [draw] (9) at (0.5, 0.5) {$d_{2}$};
		\node [draw] (10) at (4, 0.5) {$d_{3}$};
	\end{pgfonlayer}
	\begin{pgfonlayer}{edgelayer}
		\draw [thick, looseness=1.5] [out=90,in=90] (0.center) to (2.center);
		\draw [thick, looseness=1.5] [out=90,in=90] (1.center) to (3.center);
	\end{pgfonlayer}
\end{tikzpicture} \;\;  \text{ and } \;\; 
w' \;\; =\;\;
\begin{tikzpicture}[baseline=0ex,scale=0.4, color=\clr ]
	\begin{pgfonlayer}{nodelayer}
		\node [style=none] (0) at (-1, 0) {};
		\node [style=none] (1) at (2, 0) {};
		\node [style=none] (2) at (-5, 0) {};
		\node [style=none] (3) at (6, 0) {};
		\node [style=none] (4) at (-1, -0.5) {\tiny{$i$}};
		\node [style=none] (5) at (2, -0.5) {\tiny{$j$}};
		\node [style=none] (6) at (-5, -0.5) {\tiny{$p$}};
		\node [style=none] (7) at (6, -0.5) {\tiny{$q$}};
		\node [draw] (8) at (-3, 0.5) {$d_{1}$};
		\node [draw] (9) at (0.5, 0.5) {$d_{2}$};
		\node [draw] (10) at (4, 0.5) {$d_{3}$};
	\end{pgfonlayer}
	\begin{pgfonlayer}{edgelayer}
		\draw [thick, looseness=2.5] [out=90,in=90] (0.center) to (1.center);
		\draw [thick, looseness=1.25] [out=90,in=90] (2.center) to (3.center);
	\end{pgfonlayer}
\end{tikzpicture} \; ;
\end{equation*}
\item \begin{equation*}
w \;\; =\;\;
\begin{tikzpicture}[baseline=0ex,scale=0.4, color=\clr ]
	\begin{pgfonlayer}{nodelayer}
		\node [style=none] (0) at (-1, 0) {};
		\node [style=none] (1) at (1, 0) {};
		\node [style=none] (2) at (-5, 0) {};
		\node [style=none] (3) at (3, 0) {};
		\node [style=none] (4) at (-1, -0.5) {\tiny{$i$}};
		\node [style=none] (6) at (-5, -0.5) {\tiny{$p$}};
		\node [style=none] (7) at (3, -0.5) {\tiny{$q$}};
		\node [draw] (8) at (-3, 0.5) {$d_{1}$};
		\node [draw] (9) at (1, 0.5) {$d_{2}$};
		\node [style=none] (30) at (5, 4) {};
		\node [style=none] (31) at (5, 4.5) {\tiny{$k'$}};
	\end{pgfonlayer}
	\begin{pgfonlayer}{edgelayer}
		\draw [thick, looseness=1.5] [out=90,in=90] (0.center) to (2.center);
		\draw [thick, looseness=1.25] [out=270,in=90] (30.center) to (3.center);
	\end{pgfonlayer}
\end{tikzpicture} \;\;  \text{ and } \;\; 
w'\;\; =\;\;
\begin{tikzpicture}[baseline=0ex,scale=0.4, color=\clr ]
	\begin{pgfonlayer}{nodelayer}
		\node [style=none] (0) at (1, 0) {};
		\node [style=none] (1) at (1, 0) {};
		\node [style=none] (2) at (-4.5, 0) {};
		\node [style=none] (3) at (5, 0) {};
		\node [style=none] (4) at (1, -0.5) {\tiny{$i+1$}};
		\node [style=none] (6) at (-4.5, -0.5) {\tiny{$p$}};
		\node [draw] (8) at (-1, 0.5) {$d_{1}$};
		\node [draw] (9) at (3, 0.5) {$d_{2}$};
		\node [style=none] (10) at (-3, 0) {};
		\node [style=none] (11) at (-3, -0.5) {\tiny{$p+1$}};
		\node [style=none] (30) at (5, 4) {};
		\node [style=none] (31) at (5, 4.5) {\tiny{$k'$}};
	\end{pgfonlayer}
	\begin{pgfonlayer}{edgelayer}
		\draw [thick, looseness=1.5] [out=90,in=90] (10.center) to (0.center);
		\draw [thick, looseness=1] [out=90,in=270] (2.center) to (30.center);
	\end{pgfonlayer}
\end{tikzpicture} \; .
\end{equation*}
\end{enumerate}

For each pair, 
\[
w \preceq w'.
\]
\end{lemma}

\begin{proof}  For the pairs in (1), (2), and (3), since the argument is local and only involves arcs, the claim follows from \cite[Lemma 4.9, Theorem 4.10]{RT}.

For the pair in (4), first assume that $d_{2}$ is empty.  In which case $q=i+1$.  Then $i$ is an $R$-vertex, $q$ is an $L$-vertex, and so $ w \preceq wE_{i}$.  Furthermore, 
\begin{equation*}
wE_{i}\;\; =\;\;
\begin{tikzpicture}[baseline=0ex,scale=0.4, color=\clr ]
	\begin{pgfonlayer}{nodelayer}
		\node [style=none] (0) at (0, 0) {};
		\node [style=none] (1) at (3, 0) {};
		\node [style=none] (2) at (-4.5, 0) {};
		\node [style=none] (3) at (5, 0) {};
		\node [style=none] (4) at (0, -0.5) {\tiny{$i$}};
		\node [style=none] (5) at (3, -0.5) {\tiny{$i+1$}};
		\node [style=none] (6) at (-4.5, -0.5) {\tiny{$p$}};
		\node [draw] (8) at (-2, 0.5) {$d_{1}$};
		\node [style=none] (10) at (-3, 0) {};
		\node [style=none] (30) at (5, 4) {};
		\node [style=none] (31) at (5, 4.5) {\tiny{$k'$}};
	\end{pgfonlayer}
	\begin{pgfonlayer}{edgelayer}
		\draw [thick, looseness=1.5] [out=90,in=90] (1.center) to (0.center);
		\draw [thick, looseness=1] [out=90,in=270] (2.center) to (30.center);
	\end{pgfonlayer}
\end{tikzpicture} \; .
\end{equation*}  By repeatedly using (1) on the subdiagram obtained by excluding the through-string which connects vertex $p$ to vertex $k'$, one can move the arc $\langle i, i+1 \rangle$ so that it covers $d_{1}$ and obtain $w'$.  That is, one can construct a chain of webs $wE_{i} \preceq w_{1} \preceq w_{2} \preceq \dotsb \preceq w'$.  It follows that $w \preceq w'$ in the case when $d_{2}$ is empty.

To obtain the general case one only need observe that in general $d_{2}$ consists of the horizontal concatenation of subdiagrams which look like:
\begin{equation*}
\begin{tikzpicture}[baseline=0ex,scale=0.4, color=\clr ]
	\begin{pgfonlayer}{nodelayer}
		\node [style=none] (0) at (0, 0) {};
		\node [style=none] (2) at (-5, 0) {};
		\node [draw] (8) at (-2.5, 0.5) {$d_{3}$};
	\end{pgfonlayer}
	\begin{pgfonlayer}{edgelayer}
		\draw [thick, looseness=1.5] [out=90,in=90] (2.center) to (0.center);
	\end{pgfonlayer}
\end{tikzpicture} \; ,
\end{equation*} where $d_{3}$ is a possibly empty non-crossing web.  By starting with $w$ and repeatedly applying the result of the previous paragraph, the through-string can be ``moved'' past each of these subdiagrams and, hence, past $d_{2}$.  Using the result one more time to move the through-string past the arc containing $d_{1}$ yields $w'$.  In short, there is once again a chain of webs $w \preceq w_{1} \preceq w_{2} \preceq \dotsb \preceq w'$ and the result follows.
 \end{proof}

\subsection{Bubbles and crossings}\label{SS:bubbles and crossings} In this section we expand the allowable $(n+r,n)$-webs to include diagrams with crossings and bubbles.  Namely, introduce the following local relations in $W^{(n+r,n)}$:

\begin{equation}\label{E:crossingrelation}
\begin{tikzpicture}[baseline=0ex,scale=.5, color=\clr ]
	\begin{pgfonlayer}{nodelayer}
		\node [style=none] (0) at (-1, -1.5) {};
		\node [style=none] (1) at (1, -1.5) {};
		\node [style=none] (2) at (-1, 1.5) {};
		\node [style=none] (3) at (1, 1.5) {};
		\node [style=none] (4) at (0.28, -0.4) {};
		\node [style=none] (5) at (-0.28, 0.4) {};
	\end{pgfonlayer}
	\begin{pgfonlayer}{edgelayer}
		\draw [thick] (2.center) to (1.center);
		\draw [ultra thick, color=white] (4.center) to (5.center);
		\draw [thick] (3.center) to (0.center);
	\end{pgfonlayer}
\end{tikzpicture}
\; \; := \; \;
q\;\;  \begin{tikzpicture}[baseline=0ex,scale=.5, color=\clr ]
	\begin{pgfonlayer}{nodelayer}
		\node [style=none] (0) at (-1, -1.5) {};
		\node [style=none] (1) at (1, -1.5) {};
		\node [style=none] (2) at (-1, 1.5) {};
		\node [style=none] (3) at (1, 1.5) {};
		\node [style=none] (4) at (0.28, -0.4) {};
		\node [style=none] (5) at (-0.28, 0.4) {};
	\end{pgfonlayer}
	\begin{pgfonlayer}{edgelayer}
		\draw [thick] (0.center) to (2.center);
		\draw [thick] (1.center) to (3.center);
	\end{pgfonlayer}
\end{tikzpicture}
\; \; + \; \;
\begin{tikzpicture}[baseline=0ex,scale=.5, color=\clr ]
	\begin{pgfonlayer}{nodelayer}
		\node [style=none] (0) at (-1, -1.5) {};
		\node [style=none] (1) at (1, -1.5) {};
		\node [style=none] (2) at (-1, 1.5) {};
		\node [style=none] (3) at (1, 1.5) {};
		\node [style=none] (4) at (0.28, -0.4) {};
		\node [style=none] (5) at (-0.28, 0.4) {};
	\end{pgfonlayer}
	\begin{pgfonlayer}{edgelayer}
		\draw [thick, looseness=1.25] [out=90,in=90] (0.center) to (1.center);
		\draw [thick, looseness=1.25] [out=270,in=270] (2.center) to (3.center);
	\end{pgfonlayer}
\end{tikzpicture} \; ,
\end{equation} and
\begin{equation}\label{E:bubblerelation}
\begin{tikzpicture}[baseline=0ex,scale=.65, color=\clr ]
	\begin{pgfonlayer}{nodelayer}
		\node [style=none] (0) at (-1, 0) {};
		\node [style=none] (1) at (1, 0) {};
	\end{pgfonlayer}
	\begin{pgfonlayer}{edgelayer}
		\draw [thick] (0,0) circle [radius=1];

	\end{pgfonlayer}
\end{tikzpicture}
\; \; := \; \; -[2]_{q} \; .
\end{equation}   By the first we mean that, by definition, a web with a crossing of two edges as indicated is shorthand for the linear combination of two webs obtained by resolving the crossing using the equality given in \cref{E:crossingrelation} and otherwise leaving the web unchanged.  Similarly, the second is a shorthand for whenever a web has an isolated connected component (a ``bubble'') it is equal to the web obtained by deleting the component and scaling the resulting web by $-[2]_{q}$.  Going forward it is also convenient to allow for webs with one or more edges connecting a pair of vertices from among $\{1', \dotsc , r' \}$ by declaring such webs are equal to zero.

For example, the diagram
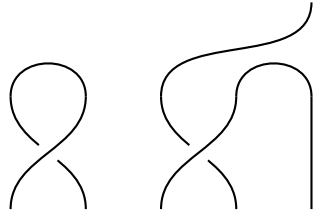
\begin{equation*}
\begin{tikzpicture}[baseline=0ex,scale=.5, color=\clr ]
	\begin{pgfonlayer}{nodelayer}
		\node [style=none] (0) at (-3, -1.5) {};
		\node [style=none] (1) at (-1, -1.5) {};
		\node [style=none] (2) at (-3, 1.5) {};
		\node [style=none] (3) at (-1, 1.5) {};
		\node [style=none] (4) at (-1.75, -0.2) {};
		\node [style=none] (5) at (-2.25, 0.2) {};
		\node [style=none] (6) at (1, -1.5) {};
		\node [style=none] (7) at (3, -1.5) {};
		\node [style=none] (8) at (1, 1.5) {};
		\node [style=none] (9) at (3, 1.5) {};
		\node [style=none] (10) at (2.25, -0.2) {};
		\node [style=none] (11) at (1.75, 0.2) {};
		\node [style=none] (12) at (5, 1.5) {};
		\node [style=none] (13) at (5, -1.5) {};
		\node [style=none] (14) at (5, 4) {};
	\end{pgfonlayer}
	\begin{pgfonlayer}{edgelayer}
		\draw [thick] [out=270,in=90] (2.center) to (1.center);
		\draw [ultra thick, color=white] (4.center) to (5.center);
		\draw [thick] [out=270,in=90] (3.center) to (0.center);
		\draw [thick] [out=270,in=90] (8.center) to (7.center);
		\draw [ultra thick, color=white] (10.center) to (11.center);
		\draw [thick] [out=270,in=90] (9.center) to (6.center);
		\draw [thick] (12.center) to (13.center);
		\draw [thick, looseness=1.5] [out=90,in=90] (3.center) to (2.center);
		\draw [thick, looseness=1.5] [out=90,in=90] (9.center) to (12.center);
		\draw [thick] [out=90,in=270] (8.center) to (14.center);
	\end{pgfonlayer}
\end{tikzpicture}
\end{equation*} is a shorthand for the following linear combination of non-crossing $(2+1,2)$-webs:
\begin{equation*}
q(q-[2]_{q}) \; \; 
\begin{tikzpicture}[baseline=-2ex,scale=.5, color=\clr ]
	\begin{pgfonlayer}{nodelayer}
		\node [style=none] (0) at (-3, -1.5) {};
		\node [style=none] (1) at (-1, -1.5) {};
		\node [style=none] (2) at (-3, 1.5) {};
		\node [style=none] (3) at (-1, 1.5) {};
		\node [style=none] (4) at (-1.75, -0.2) {};
		\node [style=none] (5) at (-2.25, 0.2) {};
		\node [style=none] (6) at (1, -1.5) {};
		\node [style=none] (7) at (3, -1.5) {};
		\node [style=none] (8) at (1, 1.5) {};
		\node [style=none] (9) at (3, 1.5) {};
		\node [style=none] (10) at (2.25, -0.2) {};
		\node [style=none] (11) at (1.75, 0.2) {};
		\node [style=none] (12) at (5, 1.5) {};
		\node [style=none] (13) at (5, -1.5) {};
	\end{pgfonlayer}
	\begin{pgfonlayer}{edgelayer}
		\draw [thick, looseness=1.5] [out=90,in=90] (0.center) to (1.center);
		\draw [thick, looseness=1.5] [out=90,in=90] (7.center) to (13.center);
		\draw [thick] [out=90,in=270] (6.center) to (12.center);
	\end{pgfonlayer}
\end{tikzpicture}
\; \; + \; \; (q-[2]_{q}) \; \;
\begin{tikzpicture}[baseline=-2ex,scale=.5, color=\clr ]
	\begin{pgfonlayer}{nodelayer}
		\node [style=none] (0) at (-3, -1.5) {};
		\node [style=none] (1) at (-1, -1.5) {};
		\node [style=none] (2) at (-3, 1.5) {};
		\node [style=none] (3) at (-1, 1.5) {};
		\node [style=none] (4) at (-1.75, -0.2) {};
		\node [style=none] (5) at (-2.25, 0.2) {};
		\node [style=none] (6) at (1, -1.5) {};
		\node [style=none] (7) at (3, -1.5) {};
		\node [style=none] (8) at (1, 1.5) {};
		\node [style=none] (9) at (3, 1.5) {};
		\node [style=none] (10) at (2.25, -0.2) {};
		\node [style=none] (11) at (1.75, 0.2) {};
		\node [style=none] (12) at (5, 1.5) {};
		\node [style=none] (13) at (5, -1.5) {};
	\end{pgfonlayer}
	\begin{pgfonlayer}{edgelayer}
		\draw [thick, looseness=1.5] [out=90,in=90] (0.center) to (1.center);
		\draw [thick, looseness=1.5] [out=90,in=90] (7.center) to (6.center);
		\draw [thick] [out=90,in=270] (13.center) to (12.center);
	\end{pgfonlayer}
\end{tikzpicture}\; .
\end{equation*}

From now on, an $(n+r,n)$-web is allowed to have finitely many crossings and bubbles.  If we need to emphasize that an $(n+r,n)$-web is in fact an element of $\WW^{(n+r,n)}$, then we will refer to it as a non-crossing web.

The action of $\HH_{2n+r}(q)$ on $W^{(n+r,n)}$ is particularly easy to describe using this expanded set of webs.  Namely, given $m \geq 1$ and $1 \leq  i \leq n-1$,  define
\begin{equation*}
\beta_{i}=\beta^{m}_{i}
\; \; = \; \;
\begin{tikzpicture}[baseline=0ex,scale=.65, color=\clr ]
	\begin{pgfonlayer}{nodelayer}
		\node [style=none] (0) at (-1, 1) {};
		\node [style=none] (1) at (1, 1) {};
		\node [style=none] (2) at (-1, -1) {};
		\node [style=none] (3) at (1, -1) {};
		\node [style=none] (4) at (-2, 1) {};
		\node [style=none] (5) at (-2, -1) {};
		\node [style=none] (6) at (2, 1) {};
		\node [style=none] (7) at (2, -1) {};
		\node [style=none] (8) at (-4, 1) {};
		\node [style=none] (9) at (-4, -1) {};
		\node [style=none] (10) at (-5, 1) {};
		\node [style=none] (11) at (-5, -1) {};
		\node [style=none] (12) at (5, 1) {};
		\node [style=none] (13) at (5, -1) {};
		\node [style=none] (14) at (4, 1) {};
		\node [style=none] (15) at (4, -1) {};
		\node [style=none] (16) at (-3, 0) {\tiny{$\cdots$}};
		\node [style=none] (17) at (3, 0) {\tiny{$\cdots$}};
		\node [style=none] (18) at (-1, 1.5) {$i$};
		\node [style=none] (19) at (1, 1.5) {$i+1$};
		\node [style=none] (20) at (-1, -1.5) {$i$};
		\node [style=none] (21) at (1, -1.5) {$i+1$};
		\node [style=none] (22) at (-0.25, 0.12) {};
		\node [style=none] (23) at (0.25, -0.12) {};
	\end{pgfonlayer}
	\begin{pgfonlayer}{edgelayer}
		\draw [thick] [out=270,in=90] (0.center) to (3.center);
		\draw [ultra thick, color=white] (22.center) to (23.center);
		\draw [thick] [out=270, in=90] (1.center) to (2.center);
		\draw [thick] (6.center) to (7.center);
		\draw [thick] (14.center) to (15.center);
		\draw [thick] (12.center) to (13.center);
		\draw [thick] (4.center) to (5.center);
		\draw [thick] (8.center) to (9.center);
		\draw [thick] (10.center) to (11.center);
	\end{pgfonlayer}
\end{tikzpicture} \; \; ,
\end{equation*} where there are a total of $m$ edges with only the $i$th and $(i+1)$th edges crossed.  Given an $(n+r,n)$-web $w$ and $1 \leq i \leq 2n+r-1$, write $w\beta_{i}$ for the web given by concatenating $\beta_{i}$ onto the bottom of $w$. Since we declare that webs with an edge connecting vertices from among $\{1', \dotsc , r' \}$ are zero,  $wT_{i}=w\beta_{i}$ in $W^{(n+r,n)}$.

It is a satisfying exercise to verify the defining relations of $\HH_{d}(q)$ using the crossing and bubble relations.  These relations can also be used to obtain the following version of the Reidemeister II move.

\begin{lemma}\label{L:RII}  In $W^{(n+r,n)}$ the following local relations hold:
\begin{equation*}
\begin{tikzpicture}[baseline=0ex,scale=.4, color=\clr ]
	\begin{pgfonlayer}{nodelayer}
		\node [style=none] (0) at (-2, 3) {};
		\node [style=none] (1) at (0, 3) {};
		\node [style=none] (3) at (0, 1) {};
		\node [style=none] (5) at (-5, 0) {};
		\node [style=none] (6) at (-5, -2) {};
		\node [style=none] (7) at (-3, -2) {};
		\node [style=none] (12) at (-5, 3) {};
		\node [style=none] (13) at (-1.25, 2.12) {};
		\node [style=none] (14) at (-0.75, 1.88) {};
		\node [style=none] (15) at (-4.25, -0.88) {};
		\node [style=none] (16) at (-3.75, -1.13) {};
		\node [style=none] (23) at (0, -2) {};
	\end{pgfonlayer}
	\begin{pgfonlayer}{edgelayer}
		\draw [thick] [out=270, in=90] (0.center) to (3.center);
		\draw [ultra thick, color=white] (13.center) to (14.center);
		\draw [thick] (12.center) to (5.center);
		\draw [thick] [out=270, in=90] (5.center) to (7.center);
		\draw [ultra thick, color=white] (15.center) to (16.center);
		\draw [thick] [out=90,in=90, looseness=1.5] (0.center) to (12.center);
		\draw [thick] (1.center) to (6.center);
		\draw [thick] (3.center) to (23.center);
	\end{pgfonlayer}
\end{tikzpicture}
\;\; = \; \;
q\; \begin{tikzpicture}[baseline=0ex,scale=.4, color=\clr ]
	\begin{pgfonlayer}{nodelayer}
		\node [style=none] (0) at (-2, 3) {};
		\node [style=none] (1) at (0, 3) {};
		\node [style=none] (3) at (0, -2) {};
		\node [style=none] (5) at (-5, 0) {};
		\node [style=none] (6) at (-5, -2) {};
		\node [style=none] (7) at (-3, -2) {};
		\node [style=none] (12) at (-5, 3) {};
		\node [style=none] (13) at (-1.25, 2.12) {};
		\node [style=none] (14) at (-0.75, 1.88) {};
		\node [style=none] (15) at (-4.25, -0.88) {};
		\node [style=none] (16) at (-3.75, -1.13) {};
	\end{pgfonlayer}
	\begin{pgfonlayer}{edgelayer}
		\draw [thick] [out=90,in=90, looseness=1.5] (3.center) to (7.center);
		\draw [thick] (1.center) to (6.center);
	\end{pgfonlayer}
\end{tikzpicture}
\hspace{.5in} \text{ and } \hspace{.5in}
\begin{tikzpicture}[baseline=0ex,xscale=-1, scale=.4, color=\clr ]
	\begin{pgfonlayer}{nodelayer}
		\node [style=none] (0) at (-2, 3) {};
		\node [style=none] (1) at (0, 3) {};
		\node [style=none] (3) at (0, 1) {};
		\node [style=none] (5) at (-5, 0) {};
		\node [style=none] (6) at (-5, -2) {};
		\node [style=none] (7) at (-3, -2) {};
		\node [style=none] (12) at (-5, 3) {};
		\node [style=none] (13) at (-0.75, 2.25) {};
		\node [style=none] (14) at (-1.25, 1.75) {};
		\node [style=none] (15) at (-3.75, -0.75) {};
		\node [style=none] (16) at (-4.25, -1.25) {};
		\node [style=none] (23) at (0, -2) {};
	\end{pgfonlayer}
	\begin{pgfonlayer}{edgelayer}
		\draw [thick] (1.center) to (6.center);
		\draw [ultra thick, color=white] (13.center) to (14.center);
		\draw [thick] [out=270, in=90] (0.center) to (3.center);
		\draw [thick] (12.center) to (5.center);
		\draw [ultra thick, color=white] (15.center) to (16.center);
		\draw [thick] [out=270, in=90] (5.center) to (7.center);
		\draw [thick] [out=90,in=90, looseness=1.5] (0.center) to (12.center);
		\draw [thick] (3.center) to (23.center);
	\end{pgfonlayer}
\end{tikzpicture}
\;\; = \; \;
q\; \begin{tikzpicture}[baseline=0ex,xscale=-1,scale=.4, color=\clr ]
	\begin{pgfonlayer}{nodelayer}
		\node [style=none] (0) at (-2, 3) {};
		\node [style=none] (1) at (0, 3) {};
		\node [style=none] (3) at (0, -2) {};
		\node [style=none] (5) at (-5, 0) {};
		\node [style=none] (6) at (-5, -2) {};
		\node [style=none] (7) at (-3, -2) {};
		\node [style=none] (12) at (-5, 3) {};
		\node [style=none] (13) at (-1.25, 2.12) {};
		\node [style=none] (14) at (-0.75, 1.88) {};
		\node [style=none] (15) at (-4.25, -0.88) {};
		\node [style=none] (16) at (-3.75, -1.13) {};
	\end{pgfonlayer}
	\begin{pgfonlayer}{edgelayer}
		\draw [thick] [out=90,in=90, looseness=1.5] (3.center) to (7.center);
		\draw [thick] (1.center) to (6.center);
	\end{pgfonlayer}
\end{tikzpicture}\;  .
\end{equation*}
\end{lemma}

\begin{remark}\label{R:LRvertices}
Note that the definition of $L$-vertex and $R$-vertex directly generalizes to $(n+r,n)$-webs with crossings and bubbles. In particular, for any $(n+r,n)$-web $w$ and any $1 \leq i \leq 2n+r-1$,  the $L$- and $R$-vertices of $w\beta_{i}$ are the same as for $w$ except that the vertex type of $i$ and $i+1$ are interchanged (e.g., if in $w$ $i$ is an $R$-vertex and $i+1$ is an $L$-vertex, then in $w\beta_{i}$ $i$ is an $L$-vertex and $i+1$ is an $R$-vertex).
\end{remark}

We end this section by giving names to several local diagrams which will be relevant in the sequel.  By ``local diagram'' we mean an area of a possibly larger web such that within this region it looks as depicted.   The local diagrams shown below will be called \emph{bigons}, \emph{triangles}, and \emph{squares}, respectively:
\begin{equation}\label{E:TriangleSquare}
\begin{tikzpicture}[baseline=2ex,scale=.4, color=\clr ]
	\begin{pgfonlayer}{nodelayer}
		\node [style=none] (0) at (-2, 3) {};
		\node [style=none] (1) at (0, 3) {};
		\node [style=none] (3) at (0, 1) {};
		\node [style=none] (5) at (-5, 0) {};
		\node [style=none] (6) at (-5, -2) {};
		\node [style=none] (7) at (-3, -2) {};
		\node [style=none] (12) at (-5, 3) {};
		\node [style=none] (13) at (-1.25, 2.12) {};
		\node [style=none] (14) at (-0.75, 1.88) {};
		\node [style=none] (15) at (-4.25, -0.88) {};
		\node [style=none] (16) at (-3.75, -1.13) {};
		\node [style=none] (23) at (0, -2) {};
		\node [style=none] (24) at (-1.5, -2.5) {S};
		\node [style=none] (24) at (-3.5, 5) {N};
		\node [style=none] (24) at (1, 0.85) {E};
		\node [style=none] (24) at (-4, -2.5) {W};
	\end{pgfonlayer}
	\begin{pgfonlayer}{edgelayer}
		\draw [thick] [out=270, in=90] (0.center) to (3.center);
		\draw [ultra thick, color=white] (13.center) to (14.center);
		\draw [thick] (12.center) to (5.center);
		\draw [thick] [out=270, in=90] (5.center) to (7.center);
		\draw [ultra thick, color=white] (15.center) to (16.center);
		\draw [thick] [out=90,in=90, looseness=1.5] (0.center) to (12.center);
		\draw [thick] (1.center) to (6.center);
		\draw [thick] (3.center) to (23.center);
	\end{pgfonlayer}
\end{tikzpicture}\; , \hspace{0.5in} 
\begin{tikzpicture}[baseline=0ex,scale=.5, color=\clr ]
	\begin{pgfonlayer}{nodelayer}
		\node [style=none] (0) at (1, 1) {};
		\node [style=none] (1) at (-3, -1) {};
		\node [style=none] (2) at (1, 1) {};
		\node [style=none] (3) at (-1, 1) {};
		\node [style=none] (6) at (1, 1) {};
		\node [style=none] (7) at (3, -1) {};
		\node [style=none] (8) at (1.75, 0.25) {};
		\node [style=none] (9) at (3, 1) {};
		\node [style=none] (10) at (-1, -3) {};
		\node [style=none] (11) at (2.25, -0.25) {};
		\node [style=none] (12) at (-1, -1) {};
		\node [style=none] (13) at (-0.25, -1.75) {};
		\node [style=none] (14) at (1, -3) {};
		\node [style=none] (15) at (-1.75, -0.25) {};
		\node [style=none] (16) at (0.25, -2.25) {};
		\node [style=none] (17) at (-3, 1) {};
		\node [style=none] (18) at (-1, -1) {};
		\node [style=none] (19) at (-2.25, 0.25) {};
		\node [style=none] (20) at (0, 3) {};
		\node [style=none] (21) at (0, 2.25) {N};
		\node [style=none] (22) at (-2, 2) {};
		\node [style=none] (23) at (2, 2) {};
		\node [style=none] (24) at (-3, 0) {W};
		\node [style=none] (25) at (3, 0) {E};
		\node [style=none] (26) at (2, -2) {S};
		\node [style=none] (27) at (0, -3) {S};
		\node [style=none] (28) at (-2, -2) {S};
	\end{pgfonlayer}
	\begin{pgfonlayer}{edgelayer}
		\draw [thick, looseness=1.5] [out=45,in=135] (3.center) to (6.center);
		\draw [thick] (14.center) to (17.center);
		\draw [ultra thick, color=white] (19.center) to (15.center);
		\draw [ultra thick, color=white] (13.center) to (16.center);
		\draw [thick] (1.center) to (3.center);
		\draw [thick] (0.center) to (2.center);
		\draw [thick] (7.center) to (6.center);
		\draw [ultra thick, color=white] (8.center) to (11.center);
		\draw [thick] (10.center) to (9.center);
	\end{pgfonlayer}
\end{tikzpicture}\; , \hspace{0.5in} \text{and} \hspace{0.5in}
\begin{tikzpicture}[baseline=0ex,scale=.5, color=\clr ]
	\begin{pgfonlayer}{nodelayer}
		\node [style=none] (0) at (1, 1) {};
		\node [style=none] (1) at (-3, -1) {};
		\node [style=none] (2) at (-1, 3) {};
		\node [style=none] (3) at (1, 3) {};
		\node [style=none] (4) at (-0.25, 2.25) {};
		\node [style=none] (5) at (0.25, 1.75) {};
		\node [style=none] (6) at (1, 1) {};
		\node [style=none] (7) at (3, -1) {};
		\node [style=none] (8) at (1.75, 0.25) {};
		\node [style=none] (9) at (3, 1) {};
		\node [style=none] (10) at (-1, -3) {};
		\node [style=none] (11) at (2.25, -0.25) {};
		\node [style=none] (12) at (-1, -1) {};
		\node [style=none] (13) at (-0.25, -1.75) {};
		\node [style=none] (14) at (1, -3) {};
		\node [style=none] (15) at (-1.75, -0.25) {};
		\node [style=none] (16) at (0.25, -2.25) {};
		\node [style=none] (17) at (-3, 1) {};
		\node [style=none] (18) at (-1, -1) {};
		\node [style=none] (19) at (-2.25, 0.25) {};
		\node [style=none] (20) at (0, 3) {N};
		\node [style=none] (21) at (0, 3) {N};
		\node [style=none] (22) at (-2, 2) {N};
		\node [style=none] (23) at (2, 2) {N};
		\node [style=none] (24) at (-3, 0) {W};
		\node [style=none] (25) at (3, 0) {E};
		\node [style=none] (26) at (2, -2) {S};
		\node [style=none] (27) at (0, -3) {S};
		\node [style=none] (28) at (-2, -2) {S};
	\end{pgfonlayer}
	\begin{pgfonlayer}{edgelayer}
		\draw [thick] (7.center) to (2.center);
		\draw [ultra thick, color=white] (4.center) to (5.center);
		\draw [ultra thick, color=white] (8.center) to (11.center);
		\draw [thick] (14.center) to (17.center);
		\draw [ultra thick, color=white] (19.center) to (15.center);
		\draw [ultra thick, color=white] (13.center) to (16.center);
		\draw [thick] (1.center) to (3.center);
		\draw [thick] (10.center) to (9.center);
	\end{pgfonlayer}
\end{tikzpicture}\; .
\end{equation}
The labels N, S, E, W are for later reference and can be ignored for now.

\section{Jimbo--Schur--Weyl Duality and Tubbenhauer--Vaz--Wedrich Diagrams}\label{S:SWDualityandTVWWebs}

\subsection{Jimbo--Schur--Weyl duality}  Since $W^{(n+r,r)}$ is a right $\HH_{d}(q)$-module with the same dimension as $S^{(n+r,r)}$, it will come as little surprise that they are isomorphic.  Proving this claim is the goal of this section.  First we record some well-known results (e.g., see \cite{CP, Kassel}).

Let $U_{q}(\gl_{2})$ be the quantized enveloping algebra defined over $\k$ which is associated to $\gl_{2}(\mathbb{C})$.  Let $V = \k^{2}$ be the natural module for $U_{q}(\gl_{2})$ with $\{ v_{0}, v_{1}\}$ the standard basis.  Define the linear map $\beta: V \otimes V \to V \otimes V$ by 
\[
\beta(v_{a} \otimes v_{b}) = \begin{cases} q v_{a} \otimes v_{b}, & \text{if $a=b$};\\
                                            v_{b}\otimes v_{a}, & \text{if $a <b$};\\
                                            v_{b} \otimes v_{a} + (q-q^{-1})v_{a}\otimes v_{b}, & \text{if $a > b$}.
\end{cases}
\]  This is the $U_{q}(\gl_{2})$-linear endomorphism defined by the $R$-matrix (e.g., see \cite[VIII.1]{Kassel}) and satisfies the equation $(\beta-q)(\beta+q^{-1})=0$.  Given $\lambda \in \Lambda_{+}(2)$, let $L_{q}(\lambda)$ be the simple $U_{q}(\gl_{2})$-module indexed by $\lambda$.  Let $S_{q}^{a}(V)$ (resp., $\Lambda_{q}^{a}(V)$) be the $a$th symmetric power (resp., exterior power)  of $V$ for any $a \geq 0$. For example, $S_{q}^{2}(V) \cong L_{q}(2,0)$ and can be realized as the $q$-eigenspace of the endomorphism $\beta$.

Since $U_{q}(\gl_{2})$ is a Hopf algebra, $V^{\otimes d}$ is a left $U_{q}(\gl_{2})$-module.  Jimbo--Schur--Weyl duality states that for every $d \geq 0$ there is a surjective algebra homomorphism 
\[
\Phi_{d}: \HH_{d}(q) \to \End_{U_{q}(\gl_{2})}\left(V^{\otimes d} \right)
\] defined by $\Phi_{d}(T_{i}) = 1_{V}^{\otimes (i-1)} \otimes \beta   \otimes 1_{V}^{\otimes (d-i-1)}$. Consequently, for any left $U_{q}(\gl_{2})$-module, $M$, $\Hom_{U_{q}(\gl_{2})}(V^{\otimes d}, M)$ is a right $\HH_{d}(q)$-module with the action of $x \in \HH_{d}(q)$ given by pre-composition by $\Phi_{d}(x)$.   In particular, as explained in \cite[Remark 2.1(ii), Proposition 4.5.8]{Donkin}, for every $\lambda \in \Lambda_{+}(2,d)$ there is an isomorphism of right $\HH_{d}(q)$-modules:
\begin{equation}\label{E:DonkinIsom}
\Hom_{U_{q}(\gl_{2})}\left(V^{\otimes d}, L_{q}(\lambda) \right) \cong S^{\lambda}.
\end{equation}

\subsection{Tubbenhauer--Vaz--Wedrich diagrams}\label{SS:TVW-webs}

Let $\rgWeb$ be the green-red diagrammatic category for $U_{q}(\mathfrak{gl}_{2})$ given by \cite[Definition 2.5]{TVW}.  Briefly, this is a $\k$-linear monoidal category which is defined by generators and relations.  The objects are words from the set $\left\{\textcolor{red}{a}, \textcolor{green}{b}, 1 \mid a,b \in \mathbb{Z}_{\geq 1} \right\}$.  The morphisms are certain diagrams generated by the following \textcolor{green}{green} and \textcolor{red}{red} ``merge'' morphisms $ab \to a+b$  and ``split'' morphisms $a+b \to ab$  (for all $a,b \in \Z_{\geq 0}$):

\begin{equation*}
\begin{tikzpicture}[baseline=0ex,scale=.65, color=\clr ]
	\begin{pgfonlayer}{nodelayer}
		\node [style=none] (0) at (-2, -1) {};
		\node [style=none] (10) at (-1.5, 0) {};
		\node [style=none] (11) at (-1.5, 1) {};
		\node [style=none] (12) at (-1, -1) {};
		\node [style=none] (20) at (-1.5, 1.5) {\tiny{$a+b$}};
		\node [style=none] (24) at (-2, -1.5) {\tiny{$a$}};
		\node [style=none] (25) at (-1, -1.5) {\tiny{$b$}};
	\end{pgfonlayer}
	\begin{pgfonlayer}{edgelayer}
		\draw [thick, dashed, green] [out=90,in=0] (12.center) to (10.center);
		\draw [thick, dashed, green] [out=90,in=180] (0.center) to (10.center);
		\draw [thick, dashed, green] (10.center) to (11.center);
	\end{pgfonlayer}
\end{tikzpicture}\; \; \; \; , \; \; \; \; 
\begin{tikzpicture}[baseline=0ex,scale=.65, color=\clr ]
	\begin{pgfonlayer}{nodelayer}
		\node [style=none] (0) at (-2, -1) {};
		\node [style=none] (10) at (-1.5, 0) {};
		\node [style=none] (11) at (-1.5, 1) {};
		\node [style=none] (12) at (-1, -1) {};
		\node [style=none] (20) at (-1.5, 1.5) {\tiny{$a+b$}};
		\node [style=none] (24) at (-2, -1.5) {\tiny{$a$}};
		\node [style=none] (25) at (-1, -1.5) {\tiny{$b$}};
	\end{pgfonlayer}
	\begin{pgfonlayer}{edgelayer}
		\draw [thick, densely dashdotted, red]  [out=90,in=0] (12.center) to (10.center);
		\draw [thick, densely dashdotted, red]  [out=90,in=180] (0.center) to (10.center);
		\draw [thick, densely dashdotted, red]  (10.center) to (11.center);
	\end{pgfonlayer}
\end{tikzpicture}\; \; \; \; , \; \; \; \; 
\begin{tikzpicture}[baseline=0ex,scale=.65, color=\clr ]
	\begin{pgfonlayer}{nodelayer}
		\node [style=none] (0) at (-2, 1) {};
		\node [style=none] (10) at (-1.5, 0) {};
		\node [style=none] (11) at (-1.5, -1) {};
		\node [style=none] (12) at (-1, 1) {};
		\node [style=none] (20) at (-1.5, -1.5) {\tiny{$a+b$}};
		\node [style=none] (24) at (-2, 1.5) {\tiny{$a$}};
		\node [style=none] (25) at (-1, 1.5) {\tiny{$b$}};
	\end{pgfonlayer}
	\begin{pgfonlayer}{edgelayer}
		\draw [thick, dashed, green] [out=270,in=0] (12.center) to (10.center);
		\draw [thick, dashed, green] [out=270,in=180] (0.center) to (10.center);
		\draw [thick, dashed, green] (10.center) to (11.center);
	\end{pgfonlayer}
\end{tikzpicture}\; \; \; \; , \; \; \; \; 
\begin{tikzpicture}[baseline=0ex,scale=.65, color=\clr ]
	\begin{pgfonlayer}{nodelayer}
		\node [style=none] (0) at (-2, 1) {};
		\node [style=none] (10) at (-1.5, 0) {};
		\node [style=none] (11) at (-1.5, -1) {};
		\node [style=none] (12) at (-1, 1) {};
		\node [style=none] (20) at (-1.5, -1.5) {\tiny{$a+b$}};
		\node [style=none] (24) at (-2, 1.5) {\tiny{$a$}};
		\node [style=none] (25) at (-1, 1.5) {\tiny{$b$}};
	\end{pgfonlayer}
	\begin{pgfonlayer}{edgelayer}
		\draw [thick, densely dashdotted, red]  [out=270,in=0] (12.center) to (10.center);
		\draw [thick, densely dashdotted, red]  [out=270,in=180] (0.center) to (10.center);
		\draw [thick, densely dashdotted, red]  (10.center) to (11.center);
	\end{pgfonlayer}
\end{tikzpicture} \; \; .
\end{equation*}
Observe that we use dashed lines ($\begin{tikzpicture}[baseline=-0.5ex,scale=.65, color=\clr] 		\node [style=none] (10) at (-0.5, 0) {};
		\node [style=none] (11) at (0.5, 0) {}; 		\draw [thick, dashed, green]  (10.center) to (11.center);
 \end{tikzpicture}$) for green diagrams and densely dashdotted lines ($\begin{tikzpicture}[baseline=-0.5ex,scale=.65, color=\clr] 		\node [style=none] (10) at (-0.5, 0) {};
		\node [style=none] (11) at (0.5, 0) {}; 		\draw [thick, densely dashdotted, red]  (10.center) to (11.center);
 \end{tikzpicture}$) for red diagrams to help distinguish them.
 
Also note that diagrams are read bottom-to-top.  The identity of an object is parallel vertical strands sequentially labeled by the letters which make up the object word.  We declare $\textcolor{red}{1}=1=\textcolor{green}{1}$.  While edges in a diagram labeled by a $1$ are usually drawn as black, they can be considered either red or green as needed in order to make a valid diagram.  As above, edges in a diagram are labeled by the generating object which is the (co)domain of the relevant generating morphism. Unmarked edges are understood to have label $1$.  Composition is given by vertical concatenation and the monoidal product is given by horizontal concatenation.  Therefore, a general morphism will be a $\k$-linear combination of diagrams made by finitely many such concatenations of the generating diagrams.  The morphisms are subject to various local relations, including that any diagram which contains a green edge labeled by $\textcolor{green}{k} \geq 3$ is equal to zero.

For example, for each $i = 1, \dotsc , d-1$, there is a morphism $\beta_{i}: 1^{\otimes d} \to 1^{\otimes d}$ in $\rgWeb$ given by the following linear combination of diagrams:
\begin{equation*}
\beta_{i} = \beta_{i}^{d} = 
q \; \begin{tikzpicture}[baseline=0ex,scale=.65, color=\clr  ]
	\begin{pgfonlayer}{nodelayer}
		\node [style=none] (0) at (-3.5, -1) {};
		\node [style=none] (1) at (-5, 1) {};
		\node [style=none] (2) at (-5, -1) {};
		\node [style=none] (3) at (-2.84, 1) {};
		\node [style=none] (4) at (0, 1) {};
		\node [style=none] (5) at (0, -1.5) {\tiny{$1$}};
		\node [style=none] (6) at (-1.5, -1.5) {\tiny{$1$}};
		\node [style=none] (7) at (-2.84, -1.5) {\tiny{$1$}};
		\node [style=none] (8) at (-2.16, 1) {};
		\node [style=none] (9) at (-5, -1.5) {\tiny{$1$}};
		\node [style=none] (10) at (-3.5, -1.5) {\tiny{$1$}};
		\node [style=none] (11) at (-2.16, -1.5) {\tiny{$1$}};
		\node [style=none] (12) at (-1.5, -1) {};
		\node [style=none] (13) at (0, -1) {};
		\node [style=none] (14) at (-1.5, 1) {};
		\node [style=none] (15) at (-3.5, 1) {};
		\node [style=none] (16) at (-2.84, -1) {};
		\node [style=none] (17) at (-2.16, -1) {};
		\node [style=none] (18) at (-4.25, 0) {\tiny{\dots}};
		\node [style=none] (19) at (-0.75, 0) {\tiny{\dots}};
		\node [style=none] (20) at (-2.5, 0.55) {};
		\node [style=none] (21) at (-2.5, -0.55) {};
	\end{pgfonlayer}
	\begin{pgfonlayer}{edgelayer}
		\draw [thick] (3.center) to (16.center);
		\draw [thick](2.center) to (1.center);
		\draw [thick] (8.center) to (17.center);
		\draw [thick] (14.center) to (12.center);
		\draw [thick] (4.center) to (13.center);
		\draw [thick] (15.center) to (0.center);
	\end{pgfonlayer}
\end{tikzpicture}
\; \; - \; \; 
\begin{tikzpicture}[baseline=0ex,scale=.65, color=\clr ]
	\begin{pgfonlayer}{nodelayer}
		\node [style=none] (0) at (-3.5, -1) {};
		\node [style=none] (1) at (-5, 1) {};
		\node [style=none] (2) at (-5, -1) {};
		\node [style=none] (3) at (-3, 1) {};
		\node [style=none] (4) at (0, 1) {};
		\node [style=none] (5) at (0, -1.5) {\tiny{$1$}};
		\node [style=none] (6) at (-1.5, -1.5)  {\tiny{$1$}};
		\node [style=none] (7) at (-3, -1.5) {\tiny{$1$}};
		\node [style=none] (8) at (-2, 1) {};
		\node [style=none] (9) at (-5, -1.5) {\tiny{$1$}};
		\node [style=none] (10) at (-3.5, -1.5) {\tiny{$1$}};
		\node [style=none] (11) at (-2, -1.5) {\tiny{$1$}};
		\node [style=none] (12) at (-1.5, -1) {};
		\node [style=none] (13) at (0, -1) {};
		\node [style=none] (14) at (-1.5, 1) {};
		\node [style=none] (15) at (-3.5, 1) {};
		\node [style=none] (16) at (-3, -1) {};
		\node [style=none] (17) at (-2, -1) {};
		\node [style=none] (18) at (-4.25, 0) {\tiny{\dots}};
		\node [style=none] (19) at (-0.75, 0) {\tiny{\dots}};
		\node [style=none] (20) at (-2.5, 0.55) {};
		\node [style=none] (21) at (-2.5, -0.55) {};
		\node [style=none] (22) at (-2.25, 0) {\tiny{\textcolor{green}{2}}};

	\end{pgfonlayer}
	\begin{pgfonlayer}{edgelayer}
		\draw [thick, looseness=1.5, dashed, green] [out=270, in=270] (3.center) to (8.center);
		\draw [thick] (2.center) to (1.center);
		\draw [thick, looseness=1.5, dashed, green] [out=90, in=90] (16.center) to (17.center);
		\draw [thick] (14.center) to (12.center);
		\draw [thick] (4.center) to (13.center);
		\draw [thick] (15.center) to (0.center);
		\draw [thick, dashed, green] (20.center) to (21.center);
	\end{pgfonlayer}
\end{tikzpicture}\; ,
\end{equation*} where the first diagram has $d$ black strands, and the second diagram has $d-2$ black strands and the green merge and split are in the $i$ and $i+1$ positions.

Recall the generators and relations defining the Hecke algebra $\HH_{d}(q)$ as given in \cref{R:Heckeisom}. By \cite[Lemma 2.25]{TVW}, there is an algebra homomorphism 
\begin{align}\label{D:zetahomom}
\zeta_{d}: \HH_{d}(q) &\to \End_{\rgWeb}\left(1^{d} \right),\\
T_{i} &\mapsto \beta_{i}.
\end{align}
In particular, there is a right action of $\HH_{d}(q)$ on the $\k$-vector space $\Hom_{\rgWeb}\left(1^{d},\textcolor{green}{2}^{n} \textcolor{red}{r} \right)$ for any $r,n \in \Z_{\geq 0}$ where $T_{i}$ acts by pre-composition by $\zeta_{d}(T_{i})$.

In \cite[Definition 3.17]{TVW} the authors define a functor of monoidal categories, 
\[
\Gamma: \rgWeb \to U_{q}(\gl_{2})\text{-modules}.
\]
  On objects $1 \mapsto V$, $\textcolor{red}{a}\mapsto S^{a}_{q}(V)$, and $\textcolor{green}{b} \mapsto \Lambda_{q}^{b}(V)$.  Of particular relevance to this paper, 
\[
1^{d} \mapsto V^{\otimes d}
\]
and
\[
\textcolor{green}{2}^{n}\textcolor{red}{r} \mapsto  \Lambda_{q}^{2}(V)^{\otimes n} \otimes S^{r}_{q}(V) \cong  L_{q}(n+r,n).
\]
The isomorphism is well-known and holds, for instance, from the fact both modules are semisimple polynomial representations of $U_{q}(\gl_{2})$ along with character considerations.

By  \cite[Theorem 3.20]{TVW}, the functor $\Gamma$ is fully faithful and so for all $d,n,r \geq 0$ defines a vector space isomorphism,
\begin{equation}\label{E:GammaIsom}
\Gamma:\Hom_{\rgWeb}\left(1^{d}, \textcolor{green}{2}^{n} \textcolor{red}{r} \right) \xrightarrow{\cong} \Hom_{U_{q}(\gl_{2})}\left(V^{\otimes d}, L_{q}(n+r,n)  \right). 
\end{equation} Furthermore, as explained in the proof of \cite[Lemma 2.25]{TVW}, the map given by $\Gamma$ is an isomorphism of right $\HH_{d}(q)$-modules (where the action of $\HH_{d}(q)$ on the latter space is given by Jimbo--Schur--Weyl duality as in the previous section).  In summary, if $d=2n+r$, then this discussion along with \cref{E:DonkinIsom} implies
\begin{equation}\label{E:TVWWebSpechtIsom}
\Hom_{\rgWeb}\left(1^{d}, \textcolor{green}{2}^{n} \textcolor{red}{r} \right) \cong S^{(n+r,n)}
\end{equation}
as right $\HH_{d}(q)$-modules.

\begin{prop}\label{T:TVW}  For all nonnegative integers $n$ and $r$, there is an isomorphism of right $\HH_{2n+r}(q)$-modules
\[
\varphi: S^{(n+r,n)} \xrightarrow{\cong} W^{(n+r,n)}.
\]  Furthermore, by rescaling if necessary, $\varphi(v_{0}) = w_{0}$.
\end{prop}  Here $v_{0}=v_{t_{0}}=z_{\lambda}$ is as in \cref{E:v0def} and $w_{0}$ is the $(n+r,n)$-web given in \cref{E:w0definition}.
We refer the reader to \cref{SS:ExampleII} and \cref{SS:ExampleIII} for examples of the matrix for $\varphi$ when $\lambda = (3,3)$ and $\lambda = (4,2)$, respectively.

\begin{proof} Put $\lambda=(n+r,n)$ and $d=2n+r$.  First we prove the existence of the isomorphism.  Combining \cref{E:TVWWebSpechtIsom,L:WebTableauxBijection} along with fact that $S^{\lambda}$ is a simple $\HH_{d}(q)$-module, it suffices to define a nonzero $\HH_{d}(q)$-module homomorphism 
\[
\eta: \Hom_{\rgWeb}\left(1^{d}, \textcolor{green}{2}^{n} \textcolor{red}{r}\right) \to W^{(n+r,n)}.
\] Consider the subspace of $\Hom_{\rgWeb}\left(1^{d}, \textcolor{green}{2}^{n} \textcolor{red}{r}\right)$ spanned by all diagrams of the form:
\begin{equation}\label{E:Diagram} D \;\; =\;\;
\begin{tikzpicture}[baseline=0ex,scale=.5, color=\clr ]
	\begin{pgfonlayer}{nodelayer}
		\node [style=none] (0) at (-3, 2) {};
		\node [style=none] (1) at (-3, 0) {};
		\node [style=none] (2) at (4, 2) {};
		\node [style=none] (3) at (4, 0) {};
		\node [style=none] (4) at (4, -1) {};
		\node [style=none] (5) at (3, 0) {};
		\node [style=none] (6) at (3, -1) {};
		\node [style=none] (7) at (-2, 0) {};
		\node [style=none] (8) at (-2, -1) {};
		\node [style=none] (9) at (-3, -1) {};
		\node [style=none] (10) at (-2.5, 3) {};
		\node [style=none] (11) at (-2.5, 4) {};
		\node [style=none] (12) at (-2, 2) {};
		\node [style=none] (13) at (0, 2) {};
		\node [style=none] (14) at (0.5, 3) {};
		\node [style=none] (15) at (0.5, 4) {};
		\node [style=none] (16) at (1, 2) {};
		\node [style=none] (17) at (2, 2) {};
		\node [style=none] (18) at (3, 3) {};
		\node [style=none] (19) at (3, 4) {};
		\node [style=none] (20) at (-2.5, 4.5) {\textcolor{green}{\tiny{2}}};
		\node [style=none] (21) at (0.5, 4.5) {\textcolor{green}{\tiny{2}}};
		\node [style=none] (22) at (3, 4.5) {\textcolor{red}{\tiny{r}}};
		\node [style=none] (23) at (-3, -1.5) {\tiny{1}};
		\node [style=none] (24) at (-2, -1.5) {\tiny{1}};
		\node [style=none] (25) at (3, -1.5) {\tiny{1}};
		\node [style=none] (26) at (4, -1.5) {\tiny{1}};
		\node [style=none] (27) at (0.5, -0.5) {};
		\node [style=none] (28) at (0.5, -0.5) {};
		\node [style=none] (29) at (0.5, -0.75) {\tiny{\dots}};
		\node [style=none] (30) at (-1, 3) {\tiny{\dots}};
		\node [style=none] (31) at (0.5, 1) {$X$};
		\node [style=none] (32) at (3, 2.5) {\tiny{\dots}};
	\end{pgfonlayer}
	\begin{pgfonlayer}{edgelayer}
		\draw [style=TB Edge] (0.center) to (1.center);
		\draw [style=TB Edge] (1.center) to (3.center);
		\draw [style=TB Edge] (3.center) to (2.center);
		\draw [style=TB Edge] (2.center) to (0.center);
		\draw [style=TB Edge] (1.center) to (9.center);
		\draw [style=TB Edge] (7.center) to (8.center);
		\draw [style=TB Edge] (5.center) to (6.center);
		\draw [style=TB Edge] (3.center) to (4.center);
		\draw [thick, dashed, green] [out=90,in=0] (12.center) to (10.center);
		\draw [thick, dashed, green] [out=90, in=180] (0.center) to (10.center);
		\draw [thick, dashed, green] [out=90, in=0] (16.center) to (14.center);
		\draw [thick, dashed, green] [out=90,in=180] (13.center) to (14.center);
		\draw [thick, dashed, green] (14.center) to (15.center);
		\draw [thick, dashed, green] (10.center) to (11.center);
		\draw [thick, densely dashdotted, densely dashdotted, red]  [out=90,in=180] (17.center) to (18.center);
		\draw [thick, densely dashdotted, densely dashdotted, red]  [out=90,in=0] (2.center) to (18.center);
		\draw [thick, densely dashdotted, densely dashdotted, red]  (18.center) to (19.center);
	\end{pgfonlayer}
\end{tikzpicture}\; ,
\end{equation}
where $X$ is any diagram in $\End_{\rgWeb}(1^{\otimes d})$ obtained by vertically and horizontally concatenating green merges, splits, and identities (recall that an edge labeled by $1$ is considered to be green), and where the red diagram is $r$ strands labeled by $1$ that are successively merged together from left-to-right by concatenating red merges.  Note that since green edges labeled by $\textcolor{green}{k} \geq 3$ do not appear in $\rgWeb$, this subspace is spanned by diagrams of this form where the $X$ consists of only green edges labeled by $1$'s and $\textcolor{green}{2}$'s.

Since this is a nonzero subspace of $\Hom_{\rgWeb}\left(1^{d}, \textcolor{green}{2}^{n} \textcolor{red}{r}\right)$ and is obviously stable under the right action of $\HH_{d}(q)$, it must equal $\Hom_{\rgWeb}\left(1^{d}, \textcolor{green}{2}^{n} \textcolor{red}{r}\right)$ by simplicity.  Thus it suffices to define $\eta$ on such diagrams. On a diagram $D$ as described above, $\eta(D)$ is defined by performing the following sequence of operations on $D$:
\begin{itemize}
\item Delete the red merges.
\item Make the following local substitutions for the green merge and split diagrams in $X$:
\begin{equation*}
\begin{tikzpicture}[baseline=0ex,scale=.65, color=\clr ]
	\begin{pgfonlayer}{nodelayer}
		\node [style=none] (0) at (-2, -1) {};
		\node [style=none] (10) at (-1.5, 0) {};
		\node [style=none] (11) at (-1.5, 1) {};
		\node [style=none] (12) at (-1, -1) {};
		\node [style=none] (13) at (1, -0.25) {};
		\node [style=none] (16) at (2, -0.25) {};
		\node [style=none] (20) at (-1.5, 1.5) {\textcolor{green}{\tiny{$2$}}};
		\node [style=none] (24) at (-2, -1.5) {\tiny{$1$}};
		\node [style=none] (25) at (-1, -1.5) {\tiny{$1$}};
		\node [style=none] (26) at (0, 0) {$\leadsto$};
	\end{pgfonlayer}
	\begin{pgfonlayer}{edgelayer}
		\draw [thick, dashed, green] [out=90, in=0] (12.center) to (10.center);
		\draw [thick, dashed, green] [out=90,in=180] (0.center) to (10.center);
		\draw [thick, dashed, green] (10.center) to (11.center);
		\draw [thick, looseness=1.5] [out=90,in=90] (13.center) to (16.center);
	\end{pgfonlayer}
\end{tikzpicture}\; \; \; \; \text{ and } \; \; \; \; 
\begin{tikzpicture}[baseline=0ex,scale=.65, color=\clr ]
	\begin{pgfonlayer}{nodelayer}
		\node [style=none] (0) at (-2, 1) {};
		\node [style=none] (10) at (-1.5, 0) {};
		\node [style=none] (11) at (-1.5, -1) {};
		\node [style=none] (12) at (-1, 1) {};
		\node [style=none] (13) at (1.5, 0.25) {};
		\node [style=none] (16) at (2.5, 0.25) {};
		\node [style=none] (20) at (-1.5, -1.5) {\textcolor{green}{\tiny{$2$}}};
		\node [style=none] (24) at (-2, 1.5) {\tiny{$1$}};
		\node [style=none] (25) at (-1, 1.5) {\tiny{$1$}};
		\node [style=none] (26) at (0, 0) {$\leadsto$};
		\node [style=none] (27) at (1, 0) {$-$};
	\end{pgfonlayer}
	\begin{pgfonlayer}{edgelayer}
		\draw [thick, dashed, green] [out=270, in=0] (12.center) to (10.center);
		\draw [thick, dashed, green] [out=270,in=180] (0.center) to (10.center);
		\draw [thick, dashed, green] (10.center) to (11.center);
		\draw [thick, looseness=1.5] [out=270,in=270] (13.center) to (16.center);
	\end{pgfonlayer}
\end{tikzpicture}\;.
\end{equation*}
\item Turn all edges black and delete all edge labels. 
\item For each connected component which is not connected to vertex along the top or bottom edge, delete the component and scale the diagram by $-[2]_{q}$.
\end{itemize}

The resulting diagram will be a scalar multiple of an $(n+r,n)$-web.  For example, $\eta$ acts on the following diagram as indicated:
\begin{equation}\label{E:etaexample}
\begin{tikzpicture}[baseline=-2ex,scale=.4, color=\clr ]
	\begin{pgfonlayer}{nodelayer}
		\node [style=none] (0) at (3.75, 2) {};
		\node [style=none] (1) at (3.75, -1) {};
		\node [style=none] (3) at (-2, -1) {};
		\node [style=none] (4) at (-2, -3.5) {};
		\node [style=none] (5) at (3, -1) {};
		\node [style=none] (6) at (3, -3.5) {};
		\node [style=none] (7) at (4.75, -1) {};
		\node [style=none] (8) at (4.75, -2) {};
		\node [style=none] (9) at (3.75, -2) {};
		\node [style=none] (10) at (-2.5, 3) {};
		\node [style=none] (11) at (-2.5, 5) {};
		\node [style=none] (12) at (4.75, 2) {};
		\node [style=none] (13) at (1, 2) {};
		\node [style=none] (14) at (1.5, 3) {};
		\node [style=none] (15) at (1.5, 5) {};
		\node [style=none] (16) at (2, 2) {};
		\node [style=none] (17) at (3, 2) {};
		\node [style=none] (18) at (3.4, 2.75) {};
		\node [style=none] (19) at (4, 5) {};
		\node [style=none] (20) at (-2.5, 5.5) {\textcolor{green}{\tiny{$2$}}};
		\node [style=none] (21) at (1.5, 5.5) {\textcolor{green}{\tiny{$2$}}};
		\node [style=none] (22) at (4.75, 5.5) {\textcolor{red}{\tiny{$4$}}};
		\node [style=none] (23) at (3.75, -4) {\tiny{$1$}};
		\node [style=none] (24) at (4.75, -4) {\tiny{$1$}};
		\node [style=none] (25) at (3, -4) {\tiny{$1$}};
		\node [style=none] (26) at (-2, -4) {\tiny{$1$}};
		\node [style=none] (27) at (-0.5, 5.5) {\textcolor{green}{\tiny{$2$}}};
		\node [style=none] (28) at (-0.5, 5) {};
		\node [style=none] (29) at (0, 2) {};
		\node [style=none] (30) at (-1, 2) {};
		\node [style=none] (31) at (-0.5, 3) {};
		\node [style=none] (33) at (4.5, 2.75) {};
		\node [style=none] (34) at (4, 3.25) {};
		\node [style=none] (35) at (-1, -3.5) {};
		\node [style=none] (36) at (-1, -1) {};
		\node [style=none] (37) at (-1, -4) {\tiny{$1$}};
		\node [style=none] (38) at (0, -3.5) {};
		\node [style=none] (39) at (0, -1) {};
		\node [style=none] (40) at (0, -4) {\tiny{$1$}};
		\node [style=none] (41) at (1, -3.5) {};
		\node [style=none] (42) at (1, -1) {};
		\node [style=none] (43) at (1, -4) {\tiny{$1$}};
		\node [style=none] (44) at (2, -3.5) {};
		\node [style=none] (45) at (2, -1) {};
		\node [style=none] (46) at (2, -4) {\tiny{$1$}};
		\node [style=none] (47) at (-3, -3.5) {};
		\node [style=none] (48) at (-3, -1) {};
		\node [style=none] (49) at (-3, -4) {\tiny{$1$}};
		\node [style=none] (51) at (1.5, 0) {};
		\node [style=none] (52) at (1.5, -0.5) {};
		\node [style=none] (54) at (0.5, 0.75) {};
		\node [style=none] (55) at (2.5, 0.75) {};
		\node [style=none] (57) at (0.5, 1.25) {};
		\node [style=none] (58) at (2.5, 1.25) {};
		\node [style=none] (59) at (-3, -1) {};
		\node [style=none] (60) at (-3, 0.5) {};
		\node [style=none] (61) at (-2, 0.5) {};
		\node [style=none] (62) at (-2, -1) {};
		\node [style=none] (63) at (-2.5, -0.5) {};
		\node [style=none] (64) at (-2.5, 0) {};
		\node [style=none] (65) at (-3, 0.5) {};
		\node [style=none] (66) at (-3, 2) {};
		\node [style=none] (67) at (-2, 2) {};
		\node [style=none] (68) at (-2, 0.5) {};
		\node [style=none] (69) at (-2.5, 1) {};
		\node [style=none] (70) at (-2.5, 1.5) {};
		\node [style=none] (95) at (2.75, 1) {\textcolor{green}{\tiny{$2$}}};
		\node [style=none] (96) at (0.75, 1) {\textcolor{green}{\tiny{$2$}}};
		\node [style=none] (97) at (-2.25, 1.25) {\textcolor{green}{\tiny{$2$}}};
		\node [style=none] (97) at (-2.25, -0.25) {\textcolor{green}{\tiny{$2$}}};
		\node [style=none] (97) at (1.75, -0.25) {\textcolor{green}{\tiny{$2$}}};
		\node [style=none] (97) at (3.35, 3.35) {\textcolor{red}{\tiny{$2$}}};
		\node [style=none] (144) at (5.75, -3.5) {};
		\node [style=none] (145) at (5.75, 2) {};
		\node [style=none] (146) at (5.75, -4) {\tiny{$1$}};
		\node [style=none] (119) at (4.75, 5) {};
		\node [style=none] (134) at (4.75, 4.25) {};
		\node [style=none] (97) at (3.85, 4) {\textcolor{red}{\tiny{$3$}}};
		\node [style=none] (108) at (4.75, -3.5) {};
		\node [style=none] (109) at (3.75, -3.5) {};
		\node [style=none] (110) at (3.375, -3) {};
		\node [style=none] (111) at (3.375, -2.5) {};
		\node [style=none] (197) at (3.75, -2.75) {\textcolor{green}{\tiny{$2$}}};
		\node [style=none] (105) at (3, -2) {};
	\end{pgfonlayer}
	\begin{pgfonlayer}{edgelayer}
		\draw [thick, dashed, green] (0.center) to (9.center);
		\draw [thick, dashed, green] (5.center) to (105.center);
		\draw [thick, dashed, green] (3.center) to (4.center);
		\draw [thick, dashed, green] [out=90,in=0] (16.center) to (14.center);
		\draw [thick, dashed, green] [out=90,in=180] (13.center) to (14.center);
		\draw [thick, dashed, green] (14.center) to (15.center);
		\draw [thick, dashed, green] (10.center) to (11.center);
		\draw [thick, densely dashdotted, densely dashdotted, red]  [out=90,in=180] (17.center) to (18.center);
		\draw [thick, dashed, green] [out=90,in=0] (29.center) to (31.center);
		\draw [thick, dashed, green] [out=90,in=180] (30.center) to (31.center);
		\draw [thick, dashed, green] (31.center) to (28.center);
		\draw [thick, dashed, green] (36.center) to (35.center);
		\draw [thick, dashed, green] (39.center) to (38.center);
		\draw [thick, dashed, green] (42.center) to (41.center);
		\draw [thick, dashed, green] (45.center) to (44.center);
		\draw [thick, dashed, green] (48.center) to (47.center);
		\draw [thick, dashed, green] [out=180,in=90] (52.center) to (42.center);
		\draw [thick, dashed, green] [out=0,in=90] (52.center) to (45.center);
		\draw [thick, dashed, green] (52.center) to (51.center);
		\draw [thick, dashed, green] [out=180,in=0] (55.center) to (51.center);
		\draw [thick, dashed, green] [out=0,in=180] (54.center) to (51.center);
		\draw [thick, dashed, green] [out=0,in=90] (55.center) to (5.center);
		\draw [thick, dashed, green] [out=180,in=90] (54.center) to (39.center);
		\draw [thick, dashed, green]  (57.center) to (54.center);
		\draw [thick, dashed, green] [out=180,in=270] (57.center) to (29.center);
		\draw [thick, dashed, green] [out=0,in=270] (57.center) to (13.center);
		\draw [thick, dashed, green] (55.center) to (58.center);
		\draw [thick, dashed, green] [out=180,in=270] (58.center) to (16.center);
		\draw [thick, dashed, green] [out=0,in=270] (58.center) to (17.center);
		\draw [thick, dashed, green] [out=180,in=90] (63.center) to (59.center);
		\draw [thick, dashed, green] [out=0,in=90] (63.center) to (62.center);
		\draw [thick, dashed, green] [out=270,in=0] (61.center) to (64.center);
		\draw [thick, dashed, green] [out=270,in=180] (60.center) to (64.center);
		\draw [thick, dashed, green] (64.center) to (63.center);
		\draw [thick, dashed, green] [out=180,in=90] (69.center) to (65.center);
		\draw [thick, dashed, green] [out=0,in=90] (69.center) to (68.center);
		\draw [thick, dashed, green] [out=270,in=0] (67.center) to (70.center);
		\draw [thick, dashed, green] [out=270,in=180] (66.center) to (70.center);
		\draw [thick, dashed, green] (70.center) to (69.center);
		\draw [thick, densely dashdotted, red]  [out=90,in=180] (18.center) to (34.center);
		\draw [thick, dashed, green] [out=180,in=90] (10.center) to (66.center);
		\draw [thick, dashed, green] [out=0,in=90] (10.center) to (67.center);
		\draw [thick, densely dashdotted, red]  [out=90,in=0] (0.center) to (18.center);
		\draw [thick, dashed, green] (108.center) to (12.center);
		\draw [thick, densely dashdotted, red]  [out=90,in=0] (12.center) to (34.center);
		\draw [thick, dashed, green] (30.center) to (36.center);
		\draw [thick, dashed, green] (144.center) to (145.center);
		\draw [thick, densely dashdotted, red]  (134.center) to (119.center);
		\draw [thick, densely dashdotted, red]  [out=90,in=180] (34.center) to (134.center);
		\draw [thick, densely dashdotted, red]  [out=90,in=0] (145.center) to (134.center);
		\draw [thick, dashed, green] [out=90,in=0] (109.center) to (110.center);
		\draw [thick, dashed, green] [out=90,in=180] (6.center) to (110.center);
		\draw [thick, dashed, green] (110.center) to (111.center);
		\draw [thick, dashed, green] [out=270,in=180] (105.center) to (111.center);
		\draw [thick, dashed, green] [out=270,in=0] (9.center) to (111.center);
	\end{pgfonlayer}
\end{tikzpicture}
\;\; \leadsto \;\;  (-1)^{6} \;\;
\begin{tikzpicture}[baseline=-2ex,scale=.4, color=\clr ]
	\begin{pgfonlayer}{nodelayer}
		\node [style=none] (0) at (3.75, 2) {};
		\node [style=none] (1) at (3.75, -1) {};
		\node [style=none] (3) at (-2, -1) {};
		\node [style=none] (4) at (-2, -2) {};
		\node [style=none] (5) at (3, -1) {};
		\node [style=none] (6) at (3, -2) {};
		\node [style=none] (7) at (4.75, -1) {};
		\node [style=none] (8) at (4.75, -2) {};
		\node [style=none] (9) at (3.75, -2) {};
		\node [style=none] (10) at (-2.5, 3) {};
		\node [style=none] (11) at (-2.5, 4) {};
		\node [style=none] (12) at (4.75, 2) {};
		\node [style=none] (13) at (1, 2) {};
		\node [style=none] (14) at (1.5, 3) {};
		\node [style=none] (15) at (1.5, 4) {};
		\node [style=none] (16) at (2, 2) {};
		\node [style=none] (17) at (3, 2) {};
		\node [style=none] (18) at (3.5, 2.75) {};
		\node [style=none] (19) at (4, 4) {};
		\node [style=none] (28) at (-0.5, 4) {};
		\node [style=none] (29) at (0, 2) {};
		\node [style=none] (30) at (-1, 2) {};
		\node [style=none] (31) at (-0.5, 3) {};
		\node [style=none] (33) at (4.5, 2.75) {};
		\node [style=none] (34) at (4, 3.25) {};
		\node [style=none] (35) at (-1, -2) {};
		\node [style=none] (36) at (-1, -1) {};
		\node [style=none] (38) at (0, -2) {};
		\node [style=none] (39) at (0, -1) {};
		\node [style=none] (41) at (1, -2) {};
		\node [style=none] (42) at (1, -1) {};
		\node [style=none] (44) at (2, -2) {};
		\node [style=none] (45) at (2, -1) {};
		\node [style=none] (47) at (-3, -2) {};
		\node [style=none] (48) at (-3, -1) {};
		\node [style=none] (51) at (1.5, 0) {};
		\node [style=none] (52) at (1.5, -0.5) {};
		\node [style=none] (54) at (0.5, 0.75) {};
		\node [style=none] (55) at (2.5, 0.75) {};
		\node [style=none] (57) at (0.5, 1.25) {};
		\node [style=none] (58) at (2.5, 1.25) {};
		\node [style=none] (59) at (-3, -1) {};
		\node [style=none] (60) at (-3, 0.5) {};
		\node [style=none] (61) at (-2, 0.5) {};
		\node [style=none] (62) at (-2, -1) {};
		\node [style=none] (63) at (-2.5, -0.5) {};
		\node [style=none] (64) at (-2.5, 0) {};
		\node [style=none] (65) at (-3, 0.5) {};
		\node [style=none] (66) at (-3, 2) {};
		\node [style=none] (67) at (-2, 2) {};
		\node [style=none] (68) at (-2, 0.5) {};
		\node [style=none] (69) at (-2.5, 1) {};
		\node [style=none] (70) at (-2.5, 1.5) {};
		\node [style=none] (112) at (5.75, 2) {};
		\node [style=none] (108) at (5.75, -2) {};

	\end{pgfonlayer}
	\begin{pgfonlayer}{edgelayer}
		\draw [thick, color=black] (0.center) to (1.center);
		\draw [thick, color=black] (7.center) to (8.center);
		\draw [thick, color=black] (3.center) to (4.center);
		\draw [thick, color=black] [out=90,in=0] (16.center) to (14.center);
		\draw [thick, color=black] [out=90,in=180] (13.center) to (14.center);
		\draw [thick, color=black] [out=90,in=0] (29.center) to (31.center);
		\draw [thick, color=black] [out=90,in=180] (30.center) to (31.center);
		\draw [thick, color=black] (36.center) to (35.center);
		\draw [thick, color=black] (39.center) to (38.center);
		\draw [thick, color=black] (42.center) to (41.center);
		\draw [thick, color=black] (45.center) to (44.center);
		\draw [thick, color=black] (48.center) to (47.center);
		\draw [thick, color=black] [out=180,in=90] (52.center) to (42.center);
		\draw [thick, color=black] [out=0,in=90] (52.center) to (45.center);
		\draw [thick, color=black] [out=180,in=0] (55.center) to (51.center);
		\draw [thick, color=black] [out=0,in=180] (54.center) to (51.center);
		\draw [thick, color=black] [out=0,in=90] (55.center) to (5.center);
		\draw [thick, color=black] [out=180,in=90] (54.center) to (39.center);
		\draw [thick, color=black] [out=180,in=270] (57.center) to (29.center);
		\draw [thick, color=black] [out=0,in=270] (57.center) to (13.center);
		\draw [thick, color=black] [out=180,in=270] (58.center) to (16.center);
		\draw [thick, color=black] [out=0,in=270] (58.center) to (17.center);
		\draw [thick, color=black] [out=180,in=90] (63.center) to (59.center);
		\draw [thick, color=black] [out=0,in=90] (63.center) to (62.center);
		\draw [thick, color=black] [out=270,in=0] (61.center) to (64.center);
		\draw [thick, color=black] [out=270,in=180] (60.center) to (64.center);
		\draw [thick, color=black] [out=180,in=90] (69.center) to (65.center);
		\draw [thick, color=black] [out=0,in=90] (69.center) to (68.center);
		\draw [thick, color=black] [out=270,in=0] (67.center) to (70.center);
		\draw [thick, color=black] [out=270,in=180] (66.center) to (70.center);
		\draw [thick, color=black] [out=180,in=90] (10.center) to (66.center);
		\draw [thick, color=black] [out=0,in=90] (10.center) to (67.center);
		\draw [thick, color=black] (7.center) to (12.center);
		\draw [thick, color=black] (30.center) to (36.center);
		\draw [thick, color=black, looseness=1.5] [out=270,in=270] (1.center) to (5.center);
		\draw [thick, color=black, looseness=1.5] [out=90,in=90] (6.center) to (9.center);
		\draw [thick, color=black]  (112.center) to (108.center);
	\end{pgfonlayer}
\end{tikzpicture}
\; \; \leadsto \; \; (-1)^{8} [2]_{q}^{2} \; \;
\begin{tikzpicture}[baseline=-2ex,scale=.4, color=\clr ]
	\begin{pgfonlayer}{nodelayer}
		\node [style=none] (0) at (3.75, 2) {};
		\node [style=none] (1) at (3.75, 1) {};
		\node [style=none] (3) at (-2, -1) {};
		\node [style=none] (4) at (-2, -2) {};
		\node [style=none] (5) at (2.75, 1) {};
		\node [style=none] (6) at (2.75, -2) {};
		\node [style=none] (7) at (4.75, 1) {};
		\node [style=none] (8) at (4.75, -2) {};
		\node [style=none] (9) at (3.75, -2) {};
		\node [style=none] (10) at (-2.5, 3) {};
		\node [style=none] (11) at (-2.5, 4) {};
		\node [style=none] (12) at (4.75, 2) {};
		\node [style=none] (13) at (1, 2) {};
		\node [style=none] (14) at (1.5, 3) {};
		\node [style=none] (15) at (1.5, 4) {};
		\node [style=none] (16) at (2, 2) {};
		\node [style=none] (17) at (3, 2) {};
		\node [style=none] (18) at (3.5, 2.75) {};
		\node [style=none] (19) at (4, 4) {};
		\node [style=none] (28) at (-0.5, 4) {};
		\node [style=none] (29) at (0, 2) {};
		\node [style=none] (30) at (-1, 2) {};
		\node [style=none] (31) at (-0.5, 3) {};
		\node [style=none] (33) at (4.5, 2.75) {};
		\node [style=none] (34) at (4, 3.25) {};
		\node [style=none] (35) at (-1, -2) {};
		\node [style=none] (36) at (-1, -1) {};
		\node [style=none] (38) at (0, -2) {};
		\node [style=none] (39) at (0, -1) {};
		\node [style=none] (41) at (1, -2) {};
		\node [style=none] (42) at (1, -1) {};
		\node [style=none] (44) at (2, -2) {};
		\node [style=none] (45) at (2, -1) {};
		\node [style=none] (47) at (-3, -2) {};
		\node [style=none] (48) at (-3, -1) {};
		\node [style=none] (51) at (1.5, 0) {};
		\node [style=none] (52) at (1.5, -0.5) {};
		\node [style=none] (54) at (0.5, 0.75) {};
		\node [style=none] (55) at (2.5, 0.75) {};
		\node [style=none] (57) at (0.5, 1.25) {};
		\node [style=none] (58) at (2.5, 1.25) {};
		\node [style=none] (59) at (-3, -1) {};
		\node [style=none] (60) at (-3, 0.5) {};
		\node [style=none] (61) at (-2, 0.5) {};
		\node [style=none] (62) at (-2, -1) {};
		\node [style=none] (63) at (-2.5, -0.5) {};
		\node [style=none] (64) at (-2.5, 0) {};
		\node [style=none] (65) at (-3, 0.5) {};
		\node [style=none] (66) at (-3, 2) {};
		\node [style=none] (67) at (-2, 2) {};
		\node [style=none] (68) at (-2, 0.5) {};
		\node [style=none] (69) at (-2.5, 1) {};
		\node [style=none] (70) at (-2.5, 1.5) {};
		\node [style=none] (107) at (5.75, -1) {};
		\node [style=none] (108) at (5.75, -2) {};
		\node [style=none] (112) at (5.75, 1) {};	
	\end{pgfonlayer}
	\begin{pgfonlayer}{edgelayer}
		\draw [thick, color=black] (7.center) to (8.center);
		\draw [thick, color=black, looseness=1.25] [out=90,in=90] (41.center) to (44.center);
		\draw [thick, color=black, looseness=1] [out=90,in=270] (35.center) to (5.center);
		\draw [thick, color=black, looseness=1.25] [out=90,in=90] (47.center) to (4.center);
%
		\draw [thick, color=black] (108.center) to (112.center);
		\draw [thick, color=black, looseness=1] [out=270,in=90] (1.center) to (38.center);
		\draw [thick, color=black, looseness=1.25] [out=90,in=90] (6.center) to (9.center);
	\end{pgfonlayer}
\end{tikzpicture} \; .
\end{equation}
The map $\eta$ is obtained by extending this rule linearly.

We need to check that $\eta$ is well-defined.  First we verify that the defining relations for $\rgWeb$ given in \cite[Definition 2.3]{TVW} involving only green edges are preserved by $\eta$.  Since diagrams with a green edge labeled by $\textcolor{green}{k} \geq 3$ are zero, the ``digon removal relation'' \cite[(2-7)]{TVW} in $\rgWeb$ reduces to the following equality:
\begin{equation*}
\begin{tikzpicture}[baseline=-2ex,scale=.4, color=\clr ]
	\begin{pgfonlayer}{nodelayer}
		\node [style=none] (57) at (-3, 0) {};
		\node [style=none] (58) at (-1, 0) {};
		\node [style=none] (59) at (-2, 1) {};
		\node [style=none] (60) at (-2, -1) {};
		\node [style=none] (61) at (-2, 3) {};
		\node [style=none] (62) at (-2, -3) {};
		\node [style=none] (63) at (-2, -3.5) {\textcolor{green}{\tiny{$2$}}};
		\node [style=none] (65) at (-2, 3.5) {\textcolor{green}{\tiny{$2$}}};
		\node [style=none] (72) at (-3.5, 0) {\textcolor{black}{\tiny{$1$}}};
		\node [style=none] (73) at (-0.5, 0) {\textcolor{black}{\tiny{$1$}}};
	\end{pgfonlayer}
	\begin{pgfonlayer}{edgelayer}
		\draw [thick, dashed, green] [out=180,in=90] (59.center) to (57.center);
		\draw [thick, dashed, green] [out=0,in=90] (59.center) to (58.center);
		\draw [thick, dashed, green] [out=0,in=270] (60.center) to (58.center);
		\draw [thick, dashed, green] [out=180,in=270] (60.center) to (57.center);
		\draw [thick, dashed, green]  (59.center) to (61.center);
		\draw [thick, dashed, green]  (60.center) to (62.center);
	\end{pgfonlayer}
\end{tikzpicture}
\;\; = \; \;  [2]_{q} \; 
\begin{tikzpicture}[baseline=-2ex,scale=.4, color=\clr ]
	\begin{pgfonlayer}{nodelayer}
		\node [style=none] (70) at (2, 3) {};
		\node [style=none] (71) at (2, -3) {};
		\node [style=none] (72) at (2, -3.5) {\textcolor{green}{\tiny{$2$}}};
		\node [style=none] (73) at (2, 3.5) {\textcolor{green}{\tiny{$2$}}};
	\end{pgfonlayer}
	\begin{pgfonlayer}{edgelayer}
		\draw [thick, dashed, green]  (71.center) to (70.center);
	\end{pgfonlayer}
\end{tikzpicture} \;  .
\end{equation*}  Similarly, an example of the ``square switch'' relation \cite[(2-8)]{TVW} is the following equality:
\begin{equation}\label{E:squareswitch}
\begin{tikzpicture}[baseline=-2ex,scale=.4, color=\clr ]
	\begin{pgfonlayer}{nodelayer}
		\node [style=none] (0) at (-4, -2) {};
		\node [style=none] (1) at (-2, -2) {};
		\node [style=none] (2) at (-4, -1) {};
		\node [style=none] (3) at (-2, 0) {};
		\node [style=none] (4) at (-2, 1) {};
		\node [style=none] (5) at (-4, 2) {};
		\node [style=none] (6) at (-4, 3) {};
		\node [style=none] (7) at (-2, 3) {};
		\node [style=none] (16) at (-4, 3.5) {\textcolor{green}{\tiny{$2$}}};
		\node [style=none] (17) at (-2, 3.5) {\textcolor{black}{\tiny{$1$}}};
		\node [style=none] (18) at (-4, -2.5) {\textcolor{green}{\tiny{$2$}}};
		\node [style=none] (19) at (-2, -2.5) {\textcolor{black}{\tiny{$1$}}};
		\node [style=none] (20) at (2, -2) {};
		\node [style=none] (21) at (4, -2) {};
		\node [style=none] (26) at (2, 3) {};
		\node [style=none] (27) at (4, 3) {};
		\node [style=none] (28) at (2, 3.5) {\textcolor{green}{\tiny{$2$}}};
		\node [style=none] (29) at (4, 3.5) {\textcolor{black}{\tiny{$1$}}};
		\node [style=none] (30) at (2, -2.5) {\textcolor{green}{\tiny{$2$}}};
		\node [style=none] (31) at (4, -2.5) {\textcolor{black}{\tiny{$1$}}};
		\node [style=none] (35) at (-4.5, 0.5) {\textcolor{black}{\tiny{$1$}}};
		\node [style=none] (32) at (-1.5, 0.5) {\textcolor{green}{\tiny{$2$}}};
		\node [style=none] (34) at (0.15, 0.4) {\textcolor{black}{=}};
	\end{pgfonlayer}
	\begin{pgfonlayer}{edgelayer}
		\draw [thick, dashed, green] (0.center) to (2.center);
		\draw [thick, dashed, green] [out=0,in=90] (2.center) to (3.center);
		\draw [thick, dashed, green] (1.center) to (3.center);
		\draw [thick, dashed, green] (3.center) to (4.center);
		\draw [thick, dashed, green] (5.center) to (2.center);
		\draw [thick, dashed, green] [out=180,in=90] (4.center) to (5.center);
		\draw [thick, dashed, green] (5.center) to (6.center);
		\draw [thick, dashed, green] (4.center) to (7.center);
		\draw [thick, dashed, green] (27.center) to (21.center);
		\draw [thick, dashed, green] (26.center) to (20.center);
	\end{pgfonlayer}
\end{tikzpicture} \; .
\end{equation}   In every case an easy check verifies that the rules for $\eta$ do the same thing to both sides of the equations.  In checking, note that the green (co)associativity relations given in \cite[(2-6)]{TVW} are zero and there is nothing to check.

As we next explain, the fact it preserves the green relations implies $\eta$ is well-defined.
Set
\begin{equation*}
E \; \; = \; \; 
\begin{tikzpicture}[baseline=-2ex,scale=.5, color=\clr ]
	\begin{pgfonlayer}{nodelayer}
		\node [style=none] (4) at (-3.75, 0) {\tiny{$\dots $}};
		\node [style=none] (9) at (0, 1) {};
		\node [style=none] (10) at (2, 1) {};
		\node [style=none] (11) at (1, 0) {};
		\node [style=none] (12) at (1, -1) {};
		\node [style=none] (13) at (1, 1) {\tiny{$\dots $}};
		\node [style=none] (14) at (0, 1.5) {\textcolor{black}{\tiny{$1$}}};
		\node [style=none] (15) at (2, 1.5) {\textcolor{black}{\tiny{$1$}}};
		\node [style=none] (16) at (1, -1.5) {\textcolor{red}{\tiny{$r$}}};
		\node [style=none] (19) at (-6.5, 1) {};
		\node [style=none] (20) at (-4.5, 1) {};
		\node [style=none] (21) at (-5.5, 0) {};
		\node [style=none] (22) at (-5.5, -1) {};
		\node [style=none] (24) at (-6.5, 1.5) {\textcolor{black}{\tiny{$1$}}};
		\node [style=none] (25) at (-4.5, 1.5) {\textcolor{black}{\tiny{$1$}}};
		\node [style=none] (26) at (-5.5, -1.5) {\textcolor{green}{\tiny{$2$}}};
		\node [style=none] (28) at (-3, 1) {};
		\node [style=none] (29) at (-1, 1) {};
		\node [style=none] (30) at (-2, 0) {};
		\node [style=none] (31) at (-2, -1) {};
		\node [style=none] (33) at (-3, 1.5) {\textcolor{black}{\tiny{$1$}}};
		\node [style=none] (34) at (-1, 1.5) {\textcolor{black}{\tiny{$1$}}};
		\node [style=none] (35) at (-2, -1.5) {\textcolor{green}{\tiny{$2$}}};
	\end{pgfonlayer}
	\begin{pgfonlayer}{edgelayer}
		\draw [thick, densely dashdotted, densely dashdotted, red]  [out=270,in=180] (9.center) to (11.center);
		\draw [thick, densely dashdotted, densely dashdotted, red]  [out=270,in=0] (10.center) to (11.center);
		\draw [thick, densely dashdotted, densely dashdotted, red]  (11.center) to (12.center);
		\draw [thick, dashed, green] [out=270,in=180] (19.center) to (21.center);
		\draw [thick, dashed, green] [out=270,in=0] (20.center) to (21.center);
		\draw [thick, dashed, green] (21.center) to (22.center);
		\draw [thick, dashed, green] [out=270,in=180](28.center) to (30.center);
		\draw [thick, dashed, green] [out=270,in=0](29.center) to (30.center);
		\draw [thick, dashed, green] (30.center) to (31.center);
	\end{pgfonlayer}
\end{tikzpicture}
\; \; \;\;\text{ and } \;\;\;\;
C \; \; = \; \; 
\begin{tikzpicture}[baseline=-2ex,scale=.5, color=\clr, yscale=-1 ]
	\begin{pgfonlayer}{nodelayer}
		\node [style=none] (4) at (-3.75, 0) {\tiny{$\dots $}};
		\node [style=none] (9) at (0, 1) {};
		\node [style=none] (10) at (2, 1) {};
		\node [style=none] (11) at (1, 0) {};
		\node [style=none] (12) at (1, -1) {};
		\node [style=none] (13) at (1, 1) {\tiny{$\dots $}};
		\node [style=none] (14) at (0, 1.5) {\textcolor{black}{\tiny{$1$}}};
		\node [style=none] (15) at (2, 1.5) {\textcolor{black}{\tiny{$1$}}};
		\node [style=none] (16) at (1, -1.5) {\textcolor{red}{\tiny{$r$}}};
		\node [style=none] (19) at (-6.5, 1) {};
		\node [style=none] (20) at (-4.5, 1) {};
		\node [style=none] (21) at (-5.5, 0) {};
		\node [style=none] (22) at (-5.5, -1) {};
		\node [style=none] (24) at (-6.5, 1.5) {\textcolor{black}{\tiny{$1$}}};
		\node [style=none] (25) at (-4.5, 1.5) {\textcolor{black}{\tiny{$1$}}};
		\node [style=none] (26) at (-5.5, -1.5) {\textcolor{green}{\tiny{$2$}}};
		\node [style=none] (28) at (-3, 1) {};
		\node [style=none] (29) at (-1, 1) {};
		\node [style=none] (30) at (-2, 0) {};
		\node [style=none] (31) at (-2, -1) {};
		\node [style=none] (33) at (-3, 1.5) {\textcolor{black}{\tiny{$1$}}};
		\node [style=none] (34) at (-1, 1.5) {\textcolor{black}{\tiny{$1$}}};
		\node [style=none] (35) at (-2, -1.5) {\textcolor{green}{\tiny{$2$}}};
	\end{pgfonlayer}
	\begin{pgfonlayer}{edgelayer}
		\draw [thick, densely dashdotted, densely dashdotted, red]  [out=270,in=180] (9.center) to (11.center);
		\draw [thick, densely dashdotted, densely dashdotted, red]  [out=270,in=0] (10.center) to (11.center);
		\draw [thick, densely dashdotted, densely dashdotted, red] (11.center) to (12.center);
		\draw [thick, dashed, green] [out=270,in=180] (19.center) to (21.center);
		\draw [thick, dashed, green] [out=270,in=0] (20.center) to (21.center);
		\draw [thick, dashed, green] (21.center) to (22.center);
		\draw [thick, dashed, green] [out=270,in=180](28.center) to (30.center);
		\draw [thick, dashed, green] [out=270,in=0](29.center) to (30.center);
		\draw [thick, dashed, green] (30.center) to (31.center);
	\end{pgfonlayer}
\end{tikzpicture} \; .
\end{equation*}
Here $E$ has $n$ green splits and the red diagram is an  $\textcolor{red}{r}$ split apart into $r$ strands labeled by $1$ by repeatedly concatenating red splits. Likewise, $C$ has $n$ green merges and and the red diagram is $r$ strands labeled by $1$ that are merged together into a single $\textcolor{red}{r}$ by concatenating red merges.  

If $D_{1}$ and $D_{2}$ are two linear combinations of diagrams, all of the form \cref{E:Diagram},  which are equal in $\rgWeb$, then we can concatenate an $E$ on top of each to obtain $ED_{1}$ and $ED_{2}$.  These are elements of $\End_{\rgWeb}(1^{d})$ and will still be equal in $\rgWeb$.

Applying \cite[Lemma 2.12, Corollary 2.13]{TVW} to the red part of $ED_{1}$ and $ED_{2}$ shows they can be rewritten as a linear combination of diagrams in $\rgWeb$ consisting of only green edges.  That is, $ED_{1}$ and $ED_{2}$ are morphisms in the subcategory $\gWeb$ (the subcategory of green objects and the morphisms in $\rgWeb$ which can be expressed as a linear combination of diagrams with only green edges).  However, by \cite[Corollary 2.16]{TVW} the category $\gWeb$ is equivalent to the ``green-ification'' of the category $\CKMWeb$ introduced by Cautis--Kamnitzer--Morrison.  Thus, if two diagrams are equal in $\gWeb$, they must be related by the application of a finite sequence of defining relations of $\CKMWeb$ which have been green-ified.  However, a check verifies that the greenification of the defining relations of $\CKMWeb$ are nothing but relations in $\rgWeb$ which follow from the all-green defining  relations.

The discussion in the previous paragraph implies there is a sequence of linear combinations of diagrams in $\End_{\gWeb}(1^{\otimes d})$, 
\[
u_{1}, u_{2}, \dotsc , u_{p},
\] where each $u_{i}$ consists of only green edges, each $u_{i}$ is obtained from $u_{i-1}$ by applying one of the green relations of $\rgWeb$, and where $u_{1}=ED_{1}$ and $u_{2}=ED_{2}$ in $\rgWeb$. Concatenating each term of this sequence on top with $C$ yields the sequence,
\[
Cu_{1}, Cu_{2}, \dotsc , Cu_{p},
\] where each $Cu_{i}$ is obtained from $Cu_{i-1}$ by applying one of the green relations of $\rgWeb$ to $u_{i-1}$ while keeping $C$ fixed.

However, note that each $Cu_{i}$ is a linear combination of diagrams of the form given in \cref{E:Diagram} and so it makes sense to apply the rule defining $\eta$ to each of them.  Since $\eta$ satisfies the green relations of $\rgWeb$ it follows that 
\[
\eta(Cu_{1})= \eta(Cu_{2})= \dotsc = \eta(Cu_{p}).
\]

Set $\kappa=[r]_{q}![2]_{q}^{n}$.  We claim that $\eta (D_{1}) = \frac{1}{\kappa}\eta (Cu_{1})$ and $\eta (D_{2}) = \frac{1}{\kappa}\eta(Cu_{p})$.  Since the rule for $\eta$ is local and $u_{1}$ (resp., $u_{p}$) is obtained by applying \cite[Lemma 2.12, Corollary 2.13]{TVW} only to the red part of $ED_{1}$ (resp., $ED_{2}$), it suffices to consider the case when $D_{1}=D_{2}=D$ is a diagram of the form given in \cref{E:Diagram} where $X$ consists of $d$ black vertical strands labeled by $1$.  That is, the case when $D = C$.  By \cite[Lemma 2.12]{TVW}, $ED = \kappa \mathcal{CL}$, where the $\mathcal{CL}$ is the green-red diagram given by the left-to-right horizontal concatenation of $n$ green clasps $\mathcal{CL}_{2}^{\textcolor{green}{g}}$ and one red clasp $ \mathcal{CL}_{r}^{\textcolor{red}{r}}$.  By \cite[Corollary 2.13, Definition 2.11]{TVW} the red clasp can be rewritten into smaller red clasps and green diagrams, and by repeatedly using this recursion one ultimately obtains $u_{1}$ for $ED$.  Using this recursion and our convention that any web with an edge connecting two vertices from among $\{1', \dotsc , r' \}$ is equal to zero, a straightforward induction on the size of the red clasp shows that $\eta(Cu_{1})=\kappa w_{0}$.  Since $\eta(D)=w_{0}$, it follows that $\eta(D) = \frac{1}{\kappa}\eta(Cu_{1})$, as desired.

Combining the previous calculations yields
\[
\eta (D_{1}) = \frac{1}{\kappa} \eta \left( C u_{1} \right) = \frac{1}{\kappa} \eta \left( Cu_{p} \right)  = \eta (D_{2}).
\] Therefore $\eta$ is well-defined.

To show that $\eta$ is a $\HH_{d}(q)$-module homomorphism requires that we analyze the two ways the generator $T_{i}$ may act on the web module as described in \cref{SS:Webmodule}.  The easier case is when at least one of the vertices $i$ and $i+1$ are not adjacent to through-strings in $\eta(D)$. In this case equivariance for the generator $T_{i}$ is immediate from the observation that when the rules for $\eta$ are applied to the green-red diagram $\beta_{i}$ the result is the linear combination of webs which gives the action of $T_{i}$ on $\eta(D)$.

The second and less obvious case is when both $i$ and $i+1$ are through-strings in $\eta(D)$. In this case, 
\[
\eta(D.T_{i})=q\eta(D) + \eta (D\tilde{E}_{i}),
\]
 where $\tilde{E}_{i}$ is the second diagram in the definition of $\beta_{i}$ given in \cref{SS:TVW-webs}, and where $D\tilde{E}_{i}$ is the green-red diagram obtained by concatenating $\tilde{E}_{i}$ onto the bottom of $D$.  It will suffice to show $D\tilde{E}_{i}=0$ in $\rgWeb$.

First, by the rule for $\eta$, the tops of through-strings in $\eta(D)$ correspond to vertices where the red part of $D$ meets $X$. Since the two through-strings in question have adjacent vertices at the bottom, they also have adjacent vertices at the top. In $D$ call these top vertices $j$ and $j+1$.  By the associativity relation \cite[Definition 2.3]{TVW} the red merge diagrams in $D$ can be concatenated in any order.  Therefore we may assume that the lowest red merge in $D$ is located at vertices $j$ and $j+1$.  For example, in \cref{E:etaexample} if $i=3$, then $j=7$ and the diagram as drawn already has lowest red merge at the location of vertices $j$ and $j+1$;  if $i=9$, then $j=9$ and we would use the associativity relation for red merges to replace the red part of $D$ with, for example, the same red diagram reflected across a vertical axis.

Second, since we know that in $\eta(D)$ we have parallel through-strings from $i$ to $j$ and from $i+1$ to $j+1$, it must be that in the $X$ portion of $D$ the corresponding green edges labeled by $1$ must travel in parallel from $i$ to $j$ and from $i+1$ to $j+1$.  As seen in \cref{E:etaexample}, the  region enclosed in $X$ by these two sequences of green edges may be bisected by some number of green edges labeled by $\textcolor{green}{2}$.  By repeatedly using the ``square switch'' relation given in \cref{E:squareswitch} these edges labeled by $\textcolor{green}{2}$ may be eliminated.  The end result is that $D\tilde{E}_{i}$ is equal in $\rgWeb$ to a green-red diagram which contains a subdiagram of the form:
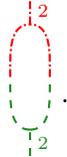
\begin{equation*}
\begin{tikzpicture}[baseline=-2ex,scale=.35, color=\clr, yscale=-1 ]
	\begin{pgfonlayer}{nodelayer}
		\node [style=none] (0) at (0, 3) {};
		\node [style=none] (1) at (0, 2) {};
		\node [style=none] (2) at (0.75, 0) {};
		\node [style=none] (3) at (-0.75, 0) {};
		\node [style=none] (4) at (0, -2) {};
		\node [style=none] (5) at (0, -3) {};
		\node [style=none] (15) at (0.5, 2.5) {\textcolor{green}{\tiny{2}}};
		\node [style=none] (15) at (0.5, -2.5) {\textcolor{red}{\tiny{2}}};
	\end{pgfonlayer}
	\begin{pgfonlayer}{edgelayer}
		\draw [thick, dashed, green] (0.center) to (1.center);
		\draw [thick, dashed, green] [out=180,in=90](1.center) to (3.center);
		\draw [thick, dashed, green] [out=0,in=90](1.center) to (2.center);
		\draw [thick, densely dashdotted,, red] [out=270,in=180] (3.center) to (4.center);
		\draw [thick, densely dashdotted,, red] [out=270,in=0] (2.center) to (4.center);
		\draw [thick, densely dashdotted,, red]  (4.center) to (5.center);
	\end{pgfonlayer}
\end{tikzpicture}\; .
\end{equation*}  However, by \cite[Lemma 2.9]{TVW} this subdiagram and, hence, $D\tilde{E}_{i}$, is equal to zero in $\rgWeb$.  Thus $\eta$ is equivariant for $T_{i}$ in this case as well.

Therefore $\eta$ is a nonzero $\HH_{d}(q)$-module homomorphism from $\Hom_{\rgWeb}(1^{\otimes d}, \textcolor{green}{2}^{\otimes n}\textcolor{red}{r})$ to $W^{(n+r,n)}$.  As mentioned above, this gives the desired $\HH_{d}(q)$-module isomorphism 
\[
\varphi : S^{(n+r,n)} \xrightarrow{\cong} W^{(n+r,n)}.
\]

Finally, observe that for $k=1, \dotsc , n$
\[
w_{0}T_{2k-1} = -q^{-1}w_{0}.
\]  Therefore, 
\[
w_{0}y_{\lambda^{\TT}} = (1+q^{-2})^{n}w_{0}.
\]   \cref{L:KeyLemma} implies that, after rescaling, $\varphi(v_{0})=\varphi(v_{t_{0}})=\varphi (v_{t^{\lambda^{\TT}}})=\varphi (z_{\lambda}) = w_{0}$, as claimed. 
\end{proof}

\section{Main Results}\label{S:MainResults}

\subsection{Unitriangularity}\label{SS:triangularity}

For the remainder of the paper let 
\[
\varphi: S^{\lambda} \to W^{(n+r,n)}
\] be the isomorphism given by \cref{T:TVW}.  Here and in the next result, recall that we write $W^{(n+r,n)}$ for the $\k$-vector space with basis the set of $(n,r)$-webs, $\WW^{(n+r,n)}$ (see \cref{SS:Webmodule}).

\begin{theorem}\label{T:uppertriangular} Let $\lambda = (n+r,n)$ be a partition of $2n+r$ with two parts.  For every $t \in \Std (\lambda^{\TT})$, 
\[
\varphi(v_{t}) =\psi (t^{\TT}) + \sum_{\substack{w \in \WW^{(n+r,n)}\\ \nest (w) < \nest (\psi (t^{\TT}))}} c_{w,t} w,
\] where $c_{w,t} \in \Z[q,q^{-1}]$.

\end{theorem}

\begin{proof}  The proof is by induction on the set $\Std (\lambda^{\TT})$ with respect to the weak Bruhat order.  The base case is the unique minimal element $t^{\lambda^{\TT}}$.  This corresponds to the basis element $v_{0}=v_{t_{0}}=v_{t^{\lambda^{\TT}}}=z_{\lambda}$.  By \cref{T:TVW}, $\varphi(v_{0}) = w_{0} = \psi (t^{\TT}_{0})$. As $w_{0}$ is the unique element of $\WW^{(n+r,n)}$ with nesting number zero,  the equation holds in this case.

Now let $t \in \Std (\lambda^{\TT})$ be given and choose $t_{1} \in \Std (\lambda^{\TT})$ such that $t_{1} \leq_{i} t$ for some $i=1, \dotsc , 2n+r-1$.  In particular, $t=t_{1}.s_{i}$ and $v_{t}= v_{t_{1}}T_{i}$. Then, 
\begin{align*}
\varphi(v_{t}) &= \varphi(v_{t_{1}}T_{i}) \\
                   & = \varphi(v_{t_{1}})T_{i} \\
                    & =\left(\psi (t_{1}) + \sum_{\substack{w \in \WW^{(n+r,n)} \\ \nest (w) < \nest (\psi (t_{1}^{\TT}))}} c_{w,t_{1}} w  \right)T_{i} \\
& = \left(q\psi (t_{1}) + \sum_{\substack{w \in \WW^{(n+r,n)} \\  \nest (w) < \nest (\psi (t_{1}^{\TT}))}} qc_{w,t_{1}} w   \right) +   \left(\psi (t_{1}) + \sum_{\substack{w \in \WW^{(n+r,n)} \\  \nest (w) < \nest (\psi (t_{1}^{\TT}))}} c_{w,t_{1}} w  \right)E_{i},
\end{align*} where the third equality is by the inductive assumption, and the fourth is by the definition of the action of $\HH_{d}(q)$ on $W^{\lambda}$.

In the first term of the last expression, since $t_{1} <_{i} t$, \cref{P:OrderingProposition} implies $\psi(t^{\TT}_{1}) \prec_{i} \psi (t^{\TT})$  and, as a consequence, $\nest (\psi (t_{1}^{\TT})) < \nest (\psi (t^{\TT}))$.  Thus, every web which appears with a nonzero coefficient in this term has nesting number strictly smaller than the nesting number of $\psi (t^{\TT})$.

In the second term of the last expression, it follows from \cref{P:OrderingProposition} that $\psi (t_{1}^{\TT})E_{i}=\psi (t^{\TT})$ and $\nest (\psi(t^{\TT })) = \nest (\psi (t_{1}^{\TT}))+1$. For the rest of the webs which appear in this term, a check of the list of possibilities given in \cref{SS:partialorderonwebs} shows that the action $E_{i}$ on a web raises the nesting number by at most one.  Thus the webs which appear with a nonzero coefficient in the sum after acting by $E_{i}$ must have nesting number strictly smaller than the nesting number of $\psi (t^{\TT})$.

Finally, since the action of each $T_{i}$ on a web is a $\Z[q,q^{-1}]$-linear combination of webs, the assertion about the coefficients is immediate.
\end{proof}

\begin{remark}\label{R:integral}
Let $R = \Z[q,q^{-1}]$.  Let $\HH_{d}(q)_{R}$ be the associative, unital $R$-algebra defined by the same generators and relations as in \cref{R:Heckeisom}.  Given $\lambda=(n+r,n)$, one can define the Specht $\HH_{2n+r}(q)_{R}$-module $S^{\lambda}_{R}$ just as in \cref{SS:SpechtModules} and by \cite[Theorem 5.6]{DipperJames} the standard basis is an $R$-basis for $S^{\lambda}_{R}$.  Similarly, let $W^{\lambda}_{R}$ be the free $R$-module spanned by the non-crossing webs, $\WW^{(n+r,n)}$.  As the action of the generators of $\HH_{2n+r}(q)$ on webs only involves coefficients from $R$, it follows that $W^{\lambda}_{R}$ is an $\HH_{d}(q)_{R}$-module.

Viewing $S^{\lambda}_{R}$ and $W^{\lambda}_{R}$ as $R$-submodules of $S^{\lambda}$ and $W^{\lambda}$, respectively, the previous result shows that the map $\varphi$ restricts to define an $\HH_{d}(q)_{R}$-module isomorphism 
\[
\varphi: S^{\lambda}_{R} \to W^{\lambda}_{R}.
\]  Therefore the results of this paper hold over $R$ and, via base change, over any integral domain $S$ and any choice of invertible $q \in S$.
\end{remark}

As in \cite{RT}, the previous result can be sharpened.

\begin{theorem}\label{T:uppertriangular2} Let $\lambda = (n+r,n)$ be a partition of $2n+r$ with two parts.  For every $t \in \Std (\lambda^{\TT})$, 
\[
\varphi(v_{t}) =\psi (t^{\TT}) + \sum_{\substack{w \in \WW^{(n+r,n)}\\ w \prec \psi (t^{\TT})}} c_{w,t} w,
\] where $c_{w,t} \in \Z[q,q^{-1}]$.
\end{theorem}

\begin{proof} Given the previous theorem, it suffices to show $c_{w, t} \neq 0$ implies $w \preceq \psi (t^{\TT})$. The proof of this is by induction on the set $\Std (\lambda^{\TT})$ with respect to the weak Bruhat order.   The base case of $t=t_{0}=t^{\lambda^{\TT}}$ is trivial. For the inductive step, let $t \in \Std (\lambda^{\TT})$ be given and choose $t_{1} \in \Std (\lambda^{\TT})$ such that $t_{1} \leq_{i} t$ for some $i=1, \dotsc , 2n+r-1$.  By the inductive assumption, $c_{\tilde{w}, t_{1}} \neq 0$ implies $\tilde{w} \preceq \psi (t_{1}^{\TT})$ for any $\tilde{w} \in \WW^{(n+r,n)}$.

Consider, 
\begin{align*}
\varphi(v_{t}) = \varphi(v_{t_{1}})T_{i} &= \sum_{\substack{\tilde{w} \in \WW^{(n+r,n)}\\ \tilde{w} \preceq \psi (t_{1}^{\TT})}} c_{\tilde{w},t_{1}} \tilde{w}T_{i}\\
               &= \sum_{\substack{\tilde{w} \in \WW^{(n+r,n)}\\ \tilde{w} \preceq \psi (t_{1}^{\TT})}} qc_{\tilde{w},t_{1}} \tilde{w} + \sum_{\substack{\tilde{w} \in \WW^{(n+r,n)}\\ \tilde{w} \preceq \psi (t_{1}^{\TT})}} c_{\tilde{w},t_{1}} \tilde{w}E_{i}.
\end{align*}   In order for $w \in \WW^{(n+r,n)}$ to satisfy $c_{w,t} \neq 0$, either $w = \tilde{w}$ for some $\tilde{w}$ in the first sum or $w = \tilde{w}E_{i}$ for some $\tilde{w}$ in the second sum (or both).  In the first case, $w = \tilde{w} \preceq \psi (t_{1}^{\TT})$.  But since $t_{1} \leq_{i} t$, by \cref{P:OrderingProposition} $\psi (t_{1}^{\TT}) \preceq_{i} \psi (t^{\TT})$ and, hence, $w \preceq \psi (t^{\TT})$, as desired.

In the second case, $w = \tilde{w}E_{i}$. We consider the cases listed in \cref{SS:partialorderonwebs} and how they apply to $\tilde{w}$.  Obviously, Case (1)  is not relevant.  If we happen to be in Case (4), then $\langle i, i+1 \rangle$ is an arc in $\tilde{w}$, $w = -[2]_{q}\tilde{w}$, and the reasoning of the previous paragraph still applies.  If we are in Case (5), then $\tilde{w} \preceq_{i} w$.  But we also have $\tilde{w} \preceq \psi (t_{1}^{\TT})$ and $\psi (t_{1}^{\TT}) \preceq_{i} \psi (t^{\TT})$.  By the fourth claim in \cref{L:WeakBruhatProperties}, this implies $w \preceq \psi (t^{\TT})$, as desired.

The remaining scenario is if $w = \tilde{w}E_{i}$, $\nest (w) < \nest (\tilde{w})$, and one of Cases (2), (3), or (6) listed in \cref{SS:partialorderonwebs} applies to $\tilde{w}$.  In each of these cases the desired result follows from \cref{L:PartialOrder}.  Namely, if we are in Case (3), then $\langle i, y \rangle$ and $\langle i+1, x \rangle $ are arcs and $\tilde{w}=w'$ in the third pair of \cref{L:PartialOrder} (where $p$ is $i$, $i$ is $i+1$, $j$ is $x$, and $q$ is $y$),  But then $w=\tilde{w}E_{i}$ is the $w$ of that pair and therefore $w \preceq \tilde{w}$.  But since $\tilde{w} \preceq \psi (t_{1}^{\TT}) \preceq \psi (t^{\TT})$, the desired inequality holds.  If we are in Case (6), then $\langle x, i+1 \rangle$ and $\langle y, i \rangle$ are arcs in $\tilde{w}$, again $\tilde{w}=w'$ in the third pair of \cref{L:PartialOrder} (where $p$ is $x$, $i$ is $y$, $j$ is $i$, and $q$ is $i+1$), $w=\tilde{w}E_{i}$ is the $w$ of that pair  and therefore $w \preceq \tilde{w}$.  But since $\tilde{w} \preceq \psi (t_{1}^{\TT}) \preceq \psi (t^{\TT})$, the desired inequality holds.

The last case to consider is Case (2). Then $\langle i, k' \rangle$ is a through-string, $\langle i+1, x \rangle$ is an arc, and locally $\tilde{w}$ looks like the last diagram given below:
 
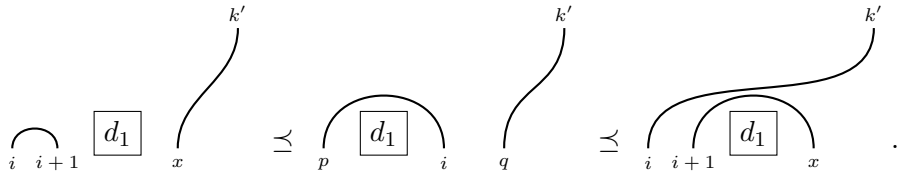
\begin{equation*}
\begin{tikzpicture}[baseline=0ex,scale=0.4, color=\clr ]
	\begin{pgfonlayer}{nodelayer}
		\node [style=none] (0) at (1, 0) {};
		\node [style=none] (1) at (1, 0) {};
		\node [style=none] (2) at (-4.5, 0) {};
		\node [style=none] (4) at (1, -0.5) {\tiny{$x$}};
		\node [style=none] (6) at (-4.5, -0.5) {\tiny{$i$}};
		\node [draw] (8) at (-1, 0.5) {$d_{1}$};
		\node [style=none] (10) at (-3, 0) {};
		\node [style=none] (11) at (-3, -0.5) {\tiny{$i+1$}};
		\node [style=none] (30) at (3, 4) {};
		\node [style=none] (31) at (3, 4.5) {\tiny{$k'$}};
	\end{pgfonlayer}
	\begin{pgfonlayer}{edgelayer}
		\draw [thick, looseness=1.5] [out=90,in=90] (2.center) to (10.center);
		\draw [thick, looseness=1] [out=90,in=270] (0.center) to (30.center);
	\end{pgfonlayer}
\end{tikzpicture} \;\; \preceq \; \; 
\begin{tikzpicture}[baseline=0ex,scale=0.4, color=\clr ]
	\begin{pgfonlayer}{nodelayer}
		\node [style=none] (0) at (-1, 0) {};
		\node [style=none] (1) at (1, 0) {};
		\node [style=none] (2) at (-5, 0) {};
		\node [style=none] (3) at (1, 0) {};
		\node [style=none] (4) at (-1, -0.5) {\tiny{$i$}};
		\node [style=none] (6) at (-5, -0.5) {\tiny{$p$}};
		\node [style=none] (7) at (1, -0.5) {\tiny{$q$}};
		\node [draw] (8) at (-3, 0.5) {$d_{1}$};
		\node [style=none] (30) at (3, 4) {};
		\node [style=none] (31) at (3, 4.5) {\tiny{$k'$}};
	\end{pgfonlayer}
	\begin{pgfonlayer}{edgelayer}
		\draw [thick, looseness=1.5] [out=90,in=90] (0.center) to (2.center);
		\draw [thick, looseness=1.25] [out=270,in=90] (30.center) to (3.center);
	\end{pgfonlayer}
\end{tikzpicture}
\;\; \preceq \;\;
\begin{tikzpicture}[baseline=0ex,scale=0.4, color=\clr ]
	\begin{pgfonlayer}{nodelayer}
		\node [style=none] (0) at (1, 0) {};
		\node [style=none] (1) at (1, 0) {};
		\node [style=none] (2) at (-4.5, 0) {};
		\node [style=none] (4) at (1, -0.5) {\tiny{$x$}};
		\node [style=none] (6) at (-4.5, -0.5) {\tiny{$i$}};
		\node [draw] (8) at (-1, 0.5) {$d_{1}$};
		\node [style=none] (10) at (-3, 0) {};
		\node [style=none] (11) at (-3, -0.5) {\tiny{$i+1$}};
		\node [style=none] (30) at (3, 4) {};
		\node [style=none] (31) at (3, 4.5) {\tiny{$k'$}};
	\end{pgfonlayer}
	\begin{pgfonlayer}{edgelayer}
		\draw [thick, looseness=1.5] [out=90,in=90] (10.center) to (0.center);
		\draw [thick, looseness=1] [out=90,in=270] (2.center) to (30.center);
	\end{pgfonlayer}
\end{tikzpicture}\; .
\end{equation*}
The first diagram is $w = \tilde{w}E_{i}$.  The first inequality holds by using the second pair of webs in \cref{L:PartialOrder}, possibly multiple times, and the second inequality holds by using the fourth pair of webs in \cref{L:PartialOrder}  (where $p$ is $i$, $p+1$ is $i+1$, $i+1$ is $x$, and $k'$ is $k'$).  Once again, $w \preceq \tilde{w} \preceq \psi (t_{1}^{\TT}) \preceq \psi (t^{\TT})$, as desired.

In short, in all cases, if $c_{w, t} \neq 0$, then $w \preceq \psi (t^{\TT})$ and the assertion is proven.
\end{proof}

\subsection{Positivity}\label{SS:positivity}

The goal of this section is to prove the following result.

\begin{theorem}\label{T:positivity}  Let $\lambda = (n+r,n)$ be a partition of $2n+r$.  For every $t \in \Std (\lambda^{\TT})$,
\[
\varphi (v_{t}) = \sum_{w \in \WW^{(n+r,n)}} c_{w,t} w
\] with $c_{w,t} \in \Z_{\geq 0}[q]$ for all $w \in \WW^{(n+r,n)}$.
\end{theorem}

 By \cref{E:standardbasis}, $v_{t}=v_{0}T_{i_{1}}\dotsb T_{i_{r}}$ where $s_{i_{1}}\dotsb s_{i_{r}}$ is a reduced word in $S_{d}$ and 
\begin{equation}\label{E:weakbruhatX}
t^{\lambda^{\TT}} =t_{0}\leq_{i_{1}} t_{1} \leq_{i_{2}} \dotsb  \leq_{i_{r}} t_{r}= t
\end{equation}
in the weak Bruhat order.  Then,
\begin{equation}\label{E:CrossingWeb}
\varphi(v_{t}) = \varphi(v_{0}T_{i_{1}}\dotsb T_{i_{r}})=\varphi(v_{0})T_{i_{1}}\dotsb T_{i_{r}} = w_{0}\beta_{i_{1}}\dotsb \beta_{i_{r}}.
\end{equation}

The proof strategy is straightforward enough.  The right hand side of \cref{E:CrossingWeb} is an $(n+r,n)$-web with crossings.  Since relation \cref{E:crossingrelation} only involves coefficients in $\Z_{\geq 0}[q]$, if the web could be rewritten into the basis of non-crossing webs using only this relation, then the same would be true of the coefficients in the resulting linear combination.  However, if one simply applies \cref{E:crossingrelation} to each crossing, then it is quite possible that bubbles will appear. This will generate undesirable coefficients when they are popped using relation \cref{E:bubblerelation}.  The example in \cref{SS:ExampleI} shows how this can happen.  The key observation, in that example and in general, is that the Reidemiester II moves from \cref{L:RII} also have coefficients in $\Z_{\geq 0}[q]$ and  can be used along with relation \cref{E:crossingrelation} to eliminate all crossings in a way which entirely avoids the dreaded bubble popping relation.

To show how this can be done we first need some terminology.  Given an $(n+r,n)$-web with crossings consider the complement of its projection onto the plane.  We call a connected component of the complement \emph{bounded} if the closure of the connected component is compact.  Otherwise we call it \emph{unbounded}.    Given a web, we refer to these components as bounded and unbounded regions, respectively.   For example, the interior of a bubble is a bounded region.

\begin{example} The $(3,3)$-web given in \cref{SS:ExampleI} has the following projection onto the plane: 
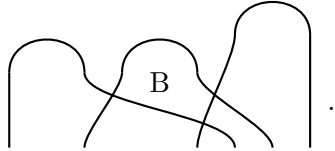
\[\label{E:Example1}
\begin{tikzpicture}[baseline=0.5cm,scale=.5, color=\clr ]
	\begin{pgfonlayer}{nodelayer}
		\node [style=none] (0) at (1, 0) {};
		\node [style=none] (1) at (2, 3) {};
		\node [style=none] (2) at (4, 3) {};
		\node [style=none] (3) at (4, 0) {};
		\node [style=none] (4) at (-2, 0) {};
		\node [style=none] (5) at (-1, 2) {};
		\node [style=none] (6) at (1, 2) {};
		\node [style=none] (7) at (-4, 0) {};
		\node [style=none] (8) at (-5, 2) {};
		\node [style=none] (9) at (-2, 2) {};
		\node [style=none] (10) at (-4, 2) {};
		\node [style=none] (11) at (2, 0) {};
		\node [style=none] (12) at (3, 0) {};
		\node [style=none] (13) at (-1.75, 1.62) {};
		\node [style=none] (14) at (-1, 1.28) {};
		\node [style=none] (15) at (0.75, 0.77) {};
		\node [style=none] (16) at (1.5, 0.5) {};
		\node [style=none] (17) at (1.25, 1.6) {};
		\node [style=none] (18) at (1.75, 1.2) {};
		\node [style=none] (20) at (0, 1.75) {B};
	\end{pgfonlayer}
	\begin{pgfonlayer}{edgelayer}
		\draw [thick, looseness=1.5] [out=90,in=90]  (1.center) to (2.center);
		\draw [thick]  (2.center) to (3.center);
		\draw [thick]  (7.center) to (10.center);
		\draw [thick, looseness=1.5] [out=90,in=90]  (10.center) to (9.center);
		\draw [thick, looseness=0.5] [out=270,in=90] (9.center) to (11.center);
		\draw [thick, looseness=0.5] [out=90,in=270] (4.center) to (5.center);
		\draw [thick, looseness=1.5] [out=90,in=90]  (5.center) to (6.center);
		\draw [thick] [out=270,in=90, looseness=0.5]  (6.center) to (12.center);
		\draw [thick, looseness=0.5] [out=90, in=270] (0.center) to (1.center);
	\end{pgfonlayer}
\end{tikzpicture} \; \; .
\]  The region marked with a $B$ is bounded whereas the other regions of the complement are path-connected to the unbounded part of the plane and so are are unbounded.
\end{example}

Observe that if one applies \cref{E:crossingrelation} to eliminate a crossing, then in each of the two resulting webs a pair of regions have been merged together (namely, the first term merges the regions above and below the crossing while the second term merges the regions on each side of the crossing).  In particular, the application of this relation does not create entirely new bounded regions. The danger lies in the existing bounded regions of the initial web as one of them may become a bubble as crossings are eliminated.  Thus the proof will focus on what happens to bounded regions as relation \cref{E:crossingrelation} and Reidemiester II moves are applied.

Now consider \cref{E:weakbruhatX}.  As discussed in \cref{R:weakBruhatOrder}, $t_{j-1} \leq_{i_{j}}  t_{j}$ if and only if $t_{j}=t_{j-1}.s_{i_{j}}$ and $i_{j}$ appears before $i_{j}+1$ when reading the entries of $t_{j-1}$ left-to-right, top-to-bottom.  Since $t_{j-1}, t_{j} \in \Std ((n+r,n)^{\TT})$, it follows that $i_{j}$ is in the second column of $(n+r,n)^{\TT}$ while $i_{j}+1$ is in the first column (much as in  the proof of \cref{P:OrderingProposition}).  That is, $i_{j}$ is in the second row of $t_{j}^{\TT}$ while $i_{j}+1$ is in the first row.  That is, $i_{j}$ is an $R$-vertex of $\psi (t_{j}^{\TT})$
 while $i_{j}+1$ is an $L$-vertex.  That is, for each $j=1, \dotsc , r$, when $\beta_{i_{j}}$ is concatenated on the bottom of $w_{0}\beta_{i_{1}}\dotsb \beta_{i_{j-1}}$ in the construction of $\varphi(v_{t})$, $i_{j}$ is an $R$-vertex of $w_{0}\beta_{i_{1}}\dotsb \beta_{i_{j-1}}$ while $i_{j}+1$ is an $L$-vertex.

In short, in the drawing of the $(n+r,n)$-web $\varphi(v_{t})$, as one reads the diagram from top-to-bottom each crossing consists of an edge connected to an $R$-vertex passing under an edge connected to a $L$-vertex.  If we were to label the vertices by their type, then locally each crossing would look like
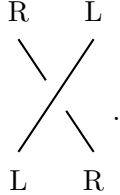
\[
\begin{tikzpicture}[baseline=-2ex,scale=.5, color=\clr ]
	\begin{pgfonlayer}{nodelayer}
		\node [style=none] (0) at (-1, -1.5) {};
		\node [style=none] (1) at (1, -1.5) {};
		\node [style=none] (2) at (-1, 1.5) {};
		\node [style=none] (3) at (1, 1.5) {};
		\node [style=none] (4) at (0.28, -0.4) {};
		\node [style=none] (5) at (-0.28, 0.4) {};
		\node [style=none] (10) at (-1, -2.25) {L};
		\node [style=none] (11) at (1, -2.25) {R};
		\node [style=none] (12) at (-1, 2.25) {R};
		\node [style=none] (13) at (1, 2.25) {L};
	\end{pgfonlayer}
	\begin{pgfonlayer}{edgelayer}
		\draw [thick] (2.center) to (1.center);
		\draw [ultra thick, color=white] (4.center) to (5.center);
		\draw [thick] (3.center) to (0.center);
	\end{pgfonlayer}
\end{tikzpicture} \; .
\]

This observation allows us to verify the diagram drawn for $\varphi(v_{t})$ satisfies the following conditions.
\begin{enumerate}
\item [(B1)]  The diagram for $\varphi(v_{t})$ contains no isolated connected components (i.e., no bubbles).
\item [(B2)]  Every bounded region of $\varphi(v_{t})$ is a bigon, triangle, or square. 
\item [(B3)]  If a bounded region of $\varphi(v_{t})$ is adjacent to second bounded region and is in the location of an S in \cref{E:TriangleSquare} relative to that second bounded region, then the first bounded region is a square.
\end{enumerate}

Condition $(B1)$ holds because $\varphi(v_{t})$ contains no local minimums other than the vertices at the bottom.  If $\varphi(v_{t})$ contains a bounded region with a side not involving one of the edges of $w_{0}$, then the side is formed by an edge connecting $R$-vertices which passes under two successive edges connecting $L$-vertices, or the side is formed by an edge connecting $L$-vertices which passes over two successive edges connecting $R$-vertices.  A bounded region with zero sides of this kind would be an arc whose legs cross --- this is forbidden by our observation about crossings in $\varphi(v_{t})$.  A bounded region with one side of this kind would involve the one side crossing the same arc twice --- this again would violate the observation about crossings in $\varphi(v_{t})$.  Consequently the bounded region must have two or more such sides.  This implies that the bottom part of the bounded region looks like the bottom half of a triangle or square as in \cref{E:TriangleSquare}. Likewise, if the top half of the bounded region also does not involve an edge of $w_{0}$, then the bounded region is a rectangle.  If it does involve an edge in $w_{0}$, then the bounded region is necessarily  a triangle.  Thus $\varphi(v_{t})$ satisfies condition $(B2)$ (bigons do not appear in $\varphi(v_{t})$ but it is convenient to include them as part of the condition).  Finally, since the bounded regions which lie in a location marked by S in \cref{E:TriangleSquare} cannot involve edges of $w_{0}$, the previous discussion also implies $\varphi(v_{t})$ satisfies condition $(B3)$.

\cref{T:positivity} will follow from the following result.

\begin{prop}\label{L:positivity}  If $u$ is an $(n+r,n)$-web which satisfies conditions $(B1)-(B3)$, then
\[
u \in  \bigoplus_{w \in \WW^{(n+r,n)}} \Z_{\geq 0}[q] w.
\]
\end{prop}

\begin{proof} The proof is by induction on the number of crossings in $u$ with the base case being trivial.  For the inductive step, consider the following ordered list of cases.  Observe that they cover all possible scenarios. 

\begin{itemize}
\item [Case 1:] \hspace{\parindent} If $u$ contains no bounded regions, then choose a crossing at random and apply relation \cref{E:crossingrelation}.  The result will be a $\Z_{\geq 0}[q]$-linear combination of webs with no bounded regions and fewer crossings.  Since these webs have no bounded regions, conditions $(B1)-(B3)$ are vacuously satisfied.
\item [Case 2:] \hspace{\parindent} If $u$ contains a bounded region, then choose a bounded region in $u$ for which the neighboring regions marked with $N$ and $W$ are unbounded.  Since $u$ only contains finitely many bounded regions, such a region necessarily exists.
\begin{itemize}
\item [Case 2a:]\hspace{\parindent}   If the selected region is a bigon, then apply the Reidemister II relation given in \cref{L:RII}.  The result will be a $\Z_{\geq 0}[q]$-multiple of a web with fewer crossings.  Since the region in $u$ marked with a W is unbounded, the same will be true of the regions marked with W and E after applying the Reidemister II move.  The region marked with a S in the picture of the bigon in \cref{E:TriangleSquare} is either unbounded or a square, and either remains unbounded or becomes a bigon after the application of the Reidemister II relation. From these observations it is easy to check that the resulting web still satisfies conditions $(B1)-(B3)$.
\item [Case 2b:] \hspace{\parindent} If the selected region is a triangle, then apply relation \cref{E:crossingrelation} to the crossing at the northwest corner of the triangle. Locally, the result will be:
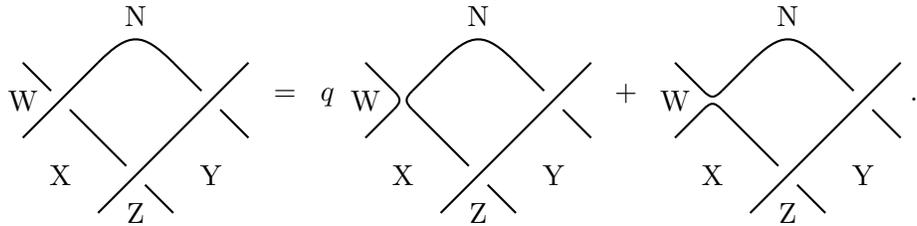
\begin{equation*}
\begin{tikzpicture}[baseline=0ex,scale=.5, color=\clr ]
	\begin{pgfonlayer}{nodelayer}
		\node [style=none] (0) at (1, 1) {};
		\node [style=none] (1) at (-3, -1) {};
		\node [style=none] (2) at (1, 1) {};
		\node [style=none] (3) at (-1, 1) {};
		\node [style=none] (6) at (1, 1) {};
		\node [style=none] (7) at (3, -1) {};
		\node [style=none] (8) at (1.75, 0.25) {};
		\node [style=none] (9) at (3, 1) {};
		\node [style=none] (10) at (-1, -3) {};
		\node [style=none] (11) at (2.25, -0.25) {};
		\node [style=none] (12) at (-1, -1) {};
		\node [style=none] (13) at (-0.25, -1.75) {};
		\node [style=none] (14) at (1, -3) {};
		\node [style=none] (15) at (-1.75, -0.25) {};
		\node [style=none] (16) at (0.25, -2.25) {};
		\node [style=none] (17) at (-3, 1) {};
		\node [style=none] (18) at (-1, -1) {};
		\node [style=none] (19) at (-2.25, 0.25) {};
		\node [style=none] (20) at (0, 3) {};
		\node [style=none] (21) at (0, 2.25) {N};
		\node [style=none] (22) at (-2, 2) {};
		\node [style=none] (23) at (2, 2) {};
		\node [style=none] (24) at (-3, 0) {W};
		\node [style=none] (25) at (3, 0) {};
		\node [style=none] (26) at (2, -2) {};
		\node [style=none] (27) at (0, -3) {Z};
		\node [style=none] (28) at (-2, -2) {X};
		\node [style=none] (29) at (2, -2) {Y};
	\end{pgfonlayer}
	\begin{pgfonlayer}{edgelayer}
		\draw [thick, looseness=1.5] [out=45,in=135] (3.center) to (6.center);
		\draw [thick] (14.center) to (17.center);
		\draw [ultra thick, color=white] (19.center) to (15.center);
		\draw [ultra thick, color=white] (13.center) to (16.center);
		\draw [thick] (1.center) to (3.center);
		\draw [thick] (0.center) to (2.center);
		\draw [thick] (7.center) to (6.center);
		\draw [ultra thick, color=white] (8.center) to (11.center);
		\draw [thick] (10.center) to (9.center);
	\end{pgfonlayer}
\end{tikzpicture}
\; \; = \; \;
q\; \; 
\begin{tikzpicture}[baseline=0ex,scale=.5, color=\clr ]
	\begin{pgfonlayer}{nodelayer}
		\node [style=none] (0) at (1, 1) {};
		\node [style=none] (1) at (-3, -1) {};
		\node [style=none] (2) at (1, 1) {};
		\node [style=none] (3) at (-1, 1) {};
		\node [style=none] (6) at (1, 1) {};
		\node [style=none] (7) at (3, -1) {};
		\node [style=none] (8) at (1.75, 0.25) {};
		\node [style=none] (9) at (3, 1) {};
		\node [style=none] (10) at (-1, -3) {};
		\node [style=none] (11) at (2.25, -0.25) {};
		\node [style=none] (12) at (-1, -1) {};
		\node [style=none] (13) at (-0.25, -1.75) {};
		\node [style=none] (14) at (1, -3) {};
		\node [style=none] (15) at (-1.75, -0.25) {};
		\node [style=none] (16) at (0.25, -2.25) {};
		\node [style=none] (17) at (-3, 1) {};
		\node [style=none] (18) at (-1, -1) {};
		\node [style=none] (19) at (-2.25, 0.25) {};
		\node [style=none] (20) at (0, 3) {};
		\node [style=none] (21) at (0, 2.25) {N};
		\node [style=none] (22) at (-2, 2) {};
		\node [style=none] (23) at (2, 2) {};
		\node [style=none] (24) at (-3, 0) {W};
		\node [style=none] (25) at (3, 0) {};
		\node [style=none] (26) at (2, -2) {};
		\node [style=none] (27) at (0, -3) {Z};
		\node [style=none] (28) at (-2, -2) {X};
		\node [style=none] (29) at (2, -2) {Y};
		\node [style=none] (90) at (-2.25, -0.25) {};
		\node [style=none] (91) at (-1.75, 0.25) {};
	\end{pgfonlayer}
	\begin{pgfonlayer}{edgelayer}
		\draw [thick, looseness=1.5] [out=45,in=135] (3.center) to (6.center);
		\draw [thick] (14.center) to (15.center);
		\draw [ultra thick, color=white] (13.center) to (16.center);
		\draw [thick] (17.center) to (19.center);
		\draw [thick] (1.center) to (90.center);
		\draw [thick] (0.center) to (2.center);
		\draw [thick] (7.center) to (6.center);
		\draw [ultra thick, color=white] (8.center) to (11.center);
		\draw [thick] (10.center) to (9.center);
		\draw [thick] (3.center) to (91.center);
		\draw [thick, looseness=1.5] [out=45,in=315] (90.center) to (19.center);
		\draw [thick, looseness=1.5] [out=225,in=135] (91.center) to (15.center);
	\end{pgfonlayer}
\end{tikzpicture} 
\; \; + \; \;
\begin{tikzpicture}[baseline=0ex,scale=.5, color=\clr ]
	\begin{pgfonlayer}{nodelayer}
		\node [style=none] (0) at (1, 1) {};
		\node [style=none] (1) at (-3, -1) {};
		\node [style=none] (2) at (1, 1) {};
		\node [style=none] (3) at (-1, 1) {};
		\node [style=none] (6) at (1, 1) {};
		\node [style=none] (7) at (3, -1) {};
		\node [style=none] (8) at (1.75, 0.25) {};
		\node [style=none] (9) at (3, 1) {};
		\node [style=none] (10) at (-1, -3) {};
		\node [style=none] (11) at (2.25, -0.25) {};
		\node [style=none] (12) at (-1, -1) {};
		\node [style=none] (13) at (-0.25, -1.75) {};
		\node [style=none] (14) at (1, -3) {};
		\node [style=none] (15) at (-1.75, -0.25) {};
		\node [style=none] (16) at (0.25, -2.25) {};
		\node [style=none] (17) at (-3, 1) {};
		\node [style=none] (18) at (-1, -1) {};
		\node [style=none] (19) at (-2.25, 0.25) {};
		\node [style=none] (20) at (0, 3) {};
		\node [style=none] (21) at (0, 2.25) {N};
		\node [style=none] (22) at (-2, 2) {};
		\node [style=none] (23) at (2, 2) {};
		\node [style=none] (24) at (-3, 0) {W};
		\node [style=none] (25) at (3, 0) {};
		\node [style=none] (26) at (2, -2) {};
		\node [style=none] (27) at (0, -3) {Z};
		\node [style=none] (28) at (-2, -2) {X};
		\node [style=none] (29) at (2, -2) {Y};
		\node [style=none] (90) at (-2.25, -0.25) {};
		\node [style=none] (91) at (-1.75, 0.25) {};
	\end{pgfonlayer}
	\begin{pgfonlayer}{edgelayer}
		\draw [thick, looseness=1.5] [out=45,in=135] (3.center) to (6.center);
		\draw [thick] (14.center) to (15.center);
		\draw [ultra thick, color=white] (13.center) to (16.center);
		\draw [thick] (19.center) to (17.center);
		\draw [thick] (1.center) to (90.center);
		\draw [thick] (3.center) to (91.center);
		\draw [thick] (0.center) to (2.center);
		\draw [thick] (7.center) to (6.center);
		\draw [ultra thick, color=white] (8.center) to (11.center);
		\draw [thick] (10.center) to (9.center);
		\draw [thick, looseness=1.5] [in=315,out=225] (91.center) to (19.center);
		\draw [thick, looseness=1.5] [out=45,in=135] (90.center) to (15.center);
	\end{pgfonlayer}
\end{tikzpicture} \;.
\end{equation*}   \hspace{\parindent} We claim both webs on right hand side of the equality still satisfy conditions $(B1)-(B3)$.  The conditions are verified by checking each bounded region one-by-one.  Since the only regions of $u$ affected by resolving the crossing are the ones which appear in the picture, it suffices to study these in detail.  First, note that no bubbles were created so condition $(B1)$ is satisfied by both. Furthermore, any unbounded regions in $u$ (e.g., those marked with N and W) remain unbounded in the right hand side of the equality and so can be ignored.

\hspace{\parindent} \hspace{\parindent} Now, in the first diagram on the right side of the equality the triangle was turned into a bigon, the region marked with an X was made unbounded, and the regions marked with a Y and Z remain unchanged (in particular, if one or both were bounded in $u$ then they were a rectangle and remain a rectangle in the new diagram), so both conditions $(B2)$ and $(B3)$ are still satisfied.  In the second diagram the triangle became an unbounded region and if X marked a bounded region of $u$, then it was rectangle and becomes a triangle in the new diagram.   Thus both of the diagrams on the right hand side have fewer crossings than $u$ and still satisfy conditions $(B1)-(B3)$. 
\item [Case 2c:] \hspace{\parindent} If the selected region is a square, then apply relation \cref{E:crossingrelation} to the crossing at the northern corner of the square. Locally, the result will be:
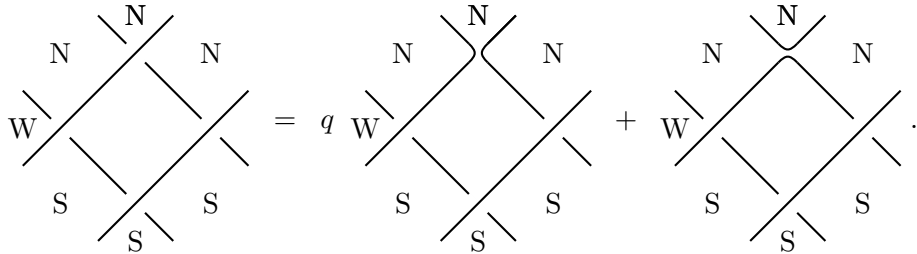
\begin{equation*}
\begin{tikzpicture}[baseline=0ex,scale=.5, color=\clr ]
	\begin{pgfonlayer}{nodelayer}
		\node [style=none] (0) at (1, 1) {};
		\node [style=none] (1) at (-3, -1) {};
		\node [style=none] (2) at (-1, 3) {};
		\node [style=none] (3) at (1, 3) {};
		\node [style=none] (4) at (-0.25, 2.25) {};
		\node [style=none] (5) at (0.25, 1.75) {};
		\node [style=none] (6) at (1, 1) {};
		\node [style=none] (7) at (3, -1) {};
		\node [style=none] (8) at (1.75, 0.25) {};
		\node [style=none] (9) at (3, 1) {};
		\node [style=none] (10) at (-1, -3) {};
		\node [style=none] (11) at (2.25, -0.25) {};
		\node [style=none] (12) at (-1, -1) {};
		\node [style=none] (13) at (-0.25, -1.75) {};
		\node [style=none] (14) at (1, -3) {};
		\node [style=none] (15) at (-1.75, -0.25) {};
		\node [style=none] (16) at (0.25, -2.25) {};
		\node [style=none] (17) at (-3, 1) {};
		\node [style=none] (18) at (-1, -1) {};
		\node [style=none] (19) at (-2.25, 0.25) {};
		\node [style=none] (20) at (0, 3) {N};
		\node [style=none] (21) at (0, 3) {N};
		\node [style=none] (22) at (-2, 2) {N};
		\node [style=none] (23) at (2, 2) {N};
		\node [style=none] (24) at (-3, 0) {W};
		\node [style=none] (26) at (2, -2) {S};
		\node [style=none] (27) at (0, -3) {S};
		\node [style=none] (28) at (-2, -2) {S};
	\end{pgfonlayer}
	\begin{pgfonlayer}{edgelayer}
		\draw [thick] (7.center) to (2.center);
		\draw [ultra thick, color=white] (4.center) to (5.center);
		\draw [ultra thick, color=white] (8.center) to (11.center);
		\draw [thick] (14.center) to (17.center);
		\draw [ultra thick, color=white] (19.center) to (15.center);
		\draw [ultra thick, color=white] (13.center) to (16.center);
		\draw [thick] (1.center) to (3.center);
		\draw [thick] (10.center) to (9.center);
	\end{pgfonlayer}
\end{tikzpicture}\; \; = \; \; q \; \;
\begin{tikzpicture}[baseline=0ex,scale=.5, color=\clr ]
	\begin{pgfonlayer}{nodelayer}
		\node [style=none] (0) at (1, 1) {};
		\node [style=none] (1) at (-3, -1) {};
		\node [style=none] (2) at (-1, 3) {};
		\node [style=none] (3) at (1, 3) {};
		\node [style=none] (4) at (-0.25, 2.25) {};
		\node [style=none] (5) at (0.25, 1.75) {};
		\node [style=none] (6) at (1, 1) {};
		\node [style=none] (7) at (3, -1) {};
		\node [style=none] (8) at (1.75, 0.25) {};
		\node [style=none] (9) at (3, 1) {};
		\node [style=none] (10) at (-1, -3) {};
		\node [style=none] (11) at (2.25, -0.25) {};
		\node [style=none] (12) at (-1, -1) {};
		\node [style=none] (13) at (-0.25, -1.75) {};
		\node [style=none] (14) at (1, -3) {};
		\node [style=none] (15) at (-1.75, -0.25) {};
		\node [style=none] (16) at (0.25, -2.25) {};
		\node [style=none] (17) at (-3, 1) {};
		\node [style=none] (18) at (-1, -1) {};
		\node [style=none] (19) at (-2.25, 0.25) {};
		\node [style=none] (20) at (0, 3) {N};
		\node [style=none] (21) at (0, 3) {N};
		\node [style=none] (22) at (-2, 2) {N};
		\node [style=none] (23) at (2, 2) {N};
		\node [style=none] (24) at (-3, 0) {W};
		\node [style=none] (26) at (2, -2) {S};
		\node [style=none] (27) at (0, -3) {S};
		\node [style=none] (28) at (-2, -2) {S};
		\node [style=none] (94) at (0.25, 2.25) {};
		\node [style=none] (95) at (-0.25, 1.75) {};
	\end{pgfonlayer}
	\begin{pgfonlayer}{edgelayer}
		\draw [thick] (4.center) to (2.center);
		\draw [thick] (5.center) to (7.center);
		\draw [ultra thick, color=white] (8.center) to (11.center);
		\draw [thick] (14.center) to (17.center);
		\draw [ultra thick, color=white] (19.center) to (15.center);
		\draw [ultra thick, color=white] (13.center) to (16.center);
		\draw [thick] (10.center) to (9.center);
		\draw [thick] (1.center) to (95.center);
		\draw [thick] (94.center) to (3.center);
		\draw [thick, looseness=1.5] [out=315,in=45] (4.center) to (95.center);
		\draw [thick, looseness=1.5] [out=135,in=225] (5.center) to (94.center);
		\draw [thick] (94.center) to (3.center);
	\end{pgfonlayer}
\end{tikzpicture}\; \; + \; \;
\begin{tikzpicture}[baseline=0ex,scale=.5, color=\clr ]
	\begin{pgfonlayer}{nodelayer}
		\node [style=none] (0) at (1, 1) {};
		\node [style=none] (1) at (-3, -1) {};
		\node [style=none] (2) at (-1, 3) {};
		\node [style=none] (3) at (1, 3) {};
		\node [style=none] (4) at (-0.25, 2.25) {};
		\node [style=none] (5) at (0.25, 1.75) {};
		\node [style=none] (6) at (1, 1) {};
		\node [style=none] (7) at (3, -1) {};
		\node [style=none] (8) at (1.75, 0.25) {};
		\node [style=none] (9) at (3, 1) {};
		\node [style=none] (10) at (-1, -3) {};
		\node [style=none] (11) at (2.25, -0.25) {};
		\node [style=none] (12) at (-1, -1) {};
		\node [style=none] (13) at (-0.25, -1.75) {};
		\node [style=none] (14) at (1, -3) {};
		\node [style=none] (15) at (-1.75, -0.25) {};
		\node [style=none] (16) at (0.25, -2.25) {};
		\node [style=none] (17) at (-3, 1) {};
		\node [style=none] (18) at (-1, -1) {};
		\node [style=none] (19) at (-2.25, 0.25) {};
		\node [style=none] (20) at (0, 3) {N};
		\node [style=none] (21) at (0, 3) {N};
		\node [style=none] (22) at (-2, 2) {N};
		\node [style=none] (23) at (2, 2) {N};
		\node [style=none] (24) at (-3, 0) {W};
		\node [style=none] (26) at (2, -2) {S};
		\node [style=none] (27) at (0, -3) {S};
		\node [style=none] (28) at (-2, -2) {S};
		\node [style=none] (94) at (0.25, 2.25) {};
		\node [style=none] (95) at (-0.25, 1.75) {};
	\end{pgfonlayer}
	\begin{pgfonlayer}{edgelayer}
		\draw [thick] (2.center) to (4.center);
		\draw [thick] (7.center) to (5.center);
		\draw [thick] (94.center) to (3.center);
		\draw [ultra thick, color=white] (8.center) to (11.center);
		\draw [thick] (14.center) to (17.center);
		\draw [ultra thick, color=white] (19.center) to (15.center);
		\draw [ultra thick, color=white] (13.center) to (16.center);
		\draw [thick] (10.center) to (9.center);
		\draw [thick, looseness=1.5] [out=315,in=225] (4.center) to (94.center);
		\draw [thick, looseness=1.5] [out=135,in=45] (5.center) to (95.center);
		\draw [thick] (1.center) to (95.center);
	\end{pgfonlayer}
\end{tikzpicture}\; .
\end{equation*} \hspace{\parindent} We again claim the diagrams on the right side of the equality still satisfy conditions $(B1)-(B3)$. The analysis is an easier version of what was done for the triangle. Obviously, no bubbles are created and so $(B1)$ is still satisfied by both diagrams. For the square the only bounded region of $u$ affected by resolving the crossing is the square itself.  In the first diagram on the right side of the equality the square becomes unbounded and in the second diagram the square becomes a triangle.  If any of the regions marked with an S were bounded in $u$, then they were a rectangle and still are rectangles in the new diagrams.  Both diagrams  still satisfy conditions $(B2)$ and $(B3)$.    Thus both of the diagrams on the right hand side have fewer crossings than $u$ and still satisfy conditions $(B1)-(B3)$. 
\end{itemize}
\end{itemize}
 Whichever case is applied, the result replaces $u$ with a $ \Z_{\geq 0}[q]$-linear combination of webs which still satisfy conditions $(B1)-(B3)$ and have strictly fewer crossings.  The claim follows by induction.
\end{proof}

\section{Examples}\label{S:Examples}

\subsection{Example I}\label{SS:ExampleI}

In this example we rewrite $w_{0}T_{4}T_{2}T_{3} = w_{0}\beta_{4}\beta_{2}\beta_{3}$ for $\lambda = (3,3)$ as a linear combination of non-crossing webs using relations \cref{E:crossingrelation,E:bubblerelation}:
\begin{align*}\label{E:Example1}
\begin{tikzpicture}[baseline=0.5cm,scale=.5, color=\clr ]
	\begin{pgfonlayer}{nodelayer}
		\node [style=none] (0) at (1, 0) {};
		\node [style=none] (1) at (2, 3) {};
		\node [style=none] (2) at (4, 3) {};
		\node [style=none] (3) at (4, 0) {};
		\node [style=none] (4) at (-2, 0) {};
		\node [style=none] (5) at (-1, 2) {};
		\node [style=none] (6) at (1, 2) {};
		\node [style=none] (7) at (-4, 0) {};
		\node [style=none] (8) at (-5, 2) {};
		\node [style=none] (9) at (-2, 2) {};
		\node [style=none] (10) at (-4, 2) {};
		\node [style=none] (11) at (2, 0) {};
		\node [style=none] (12) at (3, 0) {};
		\node [style=none] (13) at (-1.75, 1.62) {};
		\node [style=none] (14) at (-1, 1.28) {};
		\node [style=none] (15) at (0.75, 0.77) {};
		\node [style=none] (16) at (1.5, 0.5) {};
		\node [style=none] (17) at (1.25, 1.6) {};
		\node [style=none] (18) at (1.75, 1.2) {};
	\end{pgfonlayer}
	\begin{pgfonlayer}{edgelayer}
		\draw [thick, looseness=1.5] [out=90,in=90]  (1.center) to (2.center);
		\draw [thick]  (2.center) to (3.center);
		\draw [thick]  (7.center) to (10.center);
		\draw [thick, looseness=1.5] [out=90,in=90]  (10.center) to (9.center);
		\draw [thick, looseness=0.5] [out=270,in=90] (9.center) to (11.center);
		\draw [ultra thick, color=white]  (13.center) to (14.center);
		\draw [ultra thick, color=white]  (15.center) to (16.center);
		\draw [thick, looseness=0.5] [out=90,in=270] (4.center) to (5.center);
		\draw [thick, looseness=1.5] [out=90,in=90]  (5.center) to (6.center);
		\draw [thick] [out=270,in=90, looseness=0.5]  (6.center) to (12.center);
		\draw [ultra thick, color=white]  (17.center) to (18.center);
		\draw [thick, looseness=0.5] [out=90, in=270] (0.center) to (1.center);
	\end{pgfonlayer}
\end{tikzpicture}
\; \; &= \; \;
q  \begin{tikzpicture}[baseline=0.5cm,scale=.5, color=\clr ]
	\begin{pgfonlayer}{nodelayer}
		\node [style=none] (0) at (0, 0) {};
		\node [style=none] (1) at (2, 3) {};
		\node [style=none] (2) at (4, 3) {};
		\node [style=none] (3) at (4, 0) {};
		\node [style=none] (4) at (-2, 0) {};
		\node [style=none] (5) at (-1, 2) {};
		\node [style=none] (6) at (1, 2) {};
		\node [style=none] (7) at (-4, 0) {};
		\node [style=none] (8) at (-5, 2) {};
		\node [style=none] (9) at (-2, 2) {};
		\node [style=none] (10) at (-4, 2) {};
		\node [style=none] (11) at (1, 0) {};
		\node [style=none] (12) at (2, 0) {};
		\node [style=none] (13) at (-1.75, 1.62) {};
		\node [style=none] (14) at (-1, 1.23) {};
		\node [style=none] (15) at (-0.15, 0.85) {};
		\node [style=none] (16) at (0.6, 0.5) {};
		\node [style=none] (17) at (1.25, 1.6) {};
		\node [style=none] (18) at (1.75, 1.2) {};
	\end{pgfonlayer}
	\begin{pgfonlayer}{edgelayer}
		\draw [thick, looseness=1.5] [out=90,in=90]  (12.center) to (3.center);
		\draw [thick]  (7.center) to (10.center);
		\draw [thick, looseness=1.5] [out=90,in=90]  (10.center) to (9.center);
		\draw [thick, looseness=0.5] [out=270,in=90] (9.center) to (11.center);
		\draw [ultra thick, color=white]  (13.center) to (14.center);
		\draw [ultra thick, color=white]  (15.center) to (16.center);
		\draw [thick, looseness=0.5] [out=90,in=270] (4.center) to (5.center);
		\draw [thick, looseness=1.5] [out=90,in=90]  (5.center) to (6.center);
		\draw [thick, looseness=0.5] [out=270,in=90]  (6.center) to (0.center);
		\draw [ultra thick, color=white]  (17.center) to (18.center);
	\end{pgfonlayer}
\end{tikzpicture}
\; \; + \; \;
 \begin{tikzpicture}[baseline=0.5cm,scale=.5, color=\clr ]
	\begin{pgfonlayer}{nodelayer}
		\node [style=none] (0) at (0, 0) {};
		\node [style=none] (1) at (2, 2) {};
		\node [style=none] (2) at (4, 2) {};
		\node [style=none] (3) at (4, 0) {};
		\node [style=none] (4) at (-2, 0) {};
		\node [style=none] (5) at (-1, 2) {};
		\node [style=none] (6) at (1, 2) {};
		\node [style=none] (7) at (-4, 0) {};
		\node [style=none] (9) at (-2, 2) {};
		\node [style=none] (10) at (-4, 2) {};
		\node [style=none] (11) at (1, 0) {};
		\node [style=none] (12) at (2, 0) {};
		\node [style=none] (13) at (-1.75, 1.62) {};
		\node [style=none] (14) at (-1, 1.23) {};
		\node [style=none] (15) at (-0.15, 0.85) {};
		\node [style=none] (16) at (0.6, 0.5) {};
		\node [style=none] (17) at (1.25, 1.6) {};
		\node [style=none] (18) at (1.75, 1.2) {};
	\end{pgfonlayer}
	\begin{pgfonlayer}{edgelayer}
		\draw [thick, looseness=1.5] [out=90,in=90]  (6.center) to (1.center);
		\draw [thick]  (2.center) to (3.center);
		\draw [thick]  (1.center) to (12.center);
		\draw [thick]  (7.center) to (10.center);
		\draw [thick, looseness=1.5] [out=90,in=90]  (10.center) to (9.center);
		\draw [thick, looseness=0.5] [out=270,in=90] (9.center) to (11.center);
		\draw [ultra thick, color=white]  (13.center) to (14.center);
		\draw [ultra thick, color=white]  (15.center) to (16.center);
		\draw [thick, looseness=0.5] [out=90,in=270] (4.center) to (5.center);
		\draw [thick, looseness=1.5] [out=90,in=90]  (5.center) to (2.center);
		\draw [thick, looseness=0.5] [out=270,in=90]  (6.center) to (0.center);
		\draw [ultra thick, color=white]  (17.center) to (18.center);
	\end{pgfonlayer}
\end{tikzpicture} \\
\; \; &= \; \;
q^{2}  \begin{tikzpicture}[baseline=0.5cm,scale=.5, color=\clr ]
	\begin{pgfonlayer}{nodelayer}
		\node [style=none] (0) at (-1, 0) {};
		\node [style=none] (1) at (1, 2) {};
		\node [style=none] (2) at (-1, 2) {};
		\node [style=none] (3) at (1, 0) {};
		\node [style=none] (4) at (-2, 0) {};
		\node [style=none] (5) at (-4, 0) {};
		\node [style=none] (6) at (2, 0) {};
		\node [style=none] (7) at (4, 0) {};
		\node [style=none] (8) at (-0.25, 1.1) {};
		\node [style=none] (9) at (0.25, 0.9) {};
	\end{pgfonlayer}
	\begin{pgfonlayer}{edgelayer}
		\draw [thick] [out=270, in=90] (2.center) to (3.center);
		\draw [ultra thick, color=white] (8.center) to (9.center);
		\draw [thick] [out=90, in=270] (0.center) to (1.center);
		\draw [thick, looseness=1.5] [out=90,in=90] (1.center) to (2.center);
		\draw [thick, looseness=1.5] [out=90,in=90] (4.center) to (5.center);
		\draw [thick, looseness=1.5] [out=90,in=90] (7.center) to (6.center);
	\end{pgfonlayer}
\end{tikzpicture}
\; \; + \; \;
q \;  \begin{tikzpicture}[baseline=0.5cm,scale=.5, color=\clr ]
	\begin{pgfonlayer}{nodelayer}
		\node [style=none] (0) at (0, 0) {};
		\node [style=none] (1) at (2, 3) {};
		\node [style=none] (2) at (4, 3) {};
		\node [style=none] (3) at (4, 0) {};
		\node [style=none] (4) at (-2, 0) {};
		\node [style=none] (5) at (-1, 2) {};
		\node [style=none] (6) at (1, 2) {};
		\node [style=none] (7) at (-3, 0) {};
		\node [style=none] (9) at (-2, 2) {};
		\node [style=none] (10) at (-3, 2) {};
		\node [style=none] (11) at (1, 0) {};
		\node [style=none] (12) at (2, 0) {};
		\node [style=none] (13) at (-1.75, 1.62) {};
		\node [style=none] (14) at (-1, 1.23) {};
		\node [style=none] (15) at (0.25, 1.2) {};
		\node [style=none] (16) at (0.75, 0.76) {};
		\node [style=none] (17) at (1.25, 1.6) {};
		\node [style=none] (18) at (1.75, 1.2) {};
	\end{pgfonlayer}
	\begin{pgfonlayer}{edgelayer}
		\draw [thick, looseness=1.5] [out=90,in=90]  (12.center) to (3.center);
		\draw [thick]  (7.center) to (10.center);
		\draw [thick, looseness=1.25] [out=90,in=90]  (10.center) to (6.center);
		\draw [thick, looseness=1.5] [out=90,in=90] (4.center) to (11.center);
		\draw [ultra thick, color=white]  (15.center) to (16.center);
		\draw [thick] [out=270,in=90]  (6.center) to (0.center);
		\draw [ultra thick, color=white]  (17.center) to (18.center);
	\end{pgfonlayer}
\end{tikzpicture} \\
& \hspace{.5in} +
q \;  \begin{tikzpicture}[baseline=0.5cm,xscale=-1,scale=.5, color=\clr ]
	\begin{pgfonlayer}{nodelayer}
		\node [style=none] (0) at (0, 0) {};
		\node [style=none] (1) at (2, 3) {};
		\node [style=none] (2) at (4, 3) {};
		\node [style=none] (3) at (4, 0) {};
		\node [style=none] (4) at (-2, 0) {};
		\node [style=none] (5) at (-1, 2) {};
		\node [style=none] (6) at (1, 2) {};
		\node [style=none] (7) at (-3, 0) {};
		\node [style=none] (9) at (-2, 2) {};
		\node [style=none] (10) at (-3, 2) {};
		\node [style=none] (11) at (1, 0) {};
		\node [style=none] (12) at (2, 0) {};
		\node [style=none] (13) at (-1.75, 1.62) {};
		\node [style=none] (14) at (-1, 1.23) {};
		\node [style=none] (15) at (0.25, 0.72) {};
		\node [style=none] (16) at (0.75, 1.28) {};
		\node [style=none] (17) at (1.25, 1.6) {};
		\node [style=none] (18) at (1.75, 1.2) {};
	\end{pgfonlayer}
	\begin{pgfonlayer}{edgelayer}
		\draw [thick, looseness=1.5] [out=90,in=90]  (12.center) to (3.center);
		\draw [thick]  (7.center) to (10.center);
		\draw [thick, looseness=1.25] [out=90,in=90]  (10.center) to (6.center);
		\draw [thick] [out=270,in=90]  (6.center) to (0.center);
		\draw [ultra thick, color=white]  (15.center) to (16.center);
		\draw [thick, looseness=1.5] [out=90,in=90] (4.center) to (11.center);
		\draw [ultra thick, color=white]  (17.center) to (18.center);
	\end{pgfonlayer}
\end{tikzpicture}
\;\; + \; \;
\begin{tikzpicture}[baseline=0.5cm,scale=.5, color=\clr ]
	\begin{pgfonlayer}{nodelayer}
		\node [style=none] (0) at (-0.5, 0) {};
		\node [style=none] (3) at (0.5, 0) {};
		\node [style=none] (4) at (-2.5, 0) {};
		\node [style=none] (5) at (-4, 0) {};
		\node [style=none] (6) at (2.5, 0) {};
		\node [style=none] (7) at (4, 0) {};
		\node [style=none] (8) at (-0.25, 1.15) {};
		\node [style=none] (9) at (0.25, 0.8) {};
	\end{pgfonlayer}
	\begin{pgfonlayer}{edgelayer}
		\draw [thick, looseness=1.5] [out=90,in=90] (5.center) to (7.center);
		\draw [thick, looseness=1.5] [out=90,in=90] (4.center) to (3.center);
		\draw [ultra thick, color=white] (8.center) to (9.center);
		\draw [thick, looseness=1.5] [out=90,in=90] (0.center) to (6.center);
	\end{pgfonlayer}
\end{tikzpicture} \\
&=
(q^{3}-q^{2}[2]_{q} + 2q ) \; \begin{tikzpicture}[baseline=0cm,scale=.5, color=\clr ]
	\begin{pgfonlayer}{nodelayer}
		\node [style=none] (0) at (-1.5, 0) {};
		\node [style=none] (3) at (0.5, 0) {};
		\node [style=none] (4) at (-2.5, 0) {};
		\node [style=none] (5) at (-4.5, 0) {};
		\node [style=none] (6) at (1.5, 0) {};
		\node [style=none] (7) at (3.5, 0) {};
	\end{pgfonlayer}
	\begin{pgfonlayer}{edgelayer}
		\draw [thick, looseness=1.75] [out=90,in=90] (5.center) to (4.center);
		\draw [thick, looseness=1.75] [out=90,in=90] (0.center) to (3.center);
		\draw [thick, looseness=1.75] [out=90,in=90] (6.center) to (7.center);
	\end{pgfonlayer}
\end{tikzpicture} \\
& \hspace{0.5in} + \; \;
q^{2}\; \begin{tikzpicture}[baseline=0.25cm,scale=.5, color=\clr ]
	\begin{pgfonlayer}{nodelayer}
		\node [style=none] (0) at (-1.5, 0) {};
		\node [style=none] (3) at (0, 0) {};
		\node [style=none] (4) at (-2.5, 0) {};
		\node [style=none] (5) at (-4, 0) {};
		\node [style=none] (6) at (1, 0) {};
		\node [style=none] (7) at (3, 0) {};
	\end{pgfonlayer}
	\begin{pgfonlayer}{edgelayer}
		\draw [thick, looseness=1.25] [out=90,in=90] (5.center) to (3.center);
		\draw [thick, looseness=1.5] [out=90,in=90] (0.center) to (4.center);
		\draw [thick, looseness=1.5] [out=90,in=90] (6.center) to (7.center);
	\end{pgfonlayer}
\end{tikzpicture}
\; \;  + \; \; 
q^{2}\; \begin{tikzpicture}[baseline=0.25cm,xscale=-1,scale=.5, color=\clr ]
	\begin{pgfonlayer}{nodelayer}
		\node [style=none] (0) at (-1.5, 0) {};
		\node [style=none] (3) at (0, 0) {};
		\node [style=none] (4) at (-2.5, 0) {};
		\node [style=none] (5) at (-4, 0) {};
		\node [style=none] (6) at (1, 0) {};
		\node [style=none] (7) at (3, 0) {};
	\end{pgfonlayer}
	\begin{pgfonlayer}{edgelayer}
		\draw [thick, looseness=1.25] [out=90,in=90] (5.center) to (3.center);
		\draw [thick, looseness=1.5] [out=90,in=90] (0.center) to (4.center);
		\draw [thick, looseness=1.5] [out=90,in=90] (6.center) to (7.center);
	\end{pgfonlayer}
\end{tikzpicture} \\
& \hspace{0.5in} + 
q\; \begin{tikzpicture}[baseline=0.25cm,xscale=-1,scale=.5, color=\clr ]
	\begin{pgfonlayer}{nodelayer}
		\node [style=none] (0) at (-1.5, 0) {};
		\node [style=none] (3) at (0.5, 0) {};
		\node [style=none] (4) at (-2.5, 0) {};
		\node [style=none] (5) at (-4, 0) {};
		\node [style=none] (6) at (1.5, 0) {};
		\node [style=none] (7) at (3, 0) {};
	\end{pgfonlayer}
	\begin{pgfonlayer}{edgelayer}
		\draw [thick, looseness=1] [out=90,in=90] (5.center) to (7.center);
		\draw [thick, looseness=2] [out=90,in=90] (0.center) to (4.center);
		\draw [thick, looseness=2] [out=90,in=90] (6.center) to (3.center);
	\end{pgfonlayer}
\end{tikzpicture}
\; \; + \; \;
\begin{tikzpicture}[baseline=0.25cm,xscale=-1,scale=.5, color=\clr ]
	\begin{pgfonlayer}{nodelayer}
		\node [style=none] (0) at (-1.5, 0) {};
		\node [style=none] (3) at (0.5, 0) {};
		\node [style=none] (4) at (-2.5, 0) {};
		\node [style=none] (5) at (-4, 0) {};
		\node [style=none] (6) at (1.5, 0) {};
		\node [style=none] (7) at (3, 0) {};
	\end{pgfonlayer}
	\begin{pgfonlayer}{edgelayer}
		\draw [thick, looseness=1] [out=90,in=90] (5.center) to (7.center);
		\draw [thick, looseness=1.25] [out=90,in=90] (6.center) to (4.center);
		\draw [thick, looseness=1.5] [out=90,in=90] (0.center) to (3.center);
	\end{pgfonlayer}
\end{tikzpicture} \; .
\end{align*} Observe that 
\[
q^{3} - q^{2}[2]_{q}+2q = q^{3} - q^{2}(q+q^{-1})+2q = q.
\]  Thus all the coefficients given above lie in $\Z_{\geq 0}[q]$ as promised by \cref{T:positivity}.  Also note that the bubble could have been avoided by the use of one of the Reidemeister II moves given in \cref{L:RII}.

\subsection{Example II}\label{SS:ExampleII}
Set $\lambda = (3,3)$.  $t_{0} = t^{\lambda^{\TT}}$, $t_{1}=t_{0}.s_{2}$, $t_{2}=t_{0}.s_{4}$, $t_{3}=t_{0}.s_{2}s_{4}$, and $t_{4}=t_{0}.s_{4}s_{2}s_{3}$.  For $i=0,\dotsc , 4$, let $w_{i}=\phi (t_{i})$.  With respect to the ordered bases $\left\{v_{t_{i}} \right\}_{i=0}^{4}$ and $\left\{w_{i} \right\}_{i=0}^{4}$, the matrix for the isomorphism $\varphi: S^{(3,3)} \to W^{(3,3)}$ is given by

\[ \left(
\begin{matrix}
1 & q & q & q^{2} & q \\
0 & 1 & 0 & q & q^{2} \\
0 & 0 & 1 & q & q^{2} \\
0 & 0 & 0 & 1 & q \\
0 & 0 & 0 & 0 & 1 \\
\end{matrix}  \right) \; .
\]  The calculation in \cref{SS:ExampleI} verifies the last column of this matrix.

\subsection{Example III}\label{SS:ExampleIII}
Set $\lambda = (4,2)$.  $t_{0} = t^{\lambda^{\TT}}$, $t_{1}=t_{0}.s_{2}$, $t_{2}=t_{0}.s_{4}$, $t_{3}=t_{0}.s_{2}s_{4}$, $t_{4}=t_{0}.s_{4}s_{5}$,  $t_{5}=t_{0}.s_{4}s_{5}s_{2}$,  $t_{6}=t_{0}.s_{4}s_{2}s_{3}$,  $t_{7}=t_{0}.s_{4}s_{2}s_{3}s_{5}$, and  $t_{8}=t_{0}.s_{4}s_{2}s_{3}s_{5}s_{4}$.  For $i=0,\dotsc, 8 $, let $w_{i}=\phi (t_{i})$.  With respect to the ordered bases $\left\{v_{t_{i}} \right\}_{i=0}^{8}$ and $\left\{w_{i} \right\}_{i=0}^{8}$, the matrix for the isomorphism $\varphi: S^{(4,2)} \to W^{(4,2)}$ is given by

\[ \left(
\begin{matrix}
1 & q & q & q^{2} & q^{2} & q^{3} & q & q^{2} & q^{3} \\
0 & 1 & 0 & q & 0 & q^{2} & q^{2} & q^{3} & q^{4} \\
0 & 0 & 1 & q & q & q^{2} & q^{2} & q^{3} & q^{2} \\
0 & 0 & 0 & 1 & 0 & q & q & q^{2} & q^{3}+q \\
0 & 0 & 0 & 0 & 1 & q & 0 & q^{2} & q^{3} \\
0 & 0 & 0 & 0 & 0 & 1 & 0 & q & q^{2} \\
0 & 0 & 0 & 0 & 0 & 0 & 1 & q & q^{2} \\
0 & 0 & 0 & 0 & 0 & 0 & 0 & 1 & q \\
0 & 0 & 0 & 0 & 0 & 0 & 0 & 0 & 1 \\
\end{matrix}  \right) \; .
\]

\bibliographystyle{eprintamsplain}
\bibliography{Biblio}

\end{document}